\documentclass[a4paper, reqno, 11pt, notitlepage]{amsart}
\usepackage{fullpage}
\usepackage[utf8]{inputenc}

\pdfoutput=1

\usepackage{latexsym}
\usepackage{amsmath}
\usepackage{amssymb}
\usepackage{amsthm}
\usepackage{amscd}
\usepackage{mathrsfs}
\usepackage[all]{xy}
\usepackage{graphicx}
\usepackage{comment}
\usepackage{arydshln}
\usepackage{mathtools}
\usepackage{dsfont}
\usepackage{adjustbox}
\usepackage{multicol}
\usepackage{quiver}
\usepackage{changepage} 
\usepackage[activate={true,nocompatibility},final,tracking=true,kerning=true,spacing=true]{microtype}
\microtypecontext{spacing=nonfrench}
\usepackage{eucal}
\usepackage{yfonts}

\numberwithin{equation}{subsection}

\makeatletter

\renewcommand{\tocsection}[3]{%
  \indentlabel{\@ifnotempty{#2}{\bfseries\ignorespaces#1 #2\quad}}\bfseries#3}

\renewcommand{\tocsubsection}[3]{%
  \indentlabel{\@ifnotempty{#2}{\ignorespaces#1 #2\quad}}#3}

\newcommand\@dotsep{4.5}
\def\@tocline#1#2#3#4#5#6#7{\relax
  \ifnum #1>\c@tocdepth 
  \else
    \par \addpenalty\@secpenalty\addvspace{#2}%
    \begingroup \hyphenpenalty\@M
    \@ifempty{#4}{%
      \@tempdima\csname r@tocindent\number#1\endcsname\relax
    }{%
      \@tempdima#4\relax
    }%
    \parindent\z@ \leftskip#3\relax \advance\leftskip\@tempdima\relax
    \rightskip\@pnumwidth plus1em \parfillskip-\@pnumwidth
    #5\leavevmode\hskip-\@tempdima{#6}\nobreak
    \leaders\hbox{$\m@th\mkern \@dotsep mu\hbox{.}\mkern \@dotsep mu$}\hfill
    \nobreak
    \hbox to\@pnumwidth{\@tocpagenum{\ifnum#1=1\bfseries\fi#7}}\par
    \nobreak
    \endgroup
  \fi}
\AtBeginDocument{%
\expandafter\renewcommand\csname r@tocindent0\endcsname{0pt}
}
\def\l@subsection{\@tocline{2}{0pt}{2.5pc}{5pc}{}}
\makeatother

    \makeatletter
    \def\paragraph{\@startsection{paragraph}{4}%
    \z@\z@{-\fontdimen2\font}%
    {\normalfont\bfseries}}
    \makeatother

\makeatletter
\def\subsubsection{\@startsection{subsubsection}{3}%
  \z@{.5\linespacing\@plus.7\linespacing}{-.5em}%
  {\normalfont\bfseries}}
\makeatother

\usepackage[T1]{fontenc}

\usepackage{tikz-cd}
\usepackage{nicefrac}
\usepackage{authblk}
\usepackage{textcomp}
\usepackage{graphicx}
\usepackage{framed} \usepackage{color}

\usepackage[pagebackref,hypertexnames=false]{hyperref} 
    \definecolor{darkblue}{rgb}{0,0,.85} 
    \definecolor{darkred}{rgb}{0.84,0,0}
    \hypersetup{colorlinks = true,
            linkcolor = darkblue,
            urlcolor  = darkblue,
            citecolor = darkred,
            anchorcolor = darkblue,
            bookmarksdepth=3}

\usepackage{accents}

\usepackage[shortlabels]{enumitem}

\usepackage[foot]{amsaddr}

\usepackage{bm}

\usepackage{stmaryrd}

\usepackage{graphics}

\usepackage{mathtools}

\newtheorem{thm}{Theorem}[section]
\newtheorem{prop}[thm]{Proposition}

\newtheorem{thmi}{Theorem}

\newtheorem{propi}{Proposition}

\newtheorem{cori}{Corollary}

\newtheorem{lem}[thm]{Lemma}

\newtheorem{cor}[thm]{Corollary}

\theoremstyle{definition}
\newtheorem{example}[thm]{Example}

\newtheorem{nota}[thm]{Notation}

\newtheorem{construction}[thm]{Construction}

\makeatletter
\newtheoremstyle{examplestyle}
  {1em}
  {1em}
  {\addtolength{\@totalleftmargin}{1.0em}
   \addtolength{\linewidth}{-1.0em}
   \parshape 1 1.0em \linewidth}
  {}
  {\bfseries}
  {.}
  {.5em}
  {}

\theoremstyle{examplestyle}

\newtheorem{rem}[thm]{Remark}
\newtheorem{remi}{Remark}

\newtheorem{defn}[thm]{Definition}

\newtheorem*{defni*}{Definition}

\newcommand{\pddiv}{\mr{sPDd}}
\DeclareMathOperator{\Spec}{Spec}

\DeclareMathOperator{\Gal}{Gal}
\DeclareMathOperator{\id}{id}

\DeclareMathOperator{\Hom}{Hom}

\DeclareMathOperator{\Spf}{Spf}

\DeclareMathOperator{\Fil}{Fil}

\DeclareMathOperator{\Spa}{Spa}

\DeclareMathOperator{\Gr}{Gr}

\newcommand{\colim@}[2]{%
  \vtop{\m@th\ialign{##\cr
    \hfil$#1\operator@font colim$\hfil\cr
    \noalign{\nointerlineskip\kern1.5\ex@}#2\cr
    \noalign{\nointerlineskip\kern-\ex@}\cr}}%
}
\newcommand{\colim}{%
  \mathop{\mathpalette\colim@{}}\nmlimits@
}

\renewcommand{\sp}{\mathrm{sp}}

\newcommand{\B}{\mathrm B}

\newcommand{\A}{\mathbb{A}}

\newcommand{\D}{\mathbb{D}}

\newcommand{\W}{\bb W}

\newcommand{\bb}[1]{\mathbb{#1}}

\newcommand{\mc}[1]{\mathcal{#1}}
\newcommand{\mbb}[1]{\mathbb{#1}}

\newcommand{\mf}[1]{\mathfrak{#1}}

\usepackage{relsize}

\newcommand{\ol}{\overline}

\newcommand{\wh}{\widehat}
\newcommand{\wt}{\widetilde}

\newcommand{\Gm}{\mathbb{G}_{{m}}}
\newcommand{\Ga}{\mathbb{G}_{a}}

\newcommand{\hyphen}{\mathchar`-}

\newcommand{\mb}[1]{\mathbf{#1}}
\newcommand{\Z}{\mathbb{Z}}

\renewcommand{\ll}{\llbracket}
\newcommand{\rr}{\rrbracket}

\newcommand{\syn}{{\mr{syn}}}
\newcommand{\framew}{w}

\makeatletter
\renewcommand{\email}[2][]{%
  \ifx\emails\@empty\relax\else{\g@addto@macro\emails{,\space}}\fi%
  \@ifnotempty{#1}{\g@addto@macro\emails{\textrm{(#1)}\space}}%
  \g@addto@macro\emails{#2}%
}
\makeatother

\newcommand{\triv}{\mathrm{triv}}

\newcommand{\stacks}[1]{\cite[\href{https://stacks.math.columbia.edu/tag/#1}{Tag~#1}]{StacksProject}}

\newcommand{\alg}{\mathrm{alg}}

\newcommand{\an}{\mathrm{an}}

\newcommand{\cat}[1]{\mathbf{#1}}

\newcommand{\ms}[1]{\mathscr{#1}}

\newcommand{\fl}{\mathrm{fl}}
\newcommand{\Zar}{\mathrm{Zar}}

\newcommand{\perf}{\mathrm{perf}}

\def\Item(#1){\item[\llap{(}\refstepcounter{enumi}$\bullet$] #1)}

\DeclareMathOperator*{\twolim}{2-lim}

\DeclareMathAccent{\wtilde}{\mathord}{largesymbols}{"65}

\newcommand*\isomto{%
        \xrightarrow{\raisebox{-0.2 em}{\smash{\ensuremath{\sim}}}}%
    }

    \newcommand*\isomfrom{%
        \xleftarrow{\raisebox{-0.2 em}{\smash{\ensuremath{\sim}}}}%
    }

    \newcommand{\ov}[1]{\overline{#1}}
    \newcommand{\et}{\mathrm{\acute{e}t}}

    \newcommand{\qsyn}{\mathrm{qsyn}}

    \newcommand{\qrsp}{\mathrm{qrsp}}

    \newcommand{\Ainf}{\mathrm{A}_\mathrm{inf}}
    \newcommand{\Acrys}{\mathrm{A}_\mathrm{crys}}
    
    \newcommand{\Bdr}{\mathrm{B}_\mathrm{dR}}
    
   \newcommand{\defeq}{\vcentcolon=}
    
    \newcommand{\be}{\begin{equation*}}
    \newcommand{\ee}{\end{equation*}}
    \newcommand{\bx}{\begin{equation*}\xymatrix}
    \newcommand{\ex}{\end{equation*}}

\usepackage{relsize}
\usepackage[bbgreekl]{mathbbol}
\usepackage{amsfonts}

\DeclareSymbolFontAlphabet{\mathbbl}{bbold}
\newcommand{\Prism}{{\mathlarger{\mathbbl{\Delta}}}}

\renewcommand{\inf}{\mathrm{inf}}

\newcommand{\crys}{\mathrm{crys}}

\newcommand{\smallprism}{{{\mathsmaller{\Prism}}}}
\newcommand{\smallN}{{{\mathsmaller{\mc{N}}}}}
\newcommand{\mr}{\mathrm}

\newcommand{\dR}{\mathrm{dR}}

\newcommand{\cf}{\textit{cf.\ }}

\title{An integral analogue of Fontaine's crystalline functor}
\author{Naoki Imai$^{(1)}$}
\address[1]{\scriptsize Graduate School of Mathematical Sciences, The University of Tokyo,
    3-8-1 Komaba, Meguro-ku, Tokyo, 153-8914, Japan}
\email[1]{\scriptsize naoki@ms.u-tokyo.ac.jp}
\author{Hiroki Kato$^{(2)}$}
\address[2]{\scriptsize 
Institut des Hautes Études
Scientifiques, 35 route de Chartres, 91440 Bures-sur-Yvette, France
}
\address[3]{\scriptsize Department of Mathematics, Bahen Centre, University of Toronto, Toronto, ON, M5S 2E4, Canada}
\email[2]{\scriptsize hiroki@ihes.fr}
\author{Alex Youcis$^{(3)}$}

\email[3]{\scriptsize alex.youcis@gmail.com}

\date{\today}

\begin{document}
\begin{abstract}
    For a smooth formal scheme $\mf{X}$ over the Witt vectors $W$ of a perfect field $k$, we construct a functor $\mathbb{D}_\mr{crys}$ from the category of prismatic $F$-crystals $(\mc{E},\varphi_\mc{E})$ (or prismatic $F$-gauges) on $\mf{X}$ to the category of filtered $F$-crystals on $\mf{X}$. We show that $\bb{D}_\mr{crys}(\mc{E},\varphi_\mc{E})$ enjoys strong properties when $(\mc{E},\varphi_\mc{E})$ is what we call \emph{locally filtered free (lff)}. Most significantly, we show that $\mathbb{D}_\mr{crys}$ actually induces an equivalence between the category of prismatic $F$-gauges on $\mf{X}$ with Hodge--Tate weights in $[0,p-2]$ and the category of Fontaine--Laffaille modules on $\mf{X}$. Finally, we use our functor $\bb{D}_\crys$ to enhance the study of prismatic Dieudonn\'e theory of $p$-divisible groups (as initiated by Ansch\"{u}tz--Le Bras) allowing one to recover the \emph{filtered} crystalline Dieudonn\'e crystal from the prismatic Dieudonn\'e crystal. This in turn allows us to clarify the relationship between prismatic Dieudonn\'e theory and the work of Kim on classifying $p$-divisible groups using Breuil--Kisin modules.
\end{abstract}

\maketitle

\tableofcontents

\section*{Introduction}

Let $k$ be a perfect extension of $\mathbb{F}_p$, and $\mathfrak{X}$ be a smooth formal scheme over $W\defeq W(k)$ with generic fiber $X$. A question of central importance in $p$-adic Hodge theory is when a $\bb{Q}_p$-local system $\bb{V}$ on $X$ should admit `good reduction' relative to the model $\mf{X}$. The operative property singled out by Fontaine in this regard is that $\bb{V}$ is \emph{crystalline}. In this case, one may associate to $\bb{V}$ a filtered $F$-isocrystal $D_\mr{crys}(\bb{V})$ on $\mf{X}$ which one may think of as the `model' of $\bb{V}$ over $\mf{X}$.

But, for arithmetic applications, for example to the study of Shimura varieties (e.g., see \cite{IKY2}), it is important to have an \emph{integral analogue} of this picture. Namely, if $\bb{L}\subseteq \bb{V}$ is a $\bb{Z}_p$-lattice, the collection of which (as $\bb{V}$ varies) we denote $\cat{Loc}_{\bb{Z}_p}^\mr{crys}(X)$, then one would like to make sense of when $\bb{L}$ admits good reduction relative to $\mf{X}$. The question of what sort of object should model $\bb{L}$ is subtle, and there now exist several approaches.

\medskip

\noindent\textbf{Crystalline approach (\cite{Faltings89}):}
This utilizes the category $\cat{VectF}^{\varphi,\mr{sd}}(\mf{X}_\mr{crys})$ of strongly divisible filtered $F$-crystals. This theory is best behaved when the filtration is supported in the Fontaine--Laffaille range $[0,p-2]$, where they are called \emph{Fontaine--Laffaille modules} after \cite{FL82} (cf.\@ Proposition \ref{prop:filtered-F-crystals}). In particular, in \cite{Faltings89} there is constructed a fully faithful functor 
\begin{equation*}
T_\crys\colon \cat{VectF}^{\varphi,\mr{sd}}_{[0,p-2]}(\mf{X}_\crys)\to\cat{Loc}_{\Z_p,[0,p-2]}^\mr{crys}(X), 
\end{equation*}
where the target consists of those $\bb{L}$ with Hodge--Tate weights in $[0,p-2]$.

\medskip

\noindent\textbf{Prismatic approach (\cite{BhattScholzeCrystals}):} This makes use of the category $\cat{Vect}^\varphi(\mf{X}_\smallprism)$ of prismatic $F$-crystals on $\mf{X}$. The relationship to local systems is given by the functor
\begin{equation*}
T_\et\colon \cat{Vect}^\varphi(\mf{X}_\smallprism)\to \cat{Loc}^\crys_{\bb{Z}_p}(X),
\end{equation*}
(constructed in op.\@ cit.\@), which was shown to be fully faithful in \cite{GuoReinecke} or \cite{DLMS}.

\medskip

\noindent\textbf{Syntomic approach (\cite{Drinfeld}, \cite{BhattNotes}):} This utilizes the category $\cat{Vect}(\mf{X}^\mr{syn})$ of prismatic $F$-gauges. By results in \cite{GuoLi}, this also admits a fully faithful functor 
\begin{equation*}
T_\et\colon \cat{Vect}(\mf{X}^\syn)\to \cat{Loc}^\crys_{\bb{Z}_p}(X).
\end{equation*}

It is natural to ask what the precise relationship is between these three approaches. The latter two are directly related by a \emph{forgetful functor} 
\begin{equation*} \mr{R}_\mf{X}\colon \cat{Vect}(\mf{X}^\syn)\to \cat{Vect}^\varphi(\mf{X}_\smallprism),
\end{equation*}
which is shown in \cite{GuoLi} and \cite{IKY2} to be fully faithful with essential image the subcategory $\cat{Vect}^{\varphi,\mr{lff}}(\mf{X}_\smallprism)$ of so-called locally filtered free (lff) prismatic $F$-crystals (see op.\@ cit.\@).

The goal of this paper is to complete the comparisons of these three approaches by relating the crystalline approach to the prismatic and syntomic ones by defining \emph{integral analogues} of $D_\mr{crys}$
\begin{equation*}
\bb{D}_\mr{crys}\colon \cat{Vect}^\varphi(\mf{X}_\smallprism)\to \cat{VectWF}^\varphi(\mf{X}_\mr{crys}),\qquad  \bb{D}_\mr{crys}\colon\cat{Vect}(\mf{X}^\mr{syn})\to \cat{VectF}^{\varphi,\pddiv}(\mf{X}_\mr{crys});
\end{equation*}
see \S\ref{ss:category-of-filtered-f-crystals} for the definitions of $\cat{VectWF}^\varphi(\mf{X}_\mr{crys})$ and $\cat{VectF}^{\varphi,\pddiv}(\mf{X}_\mr{crys})$, but we remark that the latter agrees with $\cat{VectF}^{\varphi,\mr{sd}}(\mf{X}_\mr{crys})$ when restricted to objects with Hodge--Tate weights in the Fontaine--Laffaille range $[0,p-1]$. In the rest of the introduction, we discuss the various good properties $\bb{D}_\mr{crys}$ enjoys, which not only provide the desired bridge between these various approaches, but help clarify the structure of the categories $\cat{Vect}^\varphi(\mf{X}_\smallprism)$ and $\cat{Vect}(\mf{X}^\mr{syn})$. 

But, we immediately state that the tightest possible connection holds, namely \emph{equivalence}, in the Fontaine--Laffaille range, which clarifies the well-behavedness of Faltings's theory in this range. (See \cite[Definition 3.16]{DLMS} and \S\ref{ss:statement-of-main} for the definitions of the categories $\cat{Vect}_{[0,p-2]}^\varphi(\mf{X}_\smallprism)$ and $\cat{Vect}_{[0,p-2]}(\mf{X}^\mr{syn})$, respectively.)

\begin{thmi}[{see Theorem \ref{thm:big-equiv-diagram}}]\label{thmi:big-equiv-diagram} Suppose that $\mf{X}$ is a smooth formal $W$-scheme and that $p>2$. Then, the following diagram is commutative and all arrows are $\bb{Z}_p$-linear equivalences
\begin{equation*}
\begin{tikzcd}[sep=large]
	{\cat{Vect}^{\varphi,\mr{lff}}_{[0,p-2]}(\mf{X}_\smallprism)} && {\cat{Vect}_{[0,p-2]}(\mf{X}^\mr{syn})} \\
	{\cat{Vect}^{\varphi}_{[0,p-2]}(\mf{X}_\smallprism)} & {\cat{Loc}_{\bb{Z}_p,[0,p-2]}^{\mr{crys}}(X)} & {\cat{VectF}^{\varphi,\mr{sd}}_{[0,p-2]}(\mf{X}_\mr{crys})}
	\arrow[hook, from=1-1, to=2-1]
	\arrow["{\bb{D}_\mr{crys}}"{description}, from=1-1, to=2-3]
	\arrow["{\mr{R}_\mf{X}}"{description}, from=1-3, to=1-1]
	\arrow["{\bb{D}_\mr{crys}}"{description}, from=1-3, to=2-3]
	\arrow["{T_\et}"{description}, from=2-1, to=2-2]
	\arrow["{T_\mr{crys}}"{description}, from=2-3, to=2-2]
\end{tikzcd}
\end{equation*}
\end{thmi}

\begin{remi}\label{rem:relation-to-TVX} We comment on the relationship of Theorem \ref{thm:big-equiv-diagram} to other literature: 
\begin{itemize}[leftmargin=.3in]
    \item In \cite{Hokaj}, Hokaj obtains the equivalence of the arrow $T_\mr{crys}$ as in Theorem \ref{thm:big-equiv-diagram}. Our proof is independent of his results, but our proof of the crucial Proposition \ref{prop:lff-FL-range} is inspired by (although simpler than) techniques in op.\@ cit.
    \item In \cite{Wur23}, W\"{u}rthen constructs an equivalence between $\cat{Vect}^\varphi_{[0,p-2]}(\mf{X}_\smallprism)$ and a certain modified category of Fontaine--Laffaille modules. The relationship to our work is unclear.
    \item In \cite{TVX}, for a perfect field $k$ of characteristic $p$ the authors construct an equivalence of $\infty$-categories (with notation as in op.\@ cit):
\begin{equation*}
\Phi_{\mr{Maz}}\colon \mc{D}_{\mr{qc},[0,p-2]}(W(k)^\mr{syn})\to \ms{DMF}^\mr{big}_{[0,p-2]}(W(k)),
\end{equation*}
which is a derived analogue of $\bb{D}_\mr{crys}$ for $\mf{X}=\Spf(W(k))$ in the Fontaine--Laffaille range.
\end{itemize}

We note that in all cases above these constructions do not address objects outside the Fontaine--Laffaille range, unlike our functor $\bb{D}_\mr{crys}$.
\end{remi}

In the rest of the introduction, we discuss other aspects of the functor $\bb{D}_\mr{crys}$ and its applications to prismatic $F$-crystals, prismatic $F$-gauges, and the theory of $p$-divisible groups.

\medskip

\paragraph*{Construction of $\bb{D}_\mr{crys}$ and applications of prismatic $F$-crystals/gauges} Fix a smooth formal $W$-scheme $\mf{X}$.\footnote{Although most of these results apply for a base formal $W$-scheme $\mf{X}$ in the sense of \cite[\S1.1.5]{IKY1}.} Our construction of the $\Z_p$-linear $\otimes$-functor, the \emph{integral analogue of ${D}_\mr{crys}$},
\begin{equation*}
\bb{D}_\crys\colon \cat{Vect}^\varphi(\mf{X}_\smallprism)\to \cat{VectWF}^{\varphi}(\mf{X}_\crys)
\end{equation*}
is achieved in a surprisingly pleasant way. Namely, we observe that for an object $(\mc{E},\varphi_\mc{E})$ of $\cat{Vect}^\varphi(\mf{X}_\smallprism)$ there is always a canonical \emph{crystalline-de Rham comparison}.

\begin{thmi}[see Theorem \ref{thm:crys-dR-comparison}]\label{thmi:crys-dr-comp}
    Let $\mf X$ be a smooth formal $W$-scheme, and $(\mc E,\varphi_\mc{E})$ a prismatic $F$-crystal on $\mf X$. Then there exists a canonical isomorphism 
    \begin{equation*}
        \iota_\mf{X}\colon \underline{\bb{D}}_\crys(\mc{E},\varphi_\mc{E})|_{\mf{X}_\mr{Zar}}\isomto \bb{D}_\dR(\mc{E},\varphi_\mc{E})
    \end{equation*}
    of vector bundles on $\mf X$. 
\end{thmi}
Here we are using the following notation:
\begin{itemize}[leftmargin=.3in]
\item $\underline{\bb{D}}_\crys(\mc{E},\varphi_\mc{E})$, the \emph{crystalline realization}, denotes the $F$-crystal $\mc{E}^\crys$ associated to $\mc{E}|_{\mf{X}_k}$, 
\item and $\bb{D}_\dR(\mc{E},\varphi_\mc{E})$, the \emph{de Rham realization}, denotes the vector bundle on $\mf{X}$ given by pullback along the de Rham point $\rho_{\mr{dR},\mf{X}}^\smallprism\colon \mf{X}\to \mf{X}^\smallprism$ as in Definition \ref{defn:de-rham-point} (see Proposition \ref{prop:de-Rham-realization-BK-relationship} for a more down-to-earth description).
\end{itemize}

Then, $\bb{D}_\crys(\mc{E},\varphi_\mc{E})$ has underlying $F$-crystal $\underline{\bb{D}}_\crys(\mc{E},\varphi_\mc{E})$ and obtains a filtration on its restriction to $\mf{X}_\Zar$ via the crystalline-de Rham comparison, and the Nygaard filtration on $\phi^\ast \mc{E}$.

\begin{remi} In the case when $\mf{X}=\Spf(\Z_p)$ this is essentially contained in \cite{BhattLurieAbsolute} (see Proposition 3.6.6 of op.\@ cit.\@). A more novel aspect of our study is a concrete reinterpretation of this comparison using Breuil and Breuil--Kisin prisms, see Construction \ref{construction: crys dR isom with Breuil prism} and Proposition \ref{prop:crys-dR-comparison-using-prisms}, which is central to our study of $\bb{D}_\mr{crys}$.
\end{remi}

\begin{remi}
In the case when $\mf{X}=\Spf(W)$, one may view the crystalline-de Rham isomorphism as an integral refinement of \cite[\S 1.2.7]{KisinFCrystal} (cf.\@ Remark \ref{rem:Kisin agreement}). 
\end{remi}

As indicated by our choice of terminology, $\bb{D}_\crys$ forms a functorial lattice in $D_\mr{crys}$ in a way made precise by the following theorem.

\begin{thmi}[{see Theorem \ref{thm:integral-dcrys-strongly-divisible-and-rational-agreeance}}] There is a canonical identification
\begin{equation*}
\bb{D}_\crys[\nicefrac{1}{p}]\isomto D_\crys\circ T_\et.
\end{equation*}
\end{thmi}

The functor $\bb{D}_\mr{crys}$ enjoys even more favorable properties when restricted to the full subcategory $\cat{Vect}^{\varphi,\mr{lff}}(\mf{X}_\smallprism)$ of lff prismatic $F$-crystals as in \cite[Definition 1.24]{IKY2}.

\begin{propi}[{see Proposition \ref{prop:filtered-equiv}}]\label{propi:lff-iff-Dcrys-divisible} Suppose that $\mf{X}$ is a smooth formal $W$-scheme. Then, if $(\mc{E},\varphi_\mc{E})$ is an lff prismatic $F$-crystal on $\mf{X}$, then $\bb{D}_\mr{crys}(\mc{E},\varphi_\mc{E})$ belongs to $\cat{VectF}^{\varphi,\pddiv}(\mf{X}_\mr{crys})$.
\end{propi}

Combining Proposition \ref{propi:lff-iff-Dcrys-divisible} with the forgetful functor
\begin{equation*}
\mr{R}_\mf{X}\colon \cat{Vect}(\mf{X}^\mr{syn})\isomto \cat{Vect}^{\varphi,\mr{lff}}(\mf{X}_\smallprism)
\end{equation*}
(see \cite[Remark 6.3.4]{BhattNotes} and \cite[Proposition 1.28]{IKY2}), we obtain a functor
\begin{equation}\label{eqi:Dcrys-syntomic}\tag{1}
\bb{D}_\mr{crys}\circ \mr{R}_\mf{X}\colon \cat{Vect}(\mf{X}^\mr{syn})\to \cat{VectF}^{\varphi,\pddiv}(\mf{X}_\mr{crys}),
\end{equation}
which we also denote $\bb{D}_\mr{crys}$. Using ideas of Faltings, we may show that \eqref{eqi:Dcrys-syntomic} is exact when restricted to $\cat{Vect}^{\mr{sd}}(\mf{X}^\mr{syn})$. This latter category is, by definition, those prismatic $F$-gauges whose image under $\bb{D}_\mr{crys}$ lies in $\cat{VectF}^{\varphi,\mr{sd}}(\mf{X}_\mr{crys})$.

Using this, we obtain the following result used heavily in \cite{IKY2}.

\begin{propi}[{see Proposition \ref{prop:G_X-bi-exact}}] The equivalence
\begin{equation*}
\mr{R}_\mf{X}\colon \cat{Vect}^{\mr{sd}}(\mf{X}^\mr{syn})\isomto\cat{Vect}^{\varphi,\mr{lff},\mr{sd}}(\mf{X}_\smallprism)
\end{equation*}
given by the forgetful functor is bi-exact. 
\end{propi}

\paragraph*{Applications to $p$-divisible groups}
We finally describe the ways in which the functor $\bb{D}_\mr{crys}$ can be used to clarify the relationship between $p$-divisible groups and various types of Dieudonn\'e theory when $p>2$. To this end, let us fix $\mf{X}$ to be a base formal $W$-scheme (see \cite[\S1.1.5]{IKY1}).

Denote by $\cat{BT}_p(\mf{X})$ the category of $p$-divisible groups over $\mf{X}$. There are then several approaches to use $p$-adic Hodge theory to classify objects of $\cat{BT}_p(\mf{X})$ via various `Dieudonn\'e theories':
\begin{enumerate}[leftmargin=.3in]
\item the \emph{filtered crystalline Dieudonn\'e functor}  of Grothendieck--Messing
\begin{equation*}
\bb{D}\colon \cat{BT}_p(\mf{X})\to \cat{VectF}^{\varphi,\mr{sd}}_{[0,1]}(\mf{X}_\crys),
\end{equation*}
which is an (anti-)equivalence by results of Grothendieck, Messing, and de Jong (see \cite{deJongCrystalline}),
 \item the \emph{Breuil--Kisin--Kim Dieudonn\'e functor} (see \cite{KimBK})
 \begin{equation*}
 \mf{M}\colon \cat{BT}_p(\mf{X})\to \cat{Vect}^\varphi_{[0,1]}(\mf{S}_R,\nabla^0)
 \end{equation*}
 when $\mf{X}=\Spf(R)$ for a formally framed $W$-algebra $R$ (see \cite[\S1.1.5]{IKY1} for the definition of framed algebra, and Definition \ref{defn:Kisin-modules} for the definition of the target category), which is an (anti-)equivalence by \cite{KimBK},

\item the \emph{prismatic Dieudonn\'e functor} of Ansch\"{u}tz--Le Bras
\begin{equation*}
\mc{M}_\smallprism\colon \cat{BT}_p(\mf{X})\isomto \cat{Vect}_{[0,1]}^\varphi(\mf{X}_\smallprism),
\end{equation*}
which is an (anti-)equivalence by \cite{AnschutzLeBrasDD} (cf.\@ Theorem \ref{thm:ALB-equiv}).
\end{enumerate}

The relationship between these various Dieudonn\'e theories is important to establish the connection between historical results and modern methods (e.g., as is needed in the study of Shimura varieties in \cite{IKY2}). That said, this relationship is far from clear given the different ways in which these functors are constructed. 

In \cite[Theorem 4.44]{AnschutzLeBrasDD}, it is demonstrated that there is a canonical identification
\begin{equation*}
\underline{\bb{D}}(H)\simeq \underline{\bb{D}}_\mr{crys}(\mc{M}_\smallprism(H)),
\end{equation*}
in $\cat{Vect}^\varphi(\mf{X}_\mr{crys})$, for an object $H$ of $\cat{BT}_p(\mf{X})$. Here, recall, $\underline{\bb{D}}_\mr{crys}(\mc{M}_\smallprism(H))$ is the underlying $F$-crystal of $\bb{D}_\mr{crys}(\mc{M}_\smallprism(H))$. We further denote by $\underline{\bb{D}}(H)$ the underlying $F$-crystal of $\bb{D}(H)$. In particular, this identification ignores filtrations, which is equivalent (by Grothendieck--Messing theory) to only remembering the special fiber $H_k$ of the $p$-divisible group $H$. Thus, it is desirable to upgrade this to a filtered identification.

\begin{thmi}[{see Theorem \ref{thm:ALB-dJ-comparison}}]\label{thmi:ALB-dJ-comparison} There is an identification of functors $\cat{BT}_p(\mf{X})\to \cat{VectF}^{\varphi,\mr{sd}}_{[0,1]}(\mf{X}_\mr{crys})$
\begin{equation*}
\bb{D}\simeq \bb{D}_\mr{crys}\circ \mc{M}_\smallprism.
\end{equation*}
\end{thmi}

With respect to the Breuil--Kisin--Kim Dieudonn\'e functor, when $R=W$, it is shown in \cite[Proposition 5.18]{AnschutzLeBrasDD} that there is an identification of functors $\cat{BT}_p(R)\to \cat{Vect}^\varphi_{[0,1]}(\mf{S}_R)$
\begin{equation*}
\mf{M}\simeq \mr{ev}_{\mf{S}_R}\circ \mc{M}_\smallprism,
\end{equation*}
where $\mr{ev}_{\mf{S}_R}$ is the functor given by evaluation on the Breuil--Kisin prism $(\mf{S}_R,(E))$. But for applications, it is desirable to have versions of this identification for arbitrary $R$, which then necessitates the consideration of connections.

Using Theorem \ref{thmi:ALB-dJ-comparison}, and an intermediate study of Breuil's Dieudonn\'e functor (see \S\ref{ss:Breuil-DD-functor}), we are able to make such an identification.

\begin{propi}[{see Proposition \ref{prop:Kim-prismatic-comp}}] For any formally framed $W$-algebra $R$, there is a canonical identification of functors $\cat{BT}_p(R)\to\cat{Vect}^\varphi_{[0,1]}(\mf{S}_R,\nabla^0)$
\begin{equation*}
\mf{M}\simeq \mr{ev}_{\mf{S}_R}^\mr{K}\circ \mc{M}_\smallprism . 
\end{equation*}
\end{propi}

Here $\mr{ev}_{\mf{S}_R}^\mr{K}(\mc{M}_\smallprism(H))$ is the evaluation of $\mc{M}_\smallprism(H)$ on $(\mf{S}_R,(E))$ equipped with the connection on $\phi^\ast\mc{M}_\smallprism(\mf{S}_R,(E))/E$ coming from the identification of this with $\bb{D}_\mr{crys}(H)$ (via Theorem \ref{thmi:ALB-dJ-comparison}).

As an application of this result and work of Kim, we obtain the following strengthening of a result of Ito (cf.\@ \cite[Theorem 6.1.3]{Ito1}).

\begin{cori}[{Corollary \ref{cor:F-gauges-BK-modules}}] Let $R$ be a formally framed $W$-algebra. Then, there are equivalences of categories
\begin{equation*}
\begin{tikzcd}
	{\cat{Vect}_{[0,1]}(R^\mr{syn})} & {\cat{Vect}_{[0,1]}^\varphi(R_\smallprism)} & {\cat{Vect}^{\varphi}_{[0,1]}(\mf{S}_R,\nabla^0).}
	\arrow["{\mr{R}_\mf{X}}", from=1-1, to=1-2]
	\arrow["\sim"', from=1-1, to=1-2]
	\arrow["\sim"', from=1-2, to=1-3]
	\arrow["{\mr{ev}^{\mr{K}}_{\mf{S}_R}}", from=1-2, to=1-3]
\end{tikzcd}
\end{equation*}
If $R=W\ll t_1,\ldots,t_d\rr$ for some $d\geqslant 0$, then there are equivalences
\begin{equation*}
\begin{tikzcd}
	{\cat{Vect}_{[0,1]}(R^\mr{syn})} & {\cat{Vect}_{[0,1]}^\varphi(R_\smallprism)} & {\cat{Vect}^{\varphi}_{[0,1]}(\mf{S}_R,(E)).}
	\arrow["{\mr{R}_\mf{X}}", from=1-1, to=1-2]
	\arrow["\sim"', from=1-1, to=1-2]
	\arrow["\sim"', from=1-2, to=1-3]
	\arrow["{\mr{ev}_{\mf{S}_R}}", from=1-2, to=1-3]
\end{tikzcd}
\end{equation*}
\end{cori}

\medskip

\paragraph*{Acknowledgments} The authors would like to heartily thank Ahmed Abbes, Abhinandan,  Bhargav Bhatt,  Kazuhiro Ito, Wansu Kim, Arthur-C\'{e}sar Le Bras, Tong Liu, Kojiro Matsumoto, Koji Shimizu, Takeshi Tsuji, Vadim Vologodsky and Qixiang Wang for helpful discussions and comments. Part of this work was conducted during a visit to the Hausdorff Research Institute for Mathematics, funded by the Deutsche Forschungsgemeinschaft (DFG, German Research Foundation) under Germany's Excellence Strategy – EXC-2047/1 – 390685813. This work was supported by JSPS KAKENHI Grant Numbers 22KF0109, 22H00093 and 23K17650, the European Research Council (ERC) under the European Union’s Horizon 2020 research and innovation programme (grant agreement No. 851146), and funding through the Max Planck Institute for Mathematics in Bonn, Germany (report numbers MPIM-Bonn-2022, MPIM-Bonn-2023, MPIM-Bonn-2024). 

\medskip

\paragraph*{Notation and terminology}\label{notation-and-terminology}

Throughout this article we make use of the following notation and terminology, which is largely standard (or self-evident), and so the reader is encouraged to refer back only when necessary.

Fix the following (we refer the reader to  \cite[\S1.1.5]{IKY1} for undefined terminology):
\begin{multicols}{2}
\begin{itemize}[leftmargin=.1in,label=$\diamond$]
\item $k$ a perfect extension of $\bb{F}_p$,
\item $W\defeq W(k)$ and $K_0\defeq \mr{Frac}(W)$,
\item $K$ is a finite totally ramified extension of $K_0$, 
\item $\mc{O}_K$ the valuation ring of $K$,
\item $\pi$ a uniformizer of $\mc{O}_K$ and $k=\mc{O}_K/\pi$,
\item $e=[K:K_0]$,
\item $E\in W[u]$ the minimal polynomial for $\pi$,
\item $R$ a formally framed base $\mc{O}_K$-algebra,
\item $(\mf{S}_R,(E))$ the Breuil--Kisin prism,
\item $(S_R,(p))$ the Breuil prism,
\item $\bb{W}$ the $p$-typical Witt vector scheme,
\item $F\colon \bb{W}\to \bb{W}$ the Frobenius,
\item $V\colon \bb{W}\to\bb{W}$ the Verschiebung,
\item $[-]\colon \bb{A}^1\to \bb{W}$ the Teichm\"{u}ller lift,
\item $\mf{X}$ a base formal $\mc{O}_K$-scheme,
\item $X=\mf{X}_\eta$ the rigid generic fiber.
\end{itemize}
\end{multicols}

We refer the reader to \cite[\S2.3.1]{IKY1} for our notation and terminology concerning the crystalline site and (iso)crystals and $F$-(iso)crystals. 

We further refer the reader to \cite[\S1.1.3]{IKY1} for our notation and terminology concerning the quasi-syntomic and qrsp sites $\mf{Y}_\mr{qsyn}$ and $\mf{Y}_\mr{qrsp}$. 

We lastly refer the reader to \cite[\S1]{IKY2} for our notation and terminology concerning 
\begin{itemize}[leftmargin=.1in,label=$\diamond$]
\item the formal stacks $\mf{Y}^\smallprism$ (with its Frobenius $F_\mf{Y}\colon\mf{Y}^\smallprism\to\mf{Y}^\smallprism$), $\mf{Y}^\smallN$, and $\mf{Y}^\mr{syn}$ over $\mathbb{Z}_p$ and for the site $\mf{Y}_\smallprism$ (which we abbreviate to $S^\smallprism, S^\smallN$, $S^\mr{syn}$, and $S_\smallprism$ when $\mf{Y}=\Spf(S)$),
\item the categories $\cat{Vect}^\varphi(\mf{X}_\smallprism)$, $\cat{Vect}^{\varphi,\mr{lff}}(\mf{X}_\smallprism)$, $\cat{Vect}^{\varphi,\mr{an}}(\mf{X}_\smallprism)$, and $\cat{Vect}(\mf{X}^\mr{syn})$, of prismatic $F$-crystals, locally filtered free (lff) prismatic $F$-crystals, analytic prismatic $F$-crystals and prismatic $F$-gauges, respectively.
\item the \'etale realization functor $T_\et\colon \cat{Vect}^{\an,\varphi}(\mf{X}_\smallprism)\to\cat{Loc}_{\bb{Z}_p}^\mr{crys}(X)$ and the induced \'etale realization functors on  $\cat{Vect}^\varphi(\mf{X}_\smallprism)$, $\cat{Vect}^{\varphi,\mr{lff}}(\mf{X}_\smallprism)$, and $\cat{Vect}(\mf{X}^\mr{syn})$,
\item $\mr{R}_\mf{X}\colon \cat{Vect}(\mf{X}^\mr{syn})\isomto \cat{Vect}^{\varphi,\mr{lff}}(\mf{X}_\smallprism)$ the forgetful functor,
\item notation for filtered rings and modules and, in particular, the Rees algebra $\mr{Rees}(A,\Fil^\bullet)$, the Rees stack $\mc{R}(A,\Fil^\bullet_A)$, and the completed Rees stack $\wh{\mc{R}}(A,\Fil^\bullet)$ of a filtered ring. If $A$ is clear from context, we will often omit it from the notation.
\end{itemize} 
Finally, we freely use the notion of quasi-ideals $d\colon I\to A$ (which we sometimes write as $[I\to A]$ for clarity) and their quotients $\mr{Cone}(d)$ as in \cite{DrinfeldRingGroupoids}.

\section{The crystalline-de Rham comparison} The definition of $\bb{D}_\mr{crys}$ relies on a comparison isomorphism between the crystalline and de Rham realizations of a prismatic $F$-crystal that we formulate and prove in this subsection. Throughout we use notation and terminology from \hyperref[notation-and-terminology]{Notation and terminology} without comment.

\subsection{The crystalline realization functor} We now expand on the notion of the \emph{crystalline realization functor} $\underline{\bb{D}}_\crys\colon \cat{Vect}^\varphi(\mf{X}_\smallprism)\to \cat{Vect}^\varphi(\mf{X}_\crys)$ as discussed in \cite[Example 4.12]{BhattScholzeCrystals}.

\subsubsection{Prismatic \texorpdfstring{$F$}{F}-crystals on quasi-syntomic \texorpdfstring{$\bb{F}_p$}{Fp}-schemes}\label{par:crys-syn-equiv} Fix $Z$ to be a quasi-syntomic $k$-scheme. As in \cite[Example 4.7]{BhattScholzeCrystals}, there is an equivalence of categories
\begin{equation*}
    (-)^\crys\colon \cat{Vect}(Z_\smallprism,\mc{O}_\smallprism)\isomto \cat{Vect}(Z_\crys),\qquad \mc{F}\mapsto \mc{F}^\crys,
\end{equation*} 
defined as follows. 
Recall (e.g., see \cite[Example 1.8 and Example 2.16]{IKY1}) that for a qrsp $k$-algebra $R$, the categories $R_\smallprism$ and $(R/W)_\crys$ have initial objects $(\Acrys(R),(p),\wt{\mr{nat}}.)$ and $\Acrys(R)\twoheadrightarrow R$, respectively. Thus, by evaluation, we have functorial equivalences 
\be
\cat{Vect}(R_\smallprism,\mc O_\smallprism)\isomto \cat{Vect}(\Acrys(R))\isomfrom \cat{Vect}(R_\crys).
\ee
Passing to the limit, and using \cite[Proposition 1.31 and Proposition 2.17]{IKY1}, we deduce the existence of a diagram of equivalences 
\begin{equation*}
     \cat{Vect}(Z_\smallprism,\mc{O}_\smallprism)\isomto \twolim_{\scriptscriptstyle R\in Z_\qrsp}\cat{Vect}(R_\smallprism,\mc O_\smallprism)\isomto\twolim_{\scriptscriptstyle R\in Z_\qrsp}\cat{Vect}(R_\crys)\isomfrom \cat{Vect}(Z_\crys).
\end{equation*}
Define $(-)^\crys$ to be the obvious equivalence derived from this diagram.

For an object $\mc{V}$ of $\cat{Vect}(Z_\smallprism,\mc{O}_\smallprism)$, one has that $(\phi^\ast\mc{V})^\crys$ is naturally identified with $\phi^\ast(\mc{V}^\crys)$. Indeed, by construction it suffices to observe that for a qrsp $k$-algebra $R$ over $Z$
\begin{equation*}
\phi^\ast(\mc{V}^\crys)(\Acrys(R)\twoheadrightarrow R)=\phi_{R}^\ast\mc{V}(\Acrys(R)\twoheadrightarrow R)=\phi_R^\ast\mc{V}(\Acrys(R),(p))=(\phi^\ast\mc{V})(\Acrys(R),(p)),
\end{equation*}
where the first equality follows from \cite[Remark 2.19]{IKY1} and the second and third by definition. From this observation, one may upgrade $(-)^\crys$ to an equivalence
\begin{equation*}
    (-)^\crys\colon \cat{Vect}^\varphi(Z_\smallprism)\isomto \cat{Vect}^\varphi(Z_\crys),
\end{equation*}
which is functorial in $Z$ in the obvious way.

\subsubsection{The crystalline realization functor} From the above discussion, we obtain a natural functor $(-)^\crys\colon \cat{Vect}(\mf{X}_\smallprism)\to\cat{Vect}((\mf{X}_k)_\crys)$ given by sending $\mc{E}$ to $\mc{E}^\crys\defeq (\mc{E}|_{(\mf{X}_k)_\smallprism})^\crys$. The crystal $\mc{E}^\crys$ enjoys a concrete description when evaluated on base $\mc{O}_K$-algebras.

For the notion of formal framing (usually denoted $t$, not $\framew$, in op.\@ cit.\@) and the notations $\phi_\framew$ and $R_0^{(\phi_\framew)}$ used in the following, we refer the reader to \cite[\S 1.1.5]{IKY1}.

\begin{prop}\label{prop: pris crys in char p} Let $R=R_0\otimes_W \mc{O}_K$ be a base $\mc{O}_K$-algebra and $\framew$ a formal framing. For $\mc{F}$ in $\cat{Vect}((R_k)_\smallprism)$ there is a canonical isomorphism
\begin{equation*}
    \vartheta_\framew=\vartheta_{R,\framew}\colon(\phi^*\mc F)(R_0^{(\phi_\framew)},(p))\isomto \mc F^\crys(R_0). 
\end{equation*} 
If $\mc{F}$ carries a Frobenius structure, then $\vartheta_\framew$ is Frobenius-equivariant.
\end{prop}

Suppose $A$ is a quasi-syntomic $p$-adically complete ring. As $\Spf(A)$ has a cover in $\Spf(A)_\fl$ whose entire \v{C}ech cover is qrsp, we see by $p$-adic faithfully flat descent that the natural map
\begin{equation}\label{eq:p-adic-ff-descent-for-modules}
    M\to \varprojlim_{\scriptscriptstyle B\in A_\qrsp}(M\otimes_A B),
\end{equation}
is an isomorphism for any $p$-adically complete $A$-module $M$.

\begin{proof}[Proof of Proposition \ref{prop: pris crys in char p}]
For any object $T$ of $(R_k)_\qrsp$, there exists a canonical identification of modules $\mc F(\Acrys(T),(p))\isomto \mc F^\crys(\Acrys(T)\twoheadrightarrow T)$ by the definition of $(-)^\crys$. For each object $S$ of $R_\qrsp$, let $\framew_i^\flat$ in $ S^\flat$ be compatible sequences of $p^\text{th}$-power roots of $\framew_i$ in $S/p$. We then have maps $\beta_{\framew^\flat}=\beta\colon R_0\to \Acrys(S)$ as in \cite[\S1.1.5]{IKY1}. The map $\beta$ induces morphisms $(R_0\twoheadrightarrow R_k)\to (\Acrys(S)\twoheadrightarrow S_k)$ and $(R_0,(p),F_{R_k}\circ q)\to (\Acrys(S),(p),\wt{\mr{nat}}.)$ in $(R_k)_\crys$ and $R_\smallprism$, respectively, where $\wt{\mr{nat}}.$ is as in \cite[Example 1.8]{IKY1}, and $q\colon R\to R/\pi=R_0/p$ is the natural map. Thus, from the crystal property and \cite[Remark 1.19]{IKY1}, we obtain isomorphisms 
\begin{equation*}
    \begin{aligned} (\phi^\ast\mc{F})(R_0^{(\phi_w)},(p))\otimes_{R_0}S&\isomto \mc{F}(\Acrys(S),(p))\otimes_{\Acrys(S),\theta}S,\\ \mc F^{\crys}(R_0)\otimes_{R_0}S&\isomto \mc F^{\crys}(\Acrys(S)\twoheadrightarrow S_k)\otimes_{\Acrys(S),\theta}S,\end{aligned}
\end{equation*}
compatible in $S$. Then by the isomorphism in \eqref{eq:p-adic-ff-descent-for-modules}, we get canonical isomorphisms
\be
\begin{aligned}(\phi^*\mc F)(R_0^{(\phi_w)},(p)) &\isomto \varprojlim_{S\in R_\mr{qrsp}}\mc F(\Acrys(S),(p))\otimes_{\Acrys(S),\theta}S\\ &\isomto \varprojlim_{S\in R_\mr{qrsp}}\mc F^{\crys}(\Acrys(S)\twoheadrightarrow S_k)\otimes_{\Acrys{(S)},\theta}S\\ & \isomto \mc F^\crys(R_0),
\end{aligned}
\ee
which proves the assertion. The second claim follows from the first, via the natural identification of $(\phi^\ast\mc{E})^\crys$ and $\phi^\ast(\mc{E}^\crys)$ for an object $\mc{E}$ of $\cat{Vect}(\mf{X}_\smallprism)$.
\end{proof}

 We now define the \emph{crystalline realization functor}
\begin{equation*}
    \underline{\bb{D}}_\crys\colon \cat{Vect}^\varphi(\mf{X}_\smallprism)\to \cat{Vect}^\varphi(\mf{X}_\crys), \qquad (\mc{E},\varphi_\mc{E})\mapsto (\mc{E}^\crys,\varphi_{\mc{E}^\crys}).
\end{equation*}
While $\underline{\bb{D}}_\crys$ is far from full, it is faithful.

\begin{prop}\label{prop:underline-Dcrys-faithful} The following functors are faithful:
\begin{equation*}
    (-)^\crys\colon \cat{Vect}(\mf{X}_\smallprism)\to\cat{Vect}(\mf{X}_\crys),\qquad \underline{\bb{D}}_\crys\colon \cat{Vect}^\varphi(\mf{X}_\smallprism)\to\cat{Vect}^\varphi(\mf{X}_\crys).
\end{equation*} 
\end{prop}
\begin{proof} It suffices to prove the first claim. As passing to a cover is a faithful operation, we see by \cite[Proposition 1.11]{IKY1} that it is sufficient to show the following: if $R$ is a perfectoid ring, then the base change functor $\cat{Vect}(\Ainf(R))\to\cat{Vect}(\Acrys(R))$ given by $M\mapsto M\otimes_{\Ainf(R)}\Acrys(R)$ is faithful. As $M$ is flat over $\Ainf(R)$, and $\Ainf(R)\to\Acrys(R)$ is injective (see \cite[Example 1.9]{IKY1}), the map $M\to M\otimes_{\Ainf(R)}\Acrys(R)$ is injective, from where the claim follows.
\end{proof}

Lastly, we give a calculation of $\underline{\bb{D}}_\crys$ in terms of Breuil--Kisin modules.

\begin{prop}\label{prop:crystalline-realization-Kisin-comparison} Let $R=R_0\otimes_W \mc{O}_K$ be a formally framed base $\mc{O}_K$-algebra. Then,  there is a canonical Frobenius-equivariant isomorphism
\begin{equation*}
\underline{\bb{D}}_\crys(\mc{E},\varphi_\mc{E})(R_0)\isomto (\phi^\ast\mc{E})(\mf{S}_R,(E))/u.
\end{equation*}
\end{prop}
\begin{proof} As $E$ is an Eisenstein polynomial, the map $\mf{S}_R\to R_0$ sending $u$ to $0$ defines a morphism $(\mf{S}_R,(E))\to (R_0,(p))$ in $R_\smallprism$. The desired isomorphism then follows from applying the crystal property in conjunction with Proposition \ref{prop: pris crys in char p}.
\end{proof}

\begin{rem} The Frobenius structure on $(\phi^\ast\mc{E})(\mf{S}_R,(E))/u$ in Proposition \ref{prop:crystalline-realization-Kisin-comparison} is taken in the sense of \cite[Remark 1.19]{IKY1} (either before or after quotienting by $(u)$).
\end{rem}

\begin{example}\label{ex:Kisin-Dcrys} For $R=\mc{O}_K$ we abuse notation and define the exact $\Z_p$-linear $\otimes$-functor
\begin{equation*}
    \underline{\bb{D}}_\crys\colon \cat{Rep}_{\Z_p}^\crys(\Gal(\ov{K}/K))\to \cat{Vect}^\varphi(k_\crys)=\cat{Vect}^\varphi(W,(p)),
\end{equation*}
to be $\underline{\bb{D}}_\crys\circ T_{\et}^{-1}$. Then, there is a natural identification between $\underline{\bb{D}}_\crys(\Lambda)$ and the Frobenius module $\phi^\ast\mf{M}(\Lambda)/u$ over $W$, where $\mf{M}$ is the functor from \cite{KisinFCrystal}. Indeed, this follows from Proposition \ref{prop:crystalline-realization-Kisin-comparison} and \cite[Remark 7.11]{BhattScholzeCrystals}.
\end{example}

\subsubsection{\texorpdfstring{$\underline{\bb{D}}_\mr{crys}$}{DDcrys} in terms of stacks}

We now explain how the crystalline realization functor $\underline{\bb{D}}_\mr{crys}$ can be understood stack-theoretically.

To begin, recall that for a quasi-syntomic $k$-scheme $Z$ there is constructed in \cite[Remark 2.5.12]{BhattNotes} the formal stack $(Z/W)^\mr{crys}$ over $W$ associating to a $p$-nilpotent $W$-algebra $S$
\be
(Z/W)^\crys(S)\defeq \mr{Map}_k(\Spec(\bb{G}_a^\dR(S)),Z), 
\ee
where $\Ga^\dR$ is as in \cite[Definition 2.5.1]{BhattNotes} and $\Ga^\dR(S)$ has $k$-structure as in \cite[Corollary 2.5.10]{BhattNotes}. The absolute Frobenius $F_{Z}$ induces a morphism $F_{Z}\colon (Z/W)^\mr{crys}\to (Z/W)^\mr{crys}$ of formal stacks over $W$ lying over the Frobenius map $\phi_W\colon W\to W$. In particular, it induces a morphism\begin{equation*}
F_{Z}\colon \phi_W^\ast(Z/W)^\mr{crys}\to(Z/W)^\mr{crys}
\end{equation*}
of formal stacks over $W$.

\begin{lem}[{\cite[Corollary 2.6.8]{BhattNotes}}] There exists a natural Frobenius-equivariant isomorphism of formal stacks over $W$
\begin{equation}\label{eq:crys-prismatic-isom}
(Z/W)^\mr{crys}\isomto \phi_W^\ast(Z^\smallprism).
\end{equation}
\end{lem}
\begin{proof} By \cite[Corollary 2.6.8]{BhattNotes}, we have an isomorphism of $W$-algebra schemes
\begin{equation}\label{eq:before-transmutation}
\Ga^\dR\isomto F_*\bb{W}/p.
\end{equation}
That said, for a $p$-nilpotent $W$-algebra $S$ 
\begin{equation*}
\phi^\ast_W(Z^\smallprism)(S)=\mr{Map}_k(\Spec(\bb{W}((\phi_W)_\ast S)/p),Z),
\end{equation*}
where $(\phi_W)_\ast S$ is $S$ viewed as a $W$-algebra via restriction along $\phi_W$. But, it is easy to see that $F_\ast \bb{W}(S)/p$ is canonically isomorphic to $\bb{W}((\phi_W)_\ast(S))/p$ as a $k$-algebra. The claim follows given the identification of $Z^\smallprism$ with  $\mr{Map}_k(\bb{W}/p,Z)$ as explained in \cite[Example 5.1.12]{BhattNotes}.
\end{proof}

Observe that there is a natural functor 
\begin{equation*}
\Psi_{Z}\colon \cat{QCoh}((Z/W)^\crys)\to \cat{Crys}((Z/W)_\crys).
\end{equation*}
Namely, by definition of $\Ga^\dR$, given an object $(Z\gets U\hookrightarrow T)$ of $(Z/W)_\crys$ one obtains a map 
\begin{equation}\label{eq:crys-stack-PD-thickening}
\rho_{(Z\gets U\hookrightarrow T)}\colon T\to (Z/W)^\crys.
\end{equation}
More precisely, if $T=\Spec(B)$ and $J=\ker(\mc{O}_T(T)\to \mc{O}_U(U))$, then we have a natural map $[J\to B]\to [\Ga^\sharp(B)\to B]$ of quasi-ideals.
Set
\begin{equation*}
\Psi_{Z}(\mc{F})(Z\gets U\hookrightarrow T)\defeq (\rho_{(Z\gets U\hookrightarrow T)})^\ast(\mc{F})(T).
\end{equation*}
This is a crystal by the quasi-coherence of $\mc{F}$. 

\begin{prop}\label{prop:Psi-2-commutativity} We have a $2$-commutative diagram of Frobenius-equivariant equivalences
\begin{equation*}
\begin{tikzcd}[sep=2.25em]
	{\cat{Vect}((Z/W)^\crys)} && {\cat{Vect}((Z/W)_\crys)} \\
	& {\cat{Vect}(\phi_W^\ast(Z_\smallprism)).}
	\arrow["{\Psi_Z}", from=1-1, to=1-3]
	\arrow["{\eqref{eq:crys-prismatic-isom}}", from=2-2, to=1-1]
	\arrow["{(-)^\mr{crys}\circ (\phi_W^\ast)^{-1}}"', from=2-2, to=1-3]
\end{tikzcd}
\end{equation*}
\end{prop}
\begin{proof} To show that $\Psi_Z$ is a Frobenius-equivariant equivalence, it suffices to show this diagram $2$-commutes. It further suffices to assume $Z=\Spf(R)$ with $R$ qrsp. But, for an object $\mc{F}$ of $\cat{Vect}(\phi_W^\ast(Z_\smallprism))$, applying \eqref{eq:crys-prismatic-isom} and then $\Psi_Z$ gives the value $\mc{F}(\Acrys(R)\twoheadrightarrow R)$, with twisted $W$-structure map, on $\Acrys(R)\twoheadrightarrow R$. But, this is the same as $((\phi_W^\ast)^{-1}(\mc{F}))^\mr{crys}(\Acrys(R)\twoheadrightarrow R)$.
\end{proof}

Assume now that $\mc{O}_K=W$. Then, our base formal $W$-scheme $\mf{X}$ determines a natural map
\begin{equation*}
i_\mf{X}\defeq\rho_{(\mf X_k\xleftarrow{\id}\mf X_k\hookrightarrow\mf X)}\colon \mf{X}\to (\mf{X}_k/W)^\mr{crys},
\end{equation*}
with notation as in \eqref{eq:crys-stack-PD-thickening}.

\begin{defn} The \emph{(relative) crystalline point} of $\mf{X}$ is the morphism $\rho_{\mr{crys},\mf{X}}\colon \mf{X}\to \mf{X}^\smallprism$ of formal stacks over $W$ obtained as the composition 
\be
\mf X\xrightarrow{i_\mf{X}} (\mf X_k/W)^\crys\isomto \phi_W^\ast(\mf{X}_k^\smallprism)\xrightarrow{F_{\mf{X}_k}} \mf X_k^\smallprism\to \mf X^\smallprism.
\ee
\end{defn}

The $2$-commutativity in Proposition \ref{prop:Psi-2-commutativity} implies the following.

\begin{prop}\label{prop:crystalline-realization-stacky-equiv}For any prismatic crystal $\mc{E}$ on $\mf{X}$, there exists a natural identification
\begin{equation*}
\rho_{\mr{crys},\mf{X}}^\ast(\mc{E})\simeq \mc{E}^\mr{crys}|_{\mf{X}_\mr{Zar}}.
\end{equation*}
\end{prop}

\begin{rem}\label{rem:Dcrys-in-coordinates} Suppose that $\mf{X}=\Spf(R)$ and that $\framew$ is a formal framing. Consider the following diagram where $\rho_{(R^{(\phi_\framew)},(p))}$ is as in \cite[Construction 3.10]{BhattLuriePrismatization}
\begin{equation*}
\begin{tikzcd}[sep=2.25em]
	& {R^\smallprism} \\
	{\Spf(R)} & {\Spf(R).}
	\arrow["{\rho_{\mr{crys},\mf{X}}}", from=2-1, to=1-2]
	\arrow["{\phi_\framew}"', from=2-1, to=2-2]
	\arrow["{\rho_{(R^{(\phi_\framew)},(p))}}"', from=2-2, to=1-2]
\end{tikzcd}
\end{equation*}
This diagram naturally $2$-commutes as for any $p$-nilpotent $R$-algebra $S$, both compositions result in the Cartier--Witt divisor $\bb{W}(S)\xrightarrow{p}\bb{W}(S)$ with the structure map corresponding to 
\begin{equation*}
R\to R/p\xrightarrow{F_{R/p}}R/p\to \bb{W}(S)/p.
\end{equation*}
This gives an isomorphism $\rho_{\mr{crys},\mf{X}}^\ast(\mc{E})\simeq (\phi^\ast\mc{E})(R^{(\phi_\framew)},(p))$ which combined with Proposition \ref{prop:crystalline-realization-stacky-equiv} recovers $\vartheta_{R,\framew}$ from Proposition \ref{prop: pris crys in char p}. 
\end{rem}

\subsection{The de Rham realization functor}\label{ss:de-Rham-realization} We now discuss the notion of a \emph{de Rham realization functor} $\bb{D}_\dR\colon \cat{Vect}^\varphi(\mf{X}_\smallprism)\to \cat{Vect}(\mf{X})$. We will take a stack-theoretic approach, but the reader should consult Proposition \ref{prop:de-Rham-realization-BK-relationship} for a more down-to-earth interpretation in coordinates. In the following there are no restrictions on $\mc{O}_K$.

\begin{defn}[{cf.\@ \cite[\S6.7]{GMM}}]\label{defn:de-rham-point} Let $\mf{Y}$ be a quasi-syntomic $p$-adic formal scheme. 
\begin{itemize}[leftmargin=.3cm]
\item The \emph{(relative) Nygaard de Rham point} over $\mf{Y}$ is the morphism of formal $\mathbb{A}^1/\mathbb{G}_m$-stacks $\rho_{\mr{dR},\mf{Y}}^\smallN\colon \mf{Y}\times (\mathbb{A}^1/\mathbb{G}_m)\to \mf{Y}^\smallN$ over $\mathbb{Z}_p$ associating to a morphism $f\colon \Spec(S)\to\mf{Y}$ and generalized Cartier divisor $\alpha\colon L\to \mc{O}_{\Spec(S)}$ the filtered Cartier--Witt divisor 
\begin{equation*}
F_\ast(\bb{W})\oplus V(L)^\sharp\xrightarrow{d=(V,\alpha^\sharp)}\bb{W}
\end{equation*}
where $(-)^\sharp$ is as in \cite[Variant 2.4.3]{BhattNotes}, and where the map $\Spec(\mr{Cone}(d))\to \mf{Y}$ is induced by precomposition with $f$ from the natural map
\begin{equation*}
S=\bb{W}(S)/V(\bb{W}(S))\to \mr{Cone}(d)(S).
\end{equation*}
\item The \emph{(relative) prismatic de Rham point} over $\mf{Y}$ is the morphism $\rho_{\mathrm{dR},\mf{Y}}^\smallprism \colon \mf{Y}\to \mf{Y}^\smallprism$ obtained by restricting $\rho_{\mr{dR},\mf{Y}}^\smallN$ to $\mf{Y}=\mf{Y}\times(\mathbb{G}_m/\mathbb{G}_m)$.
\item The \emph{(relative) syntomic de Rham point} over $\mf{Y}$ is $\rho_{\mathrm{dR},\mf{Y}}^\mr{syn}\defeq j_\smallN\circ \rho_{\mr{dR},\mf{Y}}^\smallN \colon \mf Y\times(\A^1/\Gm)\to \mf{Y}^\mr{syn}$.
\end{itemize}
\end{defn}

For a perfect complex $\mc{V}$ over $\mf{Y}^\smallprism$, let us write $\mc{V}^\mr{dR}\defeq (\rho_{\mr{dR}}^\smallprism)^\ast\mc{V}$. Given the identification of $\cat{Perf}(\mf{Y}\times(\mathbb{A}^1/\mathbb{G}_m))$ with the $\infty$-category $\cat{PerfF}(\mf{Y})$ of filtered perfect complexes over $\mf{Y}$ (see \cite[Proposition 2.2.6]{BhattNotes}), we observe that for any perfect complex $\mc{V}$ on $\mf{Y}^\smallN$ we have
\begin{equation}\label{eq:de-rham-filtered-de-rham-relationship}
(\rho_{\mr{dR},\mf{Y}}^\smallN)^\ast(\mc{V})\simeq \left(\mc{V}^\mr{dR},\mr{Fil}^\bullet(\mc{V}^\mr{dR})\right),
\end{equation}
for some filtration $\mr{Fil}^\bullet\left(\mc{V}^\mr{dR}\right)$ of $\mc{V}^\mr{dR}$.  
\begin{defn} The \emph{de Rham realization functor} (resp.\@ \emph{filtered de Rham realization functor})
\begin{equation*}
\bb{D}_\mr{dR}\colon \cat{Vect}^\varphi(\mf{X}_\smallprism)\to \cat{Vect}(\mf{X}),\qquad\bigg(\text{resp.}\,\,\bb{D}_\mr{dR}^+\colon\cat{Vect}^\varphi(\mf{X}_\smallprism)\to \cat{PerfF}(\mf{X})\bigg)
\end{equation*}
is defined to be $(\rho_{\mr{dR},\mf{X}}^\smallprism)^\ast$ (resp.\@ $(\rho_{\mr{dR},\mf{X}}^\mr{syn})^\ast\circ \Pi_\mf{X}$, where $\Pi_\mf{X}$ is as in \cite[Theorem 2.31]{GuoLi}).\footnote{Observe that, by definition, $\bb{D}_\mr{dR}(\mc{E},\varphi_\mc{E})=\mc{E}^\mr{dR}$.}
\end{defn}

As $\mr{R}_\mf{X}(\Pi_\mf{X}(\mc{E},\varphi_\mc{E}))\simeq (\mc{E},\varphi_\mc{E})$ (by \cite[Theorem 2.31 (1)]{GuoLi}), one sees from \eqref{eq:de-rham-filtered-de-rham-relationship} that one may write
\begin{equation*}
\bb{D}_\mr{dR}^+(\mc{E},\varphi_\mc{E})\simeq (\bb{D}_\mr{dR}(\mc{E},\varphi_\mc{E}),\Fil^\bullet_{\bb{D}_\mr{dR}}(\bb{D}_\mr{dR}(\mc{E},\varphi_\mc{E})))
\end{equation*}
for some filtration $\Fil^\bullet_{\bb{D}_\mr{dR}}(\bb{D}_\mr{dR}(\mc{E},\varphi_\mc{E}))$ on $\bb{D}_\mr{dR}(\mc{E},\varphi_\mc{E})$. 

The following construction will help to understand $\bb{D}_\mr{dR}^+$ more concretely. Before we describe it, let us recall (e.g., see \cite[Definition 1.23]{IKY2}) that for a prismatic $F$-crystal $(\mc{E},\varphi_\mc{E})$ on $\mf{X}$ the \emph{Nygaard filtration} on $\phi^\ast\mathcal{E}$ is given as follows:
\begin{equation*}
    \Fil^\bullet_\mr{Nyg}(\phi^\ast\mc{E})\defeq \left\{x\in \phi^\ast\mathcal{E}:\varphi_\mathcal{E}(x)\in \mc{I}_\smallprism^\bullet\cdot \mc{E}\right\},
\end{equation*}
which produces a filtered module over $(\mc{O}_\smallprism,\mc{I}_\smallprism)$. For any object $(A,I)$ of $\mf{X}_\smallprism$ there is an induced filtration $\Fil^\bullet_\mr{Nyg}(\phi_A^\ast\mc{E}(A,I))$, a filtered module over $(A,I)$.

\begin{rem}\label{rem:terminology-Nyg} 
Let us remark that the terminology `Nygaard filtration' is potentially confusing when $\mc{E}$ is cohomological, i.e., $\mc{E}=R^if_\ast\mc{O}_\smallprism$ for some smooth proper morphism $f\colon \mf{Y}\to\mf{X}$. Indeed, in this case one may equip $\mc{E}$ with the filtration given by pushing forward the Nygaard filtration on $\mc{O}_\smallprism$, which is also called the Nygaard filtration in \cite{BhattScholzePrisms}. These two filtrations need not agree in general, even after Frobenius pullback, but always agree rationally and do agree integrally in certain cases (see \cite[\S7.2]{LiLiuDerived}).
\end{rem}

\begin{construction}Let $R=R_0\otimes_W \mc{O}_K$, with formal framing $\framew$. Set 
\begin{equation*}
\bigg(\mc{V}_\framew(\mc{E},\varphi_\mc{E}),\Fil^\bullet(\mc{V}_\framew(\mc{E},\varphi_\mc{E}))\bigg)\defeq \bigg(\phi_\framew^\ast\mc{E}(\mf{S}_R^{(\phi_\framew)},(E))/E,\overline{\Fil}^\bullet_\mr{Nyg}\bigg)
\end{equation*}
where 
\begin{equation*}
\overline{\Fil}^\bullet_\mr{Nyg}\defeq \mr{im}\left(\mr{Fil}^\bullet_\mr{Nyg}(\phi_\framew^\ast\mc{E}(\mf{S}_R^{(\phi_\framew)},(E)))\to \phi_\framew^\ast\mc{E}(\mf{S}_R^{(\phi_\framew)},(E))/E\right).
\end{equation*}
This construction commutes with base change along \'etale maps of framed algebras $R\to R'$.
\end{construction}

\begin{prop}\label{prop:de-Rham-realization-BK-relationship} For an open subset $\Spf(R)\subseteq \mf{X}$ with $R=R_0\otimes_W\mc{O}_K$ and formal framing $\framew$, there is a canonical isomorphism in the filtered derived category of quasi-coherent sheaves,
\begin{equation}\label{eq:rees-isom}
\psi_{R,\framew}\colon (\mc{V}_\framew(\mc{E},\varphi_\mc{E}),\Fil^\bullet(\mc{V}_\framew(\mc{E},\varphi_\mc{E})))
\isomto \mathbb{D}_\mr{dR}^+(\mc{E},\varphi_\mc{E})|_{\Spf(R)}.
\end{equation}
\end{prop}

To prove Proposition \ref{prop:de-Rham-realization-BK-relationship} we first set up some notation. Fix $\mf{Y}$ to be a quasi-syntomic $p$-adic formal scheme. 
First, observe we can interpret the Nygaard filtration using the filtered prismatization $\mf Y^{\smallN}$ as follows. 
Consider, for any object $(A,I)$ of $\mf{Y}_\smallprism$ with structure map $s\colon \Spf(A/I)\to\mf{Y}$,
the map over $\A^1/\Gm$
\begin{equation*}
\rho_{(A,I)}^+\colon \wh{\mc R}(\Fil^\bullet_I(A))\to\mf{Y}^\smallN,
\end{equation*}
defined as in \cite[Remark 5.5.19]{BhattNotes}. Note that over $\Gm/\Gm$, this recovers the composition 
\begin{equation}\label{eq:restriction-to-Gm/Gm}
\Spf(A)\xrightarrow{\rho_{(A,I)}} \mf Y^\smallprism\xrightarrow{F_\mf{Y}}\mf Y^\smallprism,
\end{equation}
where $\rho_{(A,I)}$ is the map from \cite[Construction 3.10]{BhattLuriePrismatization}. 
In fact, for a $p$-nilpotent ring $S$ and map $f\colon \Spec(S)\to \Spf(A)\simeq \{t\ne 0\}\subseteq \wh{\mc{R}}(\Fil^\bullet_I(A))$, the underlying Cartier--Witt divisor of $\rho^+_{(A,I)}\circ f$ (i.e., the image under the map $\Z_p^\smallN\to \Z_p^\smallprism$ defined in \cite[Construction 5.3.3]{BhattNotes}) of this filtered Cartier--Witt divisor is given by $\phi_A^*I\otimes_A\W(S)\to \W(S)$; note that there is a natural identification of $\bb{W}(S)$-modules $I\otimes_AF_*\W(S)\isomto F_*(\phi_A^*I\otimes_A\W(S))$.

\begin{lem}\label{lem:commutative-diagram-for-filtered-prismatization}
The following diagram is $2$-commutative
\begin{equation*}
\begin{tikzcd}
	{\Spf(A/I)\times(\A^1/\Gm)} 
    & {\wh{\mc R}(\Fil^\bullet_I(A))} 
    \\ {\mf Y\times(\A^1/\Gm)} 
    & {\mf Y^\smallN,}
	\arrow["{\iota_{(A,I)}}", from=1-1, to=1-2]
	\arrow["{s\times\mr{id}}"', from=1-1, to=2-1]
	\arrow["\rho^\smallN_{\dR,\mf{Y}}"', from=2-1, to=2-2]
	\arrow["{\rho^+_{(A,I)}}", from=1-2, to=2-2]
\end{tikzcd}
\end{equation*}
where $\iota_{(A,I)}$ corresponds to the natural map of graded rings $\mr{Rees}(\Fil^\bullet_I(A))\to \mr{Rees}(\Fil^\bullet_\mr{triv}(A/I))$. 
\end{lem}
\begin{proof} 
Let $S$ be a $p$-nilpotent ring, $f\colon \Spec(S)\to \Spf(A/I)$ a morphism, and $\alpha\colon L\to \mc{O}_{\Spec(S)}$ a generalized Cartier divisor. 
We first construct an isomorphism between the two filtered Cartier--Witt divisors obtained by two composites from the diagram, i.e., an isomorphism of $\rho^+_{(A,I)}\circ \iota_{(A,I)}$ and $\rho^\smallN_{\mr{dR},\mf{Y}}\circ (s\times\mr{id})$ after applying the map $\mf Y^\smallN(S)\to \Z_p^\smallN(S)$. 

The object of $\Z_p^\smallN(S)$ obtained from $\rho_{\mr{dR},\mf{Y}}^\smallN\circ (s\times \id)$ is given by
\begin{equation}\label{eq:Cartier--Witt--divisor-for-other-composition}
V(L)^\sharp\oplus F_\ast(\bb{W})\xrightarrow{(\alpha^\sharp,V)}\bb{W}.
\end{equation}
On the other hand, first applying $\iota_{(A,I)}$ gives us the $S$-point of $\wh{\mc{R}}(\Fil^\bullet_I(A))$ corresponding to the map $A\to A/I\xrightarrow{f}S$ and the line bundle $L$ on $S$ with the factorization $I\otimes_A S\xrightarrow{0} L\xrightarrow{\alpha} S$. Applying the construction of $\rho^+_{(A,I)}$ then gives the filtered Cartier--Witt divisor obtained as follows. 
The underlying admissible $\bb W$-module scheme is obtained by pushing forward
\begin{equation*}
    0\to I\otimes_A\Ga^\sharp\to I\otimes_A\bb W\to I\otimes_AF_*\bb W\to 0
\end{equation*}
along the zero map $I\otimes_A\Ga^\sharp\to V(L)^\sharp$. 
So, the underlying admissible $\bb W$-module scheme is the trivial extension $V(L)^\sharp\oplus (I\otimes_AF_*\bb W)$. 
The structure of a filtered Cartier--Witt divisor $V(L)^\sharp\oplus(I\otimes_AF_*\W)\to \W$ is the one naturally induced from $\alpha^\sharp\colon V(L)^\sharp\to \Ga^\sharp$. 

Applying \cite[Proposition 3.6.6]{BhattLurieAbsolute} to the Cartier--Witt divisor $\rho_{(A,I)}\circ f=(I\otimes_A\W(S)\to \W(S))$ we obtain an isomorphism of Cartier--Witt divisors
\begin{equation*}
    (\phi_A^*I\otimes_A\W(S)\to \W(S))\xrightarrow{(\gamma,\id)}(\W(S)\xrightarrow{p}\W(S)),
\end{equation*}
where $\gamma$ is induced from the factorization $I\otimes_A\W(S)\xrightarrow{\beta} \W(S)\xrightarrow{V}\W(S)$; specifically, $\gamma$ is the linearization of the $F$-semi-linear map $\beta$. 
This then induces an isomorphism of filtered Cartier--Witt divisors
\begin{equation}\label{eq:BL-3.6.6}
    \left(V(L)^\sharp\oplus(I\otimes F_*\W)\to \W \right)
    \xrightarrow{(\id_{V(L)^\sharp}\oplus F_*(\gamma),\id_{\W})} 
    \left( V(L)^\sharp\oplus F_*\W\xrightarrow{(\alpha^\sharp,V)} \W \right).
\end{equation}
This is our desired identification of these two compositions in $\bb{Z}_p^\smallN(S)$.

We now show this isomorphism respects the $\mf Y$-structures (i.e., lifts to an isomorphism in $\mf{Y}^\smallN(S)$). From the natural map 
\begin{equation*}
    A/I\to R\Gamma(\Spec(S),\W/(I\otimes_A\W))
\end{equation*}
and the given map $s\colon \Spf(A/I)\to \mf Y$, the $\mf Y$-structure on the source of \eqref{eq:BL-3.6.6} is induced from the map of quasi-ideals $[(0,\beta),\id]$: 
\begin{equation*}
\begin{tikzcd}
	{I\otimes_A\W} 
    & {V(L)^\sharp\oplus F_*\W} 
    \\ {\W} 
    & {\W .}
	\arrow["{(0,\beta)}", from=1-1, to=1-2]
	\arrow["{}"', from=1-1, to=2-1]
	\arrow["\id", from=2-1, to=2-2]
	\arrow["{(\alpha^\sharp,V)}", from=1-2, to=2-2]
\end{tikzcd}
\end{equation*}
Similarly, the $\mf Y$-structure on the target of \eqref{eq:BL-3.6.6} is induced from 
\begin{equation*}
    \begin{tikzcd}
	{I\otimes_A\W} 
    & {V(L)^\sharp\oplus F_*(\phi^*I\otimes_A \W)} 
    \\ {\W} 
    & {\W .}
	\arrow["{(0,F)}", from=1-1, to=1-2]
	\arrow["{}"', from=1-1, to=2-1]
	\arrow["\id", from=2-1, to=2-2]
	\arrow["{}", from=1-2, to=2-2]
    \end{tikzcd}
\end{equation*}
Here, by construction, the two maps $\beta$ and $F_*(\gamma)\circ F$ coincide. 
Thus, the $\mf Y$-structures agree, and hence we obtain the desired commutativity. 
\end{proof}

\begin{prop}\label{prop:de-Rham-realization-BK-relationship-refined} For any open subset $\Spf(R)\subseteq \mf{X}$ with $R=R_0\otimes_W\mc{O}_K$ and with formal framing $\framew$, there exists an isomorphism in $\cat{QCoh}(\wh{\mc{R}}(\Fil^\bullet_E(\mf S_R)))$
\begin{equation*}
\left(\rho^+_{(\mf S_R^{(\phi_\framew)},(E)),\mf X}\right)^*\Pi_\mf{X}(\mc E)
\isomto
\bigoplus_{i\in\Z}\Fil^{-i}_\mr{Nyg}(\phi_\framew^*\mc{E}(\mf{S}_R^{(\phi_\framew)},(E))),
\end{equation*}
which is compatible with inclusions of affine open subsets of $\mf{X}$ with compatible framings.\footnote{Here we implicitly use the equivalence from \cite[Chapter I, §4.3, Proposition 7]{LvO} showing that we may explicitly identify $t$-torsion-free quasi-coherent sheaves on $\wh{\mc{R}}(\Fil^\bullet_E(\mf{S}_R^{(\phi_\framew)}))$ with filtered modules over $(\mf{S}_R^{(\phi_\framew)},\Fil^\bullet_E))$.}
\end{prop}
\begin{proof}
We may assume that $\mf{X}=\Spf(R)$. To define the isomorphism of underlying vector bundles
\begin{equation*}
\left(\rho^+_{(\mf{S}_R^{(\phi_\framew)},(E)),\mf X}\right)^\ast\Pi_\mf{X}(\mc{E})\isomto \phi_\framew^*\mc E(\mf S_R^{(\phi_\framew)},(E)),
\end{equation*}
it suffices to recall that $\rho_{(A,I)}^+$ restricted to the fiber of $\bb{G}_m/\bb{G}_m$ is precisely the map \eqref{eq:restriction-to-Gm/Gm}.

Consider the faithfully flat cover $R\to\wt{R}$ by a perfectoid ring $\wt{R}$ from \cite[Lemma 1.15]{IKY1}. 
Since the left-hand side comes from a filtered module (i.e., is $t$-torsion-free) by the construction of $\Pi_\mf X$, it suffices to verify the claim after base changing along the map $\wt{\alpha}_{\mr{inf},\framew^\flat}\colon \mf S_R\to\Ainf(\wt{R})$ from \cite[\S1.1.5]{IKY1}, i.e., it suffices to construct a natural isomorphism
\begin{equation*}
\left(\rho^+_{(\Ainf(\wt{R}),(\tilde\xi)),\mf X}\right)^*\Pi_\mf{X}(\mc E)
\isomto
\bigoplus_{i\in\Z}\Fil^{-i}_\mr{Nyg}(\phi^*\mc{E}(\Ainf(\wt{R})),(\wt{\xi})). 
\end{equation*}
But, the functor $\Pi_\mf{X}$ is precisely constructed so that this holds.
\end{proof}

\begin{proof}[Proof of Proposition \ref{prop:de-Rham-realization-BK-relationship}] By Lemma \ref{lem:commutative-diagram-for-filtered-prismatization}, we have 
\begin{equation*}
    \D_\dR^+(\mc E,\varphi_\mc E)|_{\Spf(R)}\isomto \iota^*_{(\mf S_R^{(\phi_\framew)},(E)),\mf{X}}(\rho^+_{(\mf S_R^{(\phi_\framew)},(E))})^\ast\Pi_\mf{X}(\mc E).
\end{equation*}
By Proposition \ref{prop:de-Rham-realization-BK-relationship-refined}, the right-hand side is given by 
\begin{equation*}
    \bigoplus_{i\in\Z}\Fil^{-i}_\mr{Nyg}(\phi_\framew^*\mc{E}(\mf{S}_R^{(\phi_\framew)},(E)))\otimes^L_{\mr{Rees}(\Fil^\bullet_E(\mf S_R))}R[t].
\end{equation*}
The proof then follows by Lemma \ref{lem:derived Rees tensor identification} below.
\end{proof}

\begin{lem}\label{lem:derived Rees tensor identification}
    Let $(A,(d))$ be a bounded prism, and $(M,\varphi_M)$ an object of $\cat{Vect}^\varphi(A,(d))$. Set $\Fil^\bullet\defeq \Fil^\bullet_\mr{Nyg}(\phi^\ast M)$ and $\ov{\Fil}^\bullet\defeq \ov{\Fil}^\bullet_\mr{Nyg}(\phi^\ast M/(d))$. Then, the natural map
    \begin{equation}\label{eq:derived-Rees-tensor-identification}
        \left(\bigoplus_{i\in\Z}\Fil^{-i}t^i\right)\otimes^L_{\mr{Rees}(\Fil^\bullet_d(A))}(A/d)[t]
        \to
        \bigoplus_{i\in\Z}\ov{\Fil}^{-i}t^i
    \end{equation}
    is an isomorphism of graded $(A/I)[t]$-modules. 
\end{lem}
\begin{proof}
    By the proof of \cite[Lemma 1.7]{IKY2}, the source of \eqref{eq:derived-Rees-tensor-identification} is quasi-isomorphic to the complex 
    \begin{equation*}
        \bigoplus_{i\in\Z}\Fil^{-(i-1)}t^i\xrightarrow{\cdot d}
        \bigoplus_{i\in\Z}\Fil^{-i}t^i,
    \end{equation*}
placed in degree $[-1,0]$. 
Since the multiplication-by-$d$ map is injective by the assumption that $M$ is projective, the map in question, on the $r^\text{th}$ graded piece is identified with the natural map 
\be
\Fil^r/d \Fil^{r-1}\to \ov{\Fil}^r, 
\ee
which is surjective by the definition of $\ov{\Fil}^\bullet$. We can show the injectivity as follows. Let $x=dy$ be an element of $\Fil^r\cap d \phi^*M$ where here $y$ belongs to $\phi^*M$. Then, by the definition of the Nygaard filtration, we get that $d\varphi_M(y)$ is in $d^rM$, and hence that $\varphi_M(y)$ belongs to $d^{r-1}M$, that is, $y$ is an element of $\Fil^{r-1}$. Thus, the above natural map is an isomorphism from where the claim follows. 
\end{proof}

\subsection{The crystalline-de Rham comparison} We now describe the crystalline--de Rham comparison, which identifies the vector bundles on $\mf{X}$ given by $\underline{\bb{D}}_\crys(\mc{E},\varphi_\mc{E})|_{\mf X_\Zar}$ and $\bb{D}_\dR(\mc{E},\varphi_\mc{E})$ for a prismatic $F$-crystal $(\mc{E},\varphi_\mc{E})$. Throughout we will assume $\mc{O}_K=W$.

\begin{thm}[Crystalline--de Rham comparison]\label{thm:crys-dR-comparison}
Let $\mf X$ be a base formal $W$-scheme.
\begin{enumerate}[leftmargin=.3in]
\item There is an identification between $\rho_{\crys,\mf{X}}$ and $\rho_{\dR,\mf{X}}^\smallprism$ as objects of $\mr{Map}(\mf{X},\mf{X}^\smallprism)$. 
\item For an object $\mc{E}$ of $\cat{Vect}(\mf{X}_\smallprism)$, there is a canonical isomorphism in $\cat{Vect}(\mf{X})$:
\begin{equation*}
    \iota_\mf{X}\colon \mc{E}^\mr{crys}\isomto\mc{E}^\mr{dR}.
\end{equation*}
In particular, for an object $(\mc{E},\varphi_\mc E)$ of $\cat{Vect}(\mf{X}_\smallprism)$, there is a canonical isomorphism
\begin{equation}\label{eq:crys dR isom stacky construction}
\iota_{\mf{X}}\colon \underline{\D}_\crys(\mc E,\varphi_\mc{E})|_{\mf{X}_\mr{Zar}} \isomto \D_{\dR}(\mc E,\varphi_\mc{E}).
\end{equation}
\end{enumerate}
\end{thm}
\begin{proof}
Assertion (2) is an immediate consequence of assertion (1) together with Proposition \ref{prop:crystalline-realization-stacky-equiv}. 

To prove assertion (1), let $B$ be a $p$-nilpotent ring and $s\colon\Spec(B)\to \mf X$ be a $B$-point of $\mf X$. 
Recall that the underlying Cartier--Witt divisors of both $\rho^\smallprism_{\dR,\mf X}(s)$ and $\rho_{\crys,\mf X}(s)$ are $\bb W(B)\xrightarrow{p} \bb W(B)$, and that their $\mf X$-structures are given respectively by 
\begin{eqnarray*}
\Spec(\bb W(B)/p)\xrightarrow{F}\bb \Spec(\bb{W}(B)/V)\xrightarrow{s}\mf X, 
\\
\Spec(\bb W(B)/p)\isomto\bb \Spec(\Ga^\dR(B))\to \Spec(B)\xrightarrow{s}\mf X,    
\end{eqnarray*}
where the map $\bb W(B)/p\isomfrom \Ga^\dR(B)$ is from \cite[Corollary 2.6.8]{BhattNotes} (note that we are not keeping track of the $W$-structure, and hence we can ignore the twist of $W$-module structure appearing in loc.\ cit.). If $C$ denotes the animated ring $\mr{Cone}(\Ga^\sharp(B)\oplus \bb W(B)\xrightarrow{(\mr{can},V)}\bb W(B))$, this isomorphism is defined to be the composite of the two isomorphisms 
\be
\Ga^\dR(B)\isomfrom C\isomto \bb W(B)/p.
\ee
Here, the left arrow is given by the restriction map $\bb W(B)\to B$ and the first projection onto $\Ga^\sharp$. 
The right arrow is given by the Frobenius map $F\colon \bb W(B)\to \bb W(B)$ and the second projection onto $\bb W(B)$ 
(see loc.\ cit.\ for details).
So, it suffices to show that the following diagram commutes:
\begin{equation}\label{eq:}
\begin{tikzcd}
	{B} & {\bb W(B)/V(\bb W(B))} 
    \\ {\Ga^\dR(B)} & {C.}
	\arrow["{\sim}"', from=1-2, to=1-1]
	\arrow["{}"', from=1-1, to=2-1]
	\arrow["{\sim}"', from=2-2, to=2-1]
	\arrow["{\mr{can}}", from=1-2, to=2-2]
\end{tikzcd}
\end{equation}
But, this is obtained from a commutative diagram of underlying quasi-ideals, and thus a commutative diagram of quotients. Thus, we obtain a functorial identification of the two objects $\rho_{\crys,\mf X}(s)$ and $\rho^\smallprism_{\dR,\mf X}(s)$, and hence an isomorphism $\rho_{\crys,\mf X}\isomto \rho^\smallprism_{\dR,\mf X}$. 
\end{proof}

We end by giving a more down-to-earth version of the crystalline--de Rham comparison over a $W$-algebra $R$ with formal framing $\framew$. To do so, we first make a construction using the Breuil prism $(S_R^{(\phi_\framew)},(p))$ as in \cite[\S1.1.5]{IKY1}. In what follows, we drop the decoration $\phi_\framew$ on Breuil(--Kisin) rings when only the ring structure is important. 

Let $\rho_{(S_R^{(\phi_\framew)},(p))}\colon \Spf(S_R)\to R^\smallprism$ be as in \cite[Construction 3.10]{BhattLuriePrismatization} and denote by $\mr{sp}_\dR$ the map on formal spectra induced by the composition $S_R\twoheadrightarrow S_R/\Fil^1_\mr{PD}\isomto R$ (i.e., the map $u\mapsto \pi$).

\begin{construction}\label{construction: crys dR isom with Breuil prism} Let $\mc{E}$ be an object of $\cat{Vect}(R_\smallprism)$. Using \cite[Diagram (1.1.2) and Lemma 1.14]{IKY1} and the crystal property, we obtain isomorphisms:
\begin{equation}\label{eq:crys-dr-triple-isom}
    \phi_\framew^*\mc E(R^{(\phi_\framew)},(p))\otimes_{R}S_R \isomto \mc E(S_R^{(\phi_\framew)},(p)) \isomfrom \phi_\framew^*\mc E(\mf S_R^{(\phi_\framew)},(E))\otimes_{\mf S_R}S_R.
\end{equation}
Reducing this isomorphism along $\mr{sp}_\mr{dR}^\ast\colon S_R\to R$, we obtain an isomorphism of $R$-modules
\begin{equation}\label{eq:pre crys-dr-isom}
    \phi_\framew^*\mc E(R^{(\phi_\framew)},(p))\isomto \phi_\framew^*\mc E(\mf S_R^{(\phi_\framew)},(E))/E.
\end{equation}
Since the left-hand (resp.\ right-hand) side is identified with $\mc{E}^\mr{crys}$ (resp.\ $\mc{E}^\mr{dR}$) by Proposition \ref{prop: pris crys in char p} (resp.\@ Proposition \ref{prop:de-Rham-realization-BK-relationship}), we obtain an isomorphism 
\begin{equation}\label{eq:crys-dr-isom}
    \mc{E}^\mr{crys}\isomto\mc{E}^\mr{dR}.
\end{equation}
\end{construction}

\begin{prop}\label{prop:crys-dR-comparison-using-prisms} Let $R$ be a base $W$-algebra with formal framing $\framew$. Set $\mf X=\Spf(R)$. 
\begin{enumerate}
\item The identification from Theorem \ref{thm:crys-dR-comparison} naturally factors as 
\be
\rho_{\crys,\mf X}\isomto \rho_{(S_R^{(\phi_\framew)},(p))}\circ \sp_\dR \isomfrom \rho_{\dR,\mf X}^\smallprism.
\ee
\item The isomorphism in \eqref{eq:pre crys-dr-isom} can be identified with the $\iota_\mf{X}$ from Theorem \ref{thm:crys-dR-comparison}. 
\end{enumerate}
\end{prop}

\begin{proof}
To prove assertion (1), first consider the following diagram:
\begin{equation*}
\begin{tikzcd}
    &{\Spf(\mf S_R)} &{\Spf(\mf S_R)} &{} 
    \\{\Spf(R)} &{\Spf(S_R)} &{} &{\mf X^\smallprism} 
    \\{} &{\Spf(R)} &{\Spf(R).} &{}    
	\arrow["{\mr{sp}_{\dR,\mf S}}", from=2-1, to=1-2]
	\arrow["{\mr{sp}_{\dR}}", from=2-1, to=2-2]
        \arrow["{\id}"', from=2-1, to=3-2]
        \arrow["{\phi_\framew}", from=1-2, to=1-3]
	\arrow["{}", from=2-2, to=1-2]
        \arrow["{\rho_{(S_R^{(\phi_\framew)},(p))}}", from=2-2, to=2-4]
        \arrow["{}", from=2-2, to=3-2]
        \arrow["{\phi_\framew}", from=3-2, to=3-3]
        \arrow["{\rho_{(\mf{S}_R^{(\phi_\framew)},(E))}}", from=1-3, to=2-4]
        \arrow["{\rho_{(R^{(\phi_\framew)},(p))}}"', from=3-3, to=2-4]
\end{tikzcd}
\end{equation*}
Here, the map $\sp_{\dR,\mf S}$ comes from the natural isomorphism $\mf{S}_R/E\isomto R$, the two vertical maps are given by the natural inclusions $\mf S_R\to S_R$ and $R\to S_R$, and the three arrows mapping to $\mf{X}^\smallprism$ are obtained as in \cite[Construction 3.10]{BhattLuriePrismatization}. The left two triangles are obviously commutative and the right two trapezoids are $2$-commutative with obvious identifications of the compositions. 

Let $\rho'_{\dR,\mf X}$ (resp.\ $\rho'_{\crys,\mf X}$) denote the upper (lower) composite maps $\Spf(R)\to \mf X^\smallprism$. 
By the commutativity of the above diagram, we obtain isomorphisms
\be
\rho'_{\crys,\mf X}\isomto\rho_{(S_R^{(\phi_\framew)},(p))}\circ \sp_\dR \isomfrom \rho'_{\dR,\mf X}.
\ee
By Proposition \ref{prop:de-Rham-realization-BK-relationship} (resp.\ Remark \ref{rem:Dcrys-in-coordinates}), we have an isomorphism $\rho'_{\dR,\mf X}\isomto\rho^\smallprism_{\dR,\mf X}$ (resp.\ $\rho'_{\crys,\mf X}\isomto\rho_{\crys,\mf X}$). 
Thus, we have obtained the diagram of isomorphisms
\begin{equation}\label{eq: crys-dR diagram of 2-morphisms}
\begin{tikzcd}
    {\rho_{\crys,\mf X}} & &{\rho^\smallprism_{\dR,\mf X}}
    \\{\rho'_{\crys,\mf X}} &{\rho_{(S_R^{(\phi_\framew)},(p))}\circ \sp_\dR} & {\rho'_{\dR,\mf X}.}
	\arrow["{}"', from=1-1, to=1-3]
	\arrow["{}"', from=1-1, to=2-1]
        \arrow["{}"', from=1-3, to=2-3]
        \arrow["{}", from=2-1, to=2-2]
        \arrow["{}", from=2-3, to=2-2]
\end{tikzcd}
\end{equation}
This diagram commutes as each isomorphism is defined by the identity map on $\bb W$ (and a canonical identification of its quasi-ideals), which proves assertion (1). 

To prove assertion (2), we consider the following commutative diagram induced by (\ref{eq: crys-dR diagram of 2-morphisms})
\begin{equation*}
\begin{tikzcd}
    {\mc{E}^\mr{crys}} & &{\mc{E}^\mr{dR}}
    \\{\phi_\framew^*\mc E(R^{(\phi_\framew)},(p))} &{\mc E(S_R^{(\phi_\framew)},(p))\otimes_{S_R,\mr{sp}^\ast_\dR }R} & {\phi_\framew^*\mc E(\mf S_R^{(\phi_\framew)},(E))/E.}
	\arrow["{}"', from=1-1, to=1-3]
	\arrow["{}"', from=1-1, to=2-1]
        \arrow["{}"', from=1-3, to=2-3]
        \arrow["{}", from=2-1, to=2-2]
        \arrow["{}", from=2-3, to=2-2]
\end{tikzcd}
\end{equation*}
The composition of the two lower horizontal isomorphisms is the isomorphism in (\ref{eq:pre crys-dr-isom}). Moreover, the left (resp.\ right) horizontal isomorphism is the identification from Proposition \ref{prop: pris crys in char p} (resp.\ Proposition \ref{prop:de-Rham-realization-BK-relationship}) by Remark \ref{rem:Dcrys-in-coordinates} (resp.\ by definition). 
This proves assertion (2). 
\end{proof}

\begin{rem}\label{rem:Kisin agreement} When $R=W$, the isomorphism in \eqref{eq:crys-dr-triple-isom} is compatible with that from \cite[Proposition 3.6]{GaoFLstr} (and after inverting $p$, that from \cite[Lemma 1.2.6]{KisinFCrystal}). More precisely, it is equal to the isomorphism $1\otimes s$ from \cite[Proposition 3.6]{GaoFLstr} (resp.\ the base change of the isomorphism $\xi$ from \cite[Lemma 1.2.6]{KisinFCrystal} along the map $\phi\colon \mc{O}\to S_R[\nicefrac{1}{p}]$, where $\mc{O}$ is as in loc.\@ cit). This follows from the uniqueness property discussed in loc.\@ cit. 
\end{rem}

\begin{rem}\label{rem:BO-isomorphisms} Our assumption that $K=K_0$ was necessary for Construction \ref{construction: crys dR isom with Breuil prism} so that the arrow labeled $(\ast)$ in \cite[Diagram (1.1.2)]{IKY1} was a morphism in $R_\smallprism$. But, using \cite[Lemma 1.14]{IKY1}, one may adjust this for arbitrary $\mathcal{O}_K$ giving an isomorphism of $R$-modules
\begin{equation*}
    (\phi^e)^\ast\mc{E}(\mf{S}_R,(E))/E\isomto (\phi^e)^\ast\mc{E}(\mf{S}_R,(E))/u.
\end{equation*}
In fact, such an isomorphism should hold, by the same method of proof, with $(\mc{E},\varphi_\mc{E})$ replaced by an object of $\cat{D}^\varphi_\perf(\mf{X}_\smallprism)$, where the quotients should now be considered in the derived sense. One interesting implication of this would be the existence of canonically matched lattices under \begin{equation*}
Rf_\ast(\Omega^\bullet_{\mf{X}/\mf{Y}})\otimes^L_{\mathcal{O}_K}K\isomto Rf_\ast(\mathcal{O}_{(\mf{X}_k/W)_\crys})|_{\mf{Y}_\Zar}\otimes^L_{W}K,
\end{equation*}
the isomorphism of Berthelot--Ogus (see \cite[Theorem 2.4]{BerthelotOgusFIsocrystals}), where $f\colon \mf{X}\to\mf{Y}$ is a smooth proper morphism of formal $\mathcal{O}_K$-schemes, where $\mf{Y}$ is smooth. This idea is explored in \cite{AbhinandanYoucis}.
\end{rem}

\subsection{Constancy of the Nygaard filtration over Breuil's ring} The diagram of stacks that appeared in the beginning of the proof of Proposition \ref{prop:crys-dR-comparison-using-prisms} admits a filtered upgrade, as explained below. In essence this result says that the Nygaard filtration is pulled back from the base (in an appropriate sense) over Breuil's ring $S_R$. This result will be useful in later sections. 

To this end, let $R$ be a base $W$-algebra with formal framing $w$. 
We may then consider the following diagram of stacks
\begin{equation}\label{eq:filtered dR point as F of HT point}
\begin{tikzcd}
	{\wh{\mc R}(\Fil^\bullet_\mr{PD}(S_R))} 
    & {\wh{\mc R}(\Fil^\bullet_E(\mf S_R))} 
    \\ {\Spf(R)\times(\A^1/\Gm)} 
    & {R^\smallN,}
	\arrow["{\alpha}", from=1-1, to=1-2]
	\arrow["{\pi}"', from=1-1, to=2-1]
	\arrow["\rho^\smallN_{\dR,\Spf(R)}"', from=2-1, to=2-2]
	\arrow["{\rho^+_{(\mf S,(E))}}", from=1-2, to=2-2]
\end{tikzcd}
\end{equation}
where $\Fil^\bullet_\mr{PD}(S_R)$ denotes the PD filtration on $S_R$ associated to the PD ideal $\ker(S_R\to R;u\mapsto p)$, the upper horizontal arrow $\alpha$ (resp.\ left vertical arrow $\pi$) is induced by the map of filtered rings $(\mf S_R,\Fil^\bullet_E)\to (S_R,\Fil^\bullet_\mr{PD})$ (resp.\ $(R,\Fil^\bullet_\mr{triv})\to (S_R,\Fil^\bullet_\mr{PD})$) with the underlying map $\mf S_R\to S_R$ (resp.\ $R\to S_R$) being the natural ones. 

\begin{rem}\label{rem:filtration-defn} Consider the isomorphism 
        \begin{equation*}
            \phi_\framew^*\mc E(R^{(\phi_\framew)},(p))\otimes_{R}S_R \isomto \phi_\framew^*\mc E(\mf S_R^{(\phi_\framew)},(E))\otimes_{\mf S_R}S_R
        \end{equation*}
from {Equation \eqref{eq:crys-dr-triple-isom}}. We may equip the source and target of this isomorphism with the following filtrations:
\begin{itemize}
\item We equip $\phi^*_w\mc E(R^{(\phi_w)},(p))=\underline{\D}_\crys(\mc E)(R)$ with the filtration denoted by $\Fil^\bullet_{\D_\dR}\D_\dR(\mc E,\varphi)$ in \S\ref{ss:de-Rham-realization}. We may then equip the source with the filtration obtained by the filtered scalar extension along $(R,\Fil^\bullet_\mr{triv})\to (S_R,\Fil^\bullet_\mr{PD})$ with $R\to S_R$ the natural inclusion. 
\item We equip the target with the filtration obtained from $\Fil^\bullet_\mr{Nyg}(\phi^*\mc E)$ by the filtered scalar extension $(\mf S_R,\Fil^\bullet_E)\to (S_R,\Fil^\bullet_\mr{PD})$ with $\mf S_R\to S_R$ the natural inclusion.
\end{itemize}
\end{rem}

\begin{prop}\label{filtered commutativity of Breuil prism diagram}
    The following hold. 
    \begin{enumerate}
        \item \emph{Diagram \eqref{eq:filtered dR point as F of HT point}} is $2$-commutative. 
        \item For an object $(\mc E,\varphi_\mc E)$ of $\cat{Vect}(\mf{X}_\smallprism)$, the isomorphism from \emph{Equation \eqref{eq:crys-dr-triple-isom}}
        \begin{equation*}
            \phi_\framew^*\mc E(R^{(\phi_\framew)},(p))\otimes_{R}S_R \isomto \phi_\framew^*\mc E(\mf S_R^{(\phi_\framew)},(E))\otimes_{\mf S_R}S_R
        \end{equation*}
        respects the filtrations from \emph{Remark \ref{rem:filtration-defn}}. 
    \end{enumerate}
\end{prop}

Recall that, in general, the valued points of the Rees stack viewed as a stack over $\A^1/\Gm$ are given as follows. 
Let $(A,\Fil^\bullet)$ be a filtered ring, $B$ a ring, and $(L,s)$ be an object of $(\A^1/\Gm)(B)$, i.e., a line bundle $L$ together with a $B$-linear map $s\colon L\to B$. 
The groupoid of maps $\Spec(B)\to \mc R(\Fil^\bullet)$ over $\A^1/\Gm$ (is discrete and) consists of homomorphisms of graded $\Z[t]$-algebras $\bigoplus_{i\in \mathbb Z}\Fil^i\cdot t^{-i}\to \bigoplus_{i\in \mathbb Z}L^{\otimes i}\cdot t^{-i}$, where the $\Z[t]$-algebra structure of the target is defined using $s\colon L\to B$.  

When the filtration is given by an invertible ideal $I$, namely when $\Fil^\bullet=I^\bullet\cap A$, 
such a homomorphism is equivalent to giving a ring homomorphism $A\to B$ and an $A$-linear map $I\to L$ such that the composite $I\to L\to B$ agrees with the composite $I\to A\to B$. 

\begin{proof}[Proof of Proposition \ref{filtered commutativity of Breuil prism diagram}]
    Note that this diagram naturally lives over $\A^1/\Gm$. Restricting along the open embedding $\Gm/\Gm\to \A^1/\Gm$, we recover the $2$-commutative diagram 
    \begin{equation}\label{Breuil prism diagram without filtration}
        \begin{tikzcd}
	{\Spf(S_R)} 
    & {\Spf(\mf S_R)} 
    \\ {\Spf(R)} 
    & {R^\smallprism.}
	\arrow["{}", from=1-1, to=1-2]
	\arrow["{}"', from=1-1, to=2-1]
	\arrow["\rho_{(R,(p))\circ\phi_w}"', from=2-1, to=2-2]
	\arrow["{\rho_{(\mf S,(E))}\circ\phi_w}", from=1-2, to=2-2]
\end{tikzcd}
    \end{equation}
    Thus, assertion (2) follows from assertion (1). We now prove assertion (1). 
    
    Fix $B$ a $p$-nilpotent ring, $s\colon L\to B$ a $B$-linear map with $L$ an invertible $B$-module, and a map $\Spec(B)\to \wh{\mc R}(\Fil^\bullet_\mr{PD}(S_R))$ over $\A^1/\Gm$. 
    Let $b=\bigoplus_{i\in \Z} b^i$ denote the corresponding homomorphism of graded $\Z[t]$-algebras $\bigoplus_{i\in \mathbb Z}\Fil^i_\mr{PD}(S_R)\cdot t^{-i}\to \bigoplus_{i\in \mathbb Z}L^{\otimes i}\cdot t^{-i}$. 
    
    We first describe $(\rho^\smallN_{\dR,\Spf(R)}\circ\pi)(b)$. 
    The object $\pi(b)$, when viewed as an $\bb{A}^1/\bb{G}_m$-equivariant map $\Spec(B)\to \Spf(R)\times(\A^1/\Gm)$, corresponds to the composite $R\to S_R\to B$. 
    Then, by definition, the object $(\rho^\smallN_{\dR,\Spf(R)}\circ\pi)(b)$ is given by the filtered Cartier--Witt divisor $(s^\sharp,V)\colon V(L)^\sharp\oplus F_*\W_B\to \W_B$ 
    on $B$ together with the natural $R$-structure. 
    
    On the other hand, $(\rho^+_{(\mf S,(E))}\circ\alpha)(b)$ may be computed as follows. 
    The object $\alpha(b)$ viewed as a map of stacks over $\A^1/\Gm$ corresponds to the pair $(b^0\circ\mr{nat}\colon \mf S_R\to B,b^1\circ\mr{nat}\colon E\mf S_R\to L)$. 
    Then the underlying filtered Cartier--Witt divisor of $(\rho^+_{(\mf S,(E))}\circ\alpha)(b)$ is given by the following construction. 
    Form a push-out diagram 
    \begin{equation*}
        \begin{tikzcd}
	{E\mf S_R\otimes_{\mf S_R}{(\Ga^\sharp)}_B} 
    & {E\mf S_R\otimes_{\mf S_R}\W_B} 
    \\ {V(L)^\sharp} 
    & {M,}
	\arrow["{}", from=1-1, to=1-2]
	\arrow["{(b^1\circ\mr{nat})^\sharp}"', from=1-1, to=2-1]
	\arrow[""', from=2-1, to=2-2]
	\arrow["", from=1-2, to=2-2]
\end{tikzcd}
    \end{equation*}
    in the category of $W$-module schemes on $\Spec(B)$ and define a map $M\to \W_B$ corresponding to $s^\sharp\colon V(L)^\sharp\to \Ga^\sharp$ and the natural map $E\mf S_R\otimes_{\mf S_R}\W_B\to \W_B$. 
    The $R$-structure is the natural one. 

    To identify $(\rho^\smallN_{\dR,\Spf(R)}\circ\pi)(b)$ and $(\rho^+_{(\mf S,(E))}\circ \alpha)(b)$, we first recall the description of the $2$-commutativity isomorphism of the diagram in (\ref{Breuil prism diagram without filtration}) in terms of Cartier--Witt divisors. 
    Namely, it is given by the multiplication-by-$p$ map $\W(B)\to \phi(E)\mf S_R\otimes_{\mf S_R}\W(B)$, which is an isomorphism as $B$ is an $S_R$-algebra and that $p$ and $\phi(E)$ generate the same ideal in $S_R$. 
    Note that this induces an isomorphism 
    \begin{equation*}
    F_*(p)\colon F_*\W_B \isomto F_*(\phi(E)\mf S_R\otimes_{\mf S_R}\W_B)=E\mf S_R\otimes_{\mf S_R} F_*\W_B
    \end{equation*}
    of $\W$-module schemes. 
    
    We may then construct a $\W$-linear map $\beta\colon E\mf S_R\otimes_{\mf S_R}\W_B\to V(L)^\sharp$. 
    As the source is a free $\W_B$-module scheme of rank one, giving such a map is equivalent to giving a $B$-valued point of $V(L)^\sharp$, via the basis $E\otimes 1$.  
    Moreover, as $V(L)^\sharp$ is defined to be the PD hull of the line bundle $V(L)=\Spec(\mr{Sym}^\bullet L^\vee)$ along the zero section, giving a $B$-valued point is equivalent to giving a sequence $(\beta^{[1]},\beta^{[2]},\beta^{[3]},\ldots)$ with each $\beta^{[n]}$ an element of $L^{\otimes n}$ satisfying $\beta^{[n]}\otimes\beta^{[m]}={m+n\choose n}\beta^{[m+n]}$ for any pair $(n,m)$. 
    For this, we choose $\beta^{[n]}$ to be the image of $E^{[n]}\defeq E^n/n!\in \Fil^n_\mr{PD}(S_R)$ along the map $b^n\colon \Fil^n_\mr{PD}(S_R)\to L^{\otimes n}$. 

    The map $\beta$ together with the isomorphism $F_*(p)^{-1}$ gives a natural morphism of short exact sequences of $\W$-module schemes
    \begin{equation*}
        \begin{tikzcd}
	        0 
            &   {V(L)^\sharp} 
            &   {M} 
            &   E\mf S_R\otimes_{\mf S_R}F_*\W_B
            & 0
            \\ 0
            &   V(L)^\sharp
            &   V(L)^\sharp\oplus F_*\W_B
            &   F_*\W_B
            & {0,}
	\arrow["{}", from=1-1, to=1-2]
    \arrow["{}", from=1-2, to=1-3]
    \arrow["{}", from=1-3, to=1-4]
    \arrow["{}", from=1-4, to=1-5]
	\arrow[""', from=2-1, to=2-2]
    \arrow["{}", from=2-2, to=2-3]
    \arrow["{}", from=2-3, to=2-4]
    \arrow["{}", from=2-4, to=2-5]
	\arrow["\id", from=1-2, to=2-2]
    \arrow["", from=1-3, to=2-3]
    \arrow["F_*(p)^{-1}", from=1-4, to=2-4]
\end{tikzcd}
    \end{equation*}
    hence assertion (1) follows. 
\end{proof}

As a consequence of assertion (2) of Proposition \ref{filtered commutativity of Breuil prism diagram}, we obtain the Griffiths transversality property for the Nygaard filtration with respect to the natural differential operator on evaluation at the Breuil prism of an lff prismatic $F$-crystal. 
Although we do not use it explicitly in the later sections, we record it below in a precise form as it seems interesting in its own right. 

\begin{construction} Let $R$ be a base $W$-algebra with framing $w$, and let
    $(\mc E,\varphi_\mc E)$ be an lff prismatic $F$-crystal on $R$. 
    Set $\mc M\defeq\mc E(S_R,(p))$ and let $\partial_u$ denote the differential operator corresponding to $\id\otimes \frac{d}{du}$ on $\phi_w^*\mc E(R,(p))\otimes_{R,\mr{nat}}S_R$ via the isomorphism 
    \begin{equation*}\phi_w^\ast\mc E(R,(p))\otimes_{R,\mr{nat}}S_R\isomto \mc M. 
    \end{equation*}
    Finally, denote by $\Fil^\bullet_{\mr{Nyg}\hyphen\mr{PD}}$ the filtration on $\mc M$ obtained as the filtered scalar extension 
    \begin{equation*} 
    (\phi_w^*\mc E(\mf S_R,(E)),\Fil^\bullet_\mr{Nyg})\otimes_{(\mf S_R,\Fil^\bullet_E)}(S_R,\Fil^\bullet_\mr{PD}).
    \end{equation*}
\end{construction}
\begin{cor}\label{cor: Griffiths transversality of PD-Nygaard filtration}
    The filtration $\Fil^\bullet_{\mr{Nyg}\hyphen\mr{PD}}$ on $\mc M$ satisfies Griffiths transversality with respect to the differential operator $\partial_u$, i.e., for each integer $i$, we have $\partial_u(\Fil^i_{\mr{Nyg}\hyphen\mr{PD}})\subset \Fil^{i-1}_{\mr{Nyg}\hyphen\mr{PD}}$. 
\end{cor}

\begin{rem}
    For $R=W$, the connection $\partial_u$ (more precisely $\partial_u\otimes du$) coincides with the connection $\nabla$ from \cite[Corollary 1.3.15]{KisinFCrystal} after scalar extension to $S_R[\nicefrac{1}{p}]$. In fact, both coincide with $-\lambda^{-1} N_\nabla \frac{du}{u}$, where $\lambda$ is the element defined in \cite[1.1.1]{KisinFCrystal} and $N_\nabla$ right above \cite[Lemma 1.2.1]{KisinFCrystal} (recall from Remark \ref{rem:Kisin agreement} that our crystalline-de Rham comparison isomorphism agrees with Kisin's isomorphism $\xi$). 
    In \cite[Lemma 1.2.12 (2)]{KisinFCrystal} (a strong form of)
    the Griffiths transversality of $\nabla$ is obtained. 
    Thus, Corollary \ref{cor: Griffiths transversality of PD-Nygaard filtration} can be thought of as an integral refinement of his observation. 
\end{rem}

\section{The functor \texorpdfstring{$\bb{D}_\mr{crys}$}{DDcrys}}\label{s:integral-Dcrys}

In this section we use the crystalline-de Rham comparison (see Theorem \ref{thm:crys-dR-comparison}) to give the definition of our integral analogue $\bb{D}_\mr{crys}$ and establish basic properties about it. Throughout we use notation and terminology from \hyperref[notation-and-terminology]{Notation and terminology} without comment.

\subsection{Various categories of filtered Frobenius crystals}\label{ss:category-of-filtered-f-crystals} 

\subsubsection{Naive filtered \texorpdfstring{$F$}{F}-crystals} A \emph{naive filtered $F$-crystal} on $\mf{X}$ is a triple $(\mc{F},\varphi_\mc{F},\Fil^\bullet_\mc{F})$ with $(\mc{F},\varphi_\mc{F})$ an object of $\cat{Vect}^\varphi(\mf{X}_{\crys})$ and $\Fil^\bullet_\mc{F}$ a filtration by $\mc{O}_\mf{X}$-submodules of $\mc{F}_\mf{X}\defeq \mc{F}|_{\mf{X}_\mr{Zar}}$. Morphisms of naive filtered $F$-crystals are morphisms of $F$-crystals respecting filtrations. Denote the category of naive filtered $F$-crystals by $\cat{VectNF}^\varphi(\mf{X}_\crys)$, given $\Z_p$-linear $\otimes$-structure where 
\begin{equation*}
    \Fil^k_{\mc{F}_1\otimes \mc{F}_2}=\sum_{i+j=k}\Fil^i_{\mc{F}_1}\otimes \Fil^j_{\mc{F}_2}.
\end{equation*}
It has an exact structure where a sequence is exact if its associated sequence in $\cat{Vect}^\varphi(\mf{X}_\crys)$ is exact, and for all $i$ the sequence of $\mc O_\mf X$-modules on $\mf{X}$ given by the $i^\text{th}$-graded piece is exact. We say that a naive filtered $F$-crystal has \emph{level} in $[0,a]$ if $\Fil^0_\mc{F}=\mc F_{\mf X}$ and $\Fil^{a+1}_\mc{F}=0$. 

\subsubsection{Filtered \texorpdfstring{$F$}{F}-crystals} We begin by explicating several categories of $F$-crystals with filtration that will be important in the sequel.

\begin{defn}\label{defn: FFCrys} For a base formal $W$-scheme $\mf{X}$, we say an object $(\mc{F},\varphi_\mc{F},\Fil^\bullet_\mc{F})$ of $\cat{VectNF}^\varphi(\mf{X}_\mr{crys})$: 
    \begin{enumerate}[leftmargin=.3in]
         \item is a \emph{weakly filtered $F$-crystal} if $\Fil^\bullet_\mc{F}$ satisfies Griffiths transversality with respect to $\nabla_\mc{F}$ after inverting $p$, and the filtration $\mathrm{Fil}^\bullet_\mc{F}[\nicefrac{1}{p}]\subset \mc{F}_\mf{X}[\nicefrac{1}{p}]$ is locally split,\footnote{Recall that a filtered $R$-module $\Fil^\bullet\subseteq M$ is \emph{locally split} if $\mr{Gr}^r(\Fil^\bullet)$ is a finite projective $R$-module for all $r$.}
         \item is \emph{graded $p$-torsion-free (gtf)} if $\mr{Gr}^r(\Fil^\bullet_\mc{F})$ is a $p$-torsion-free $\mc{O}_\mf{X}$-module for all $r$,
        \item is a \emph{filtered $F$-crystal} if $\Fil^\bullet_\mc{F}$ satisfies Griffiths transversality with respect to $\nabla_\mc{F}$, and the filtration $\mathrm{Fil}^\bullet_\mc{F}\subset \mc F_\mf X$ is locally split.
    \end{enumerate}
\end{defn}

\begin{rem}We use different terms than \cite{LoveringFCrystals}: a filtered $F$-crystal (resp. strongly divisible filtered $F$-crystal) here is a weak filtered $F$-crystal (resp. filtered $F$-crystal) there.
\end{rem}

We give notation to these full subcategories of $\cat{VectNF}^\varphi(\mf{X}_\crys)$ defined by these objects:
\begin{itemize}[leftmargin=.3in]
    \item $\cat{VectWF}^\varphi(\mf{X}_\crys)$ is the full subcategory consisting of weakly filtered $F$-crystals,
    \item $\cat{VectNF}^{\varphi,\mr{gtf}}(\mf{X}_\crys)$ is the full subcategory consisting of gtf naive filtered $F$-crystals,
    \item $\cat{VectF}^\varphi(\mf{X}_\crys)$ is the full subcategory consisting of filtered $F$-crystals.
\end{itemize}
We obtain further full subcategories of $\cat{VectNF}^\varphi(\mf{X}_\crys)$ by intersection, which are denoted by the obvious symbols. Furthermore, the subscript $[0,a]$ will denote the intersection with $\cat{VectNF}_{[0,a]}^\varphi(\mf{X}_\crys)$. 

Lastly, we note that if $\mf{X}\to\Spf(W)$ is smooth, there is an exact $\Z_p$-linear $\otimes$-functor $\cat{VectWF}^\varphi(\mf{X}_\crys)\to \cat{IsocF}^\varphi(\mf{X})$ sending $(\mc{F},\varphi_\mc{F},\Fil^\bullet_{\mc{F}})$ to the filtered $F$-isocrystal $(\mc{F},\varphi_\mc{F},\Fil^\bullet_{\mc{F}})[\nicefrac{1}{p}]$ which is defined to be $(\mc{F}[\nicefrac{1}{p}],\varphi_\mc{F},\Fil^\bullet_{\mc{F}}[\nicefrac{1}{p}])=(\mc{F}[\nicefrac{1}{p}],\varphi_\mc{F},\Fil^\bullet_F)$.

\subsubsection{The category \texorpdfstring{$\cat{VectF}^\nabla(\mf{X})$}{VectFNablaX}} We next discuss some results of Tsuji which allow one to understand the filtrations of some naive filtered $F$-crystals via the crystalline site.

Consider a pair $(\mc{E},\Fil^\bullet_\mc{E})$, where $\mc{E}$ is an object of $\cat{Vect}(\mf{X}_\crys)$, and $\Fil^\bullet_\mc{E}\subseteq \mc{E}$ is a filtration by locally quasi-coherent (see \stacks{07IS}) $\mc{O}_\crys$-submodules. We call $(\mc{E},\Fil^\bullet_\mc{E})$ a \emph{filtered crystal} if 
\begin{enumerate}[leftmargin=.3in]
    \item $(\mc{E},\Fil^\bullet_\mc{E})(x)$ is a filtered module over $(A,\Fil^\bullet_\mr{PD})$ for all $x=(i\colon A\twoheadrightarrow B,\gamma)$ in $(\mf{X}/W)_\crys$,
    \item for a morphism $x=(i\colon A\to B,\gamma)\to (i'\colon A'\to B',\gamma')=y $ in $(\mf{X}/W)_\crys$, the natural map $(\mc{E},\Fil^\bullet_\mc{E})(x)\otimes_{(A,\Fil^\bullet_\mr{PD}(A))}(A',\Fil^\bullet_\mr{PD}(A'))\to (\mc{E},\Fil^\bullet_\mc{E})(y)$ is an isomorphism.
\end{enumerate}
We will mostly be interested in the case when $(\mc{E},\Fil^\bullet_\mc{E})$ is locally free over $(\mc{O}_\crys,\Fil^\bullet_\mr{PD})$ in the sense of \cite[Definition 1.1]{IKY2}, and we denote the category of such by $\cat{VectF}(\mf{X}_\crys)$.

Now, define $\cat{VectF}^\nabla(\mf{X})$ to consist of triples $(\mc{V},\nabla_\mc{V},\Fil^\bullet_\mc{V})$ where $(\mc{V},\nabla_\mc{V})$ is an object of $\cat{Vect}^\nabla(\mf{X})$ and $\Fil^\bullet_\mc{V}$ is a locally split filtration  on $\mc{V}$ which satisfies Griffiths transversality.

\begin{prop}[{cf.\@ \cite[Theorem 29]{Tsu20}}]\label{prop:tsuji-filtered-equiv} The functor
\begin{equation*}
    \cat{VectF}(\mf{X}_\crys)\to \cat{VectF}^\nabla(\mf{X}),\qquad \mc{E}\mapsto (\mc{E}|_{\mf{X}_\mr{Zar}},\nabla_\mc{E},(\Fil^\bullet_\mc{E})|_{\mf{X}_\mr{Zar}}),
\end{equation*}
is an equivalence.
\end{prop}

\begin{rem} By definition, $\mc{E}|_{\mf{X}_\mr{Zar}}(\mf{U})=\mc{F}(\mr{id}\colon \mf{U}\to\mf{U},\gamma)$. But while $\mc{E}(\mr{id}\colon \mf{U}\to\mf{U},\gamma)$ agrees with $\mc{E}(\mf{U}_0\hookrightarrow\mf{U},\gamma)$, this is not true for $\Fil^\bullet_\mc{E}$ as it is a \emph{filtered} crystal and the PD filtrations for $(\mr{id}\colon \mf{U}\to\mf{U},\gamma)$ and $(\mf{U}_0\hookrightarrow\mf{U},\gamma)$. 
\end{rem}

Note that, by definition, we have
\begin{equation*}
    \cat{VectF}^\varphi(\mf X_\crys)=\cat{Vect}^\varphi(\mf X_\crys)\times_{\cat{Vect}^\nabla(\mf X)}\cat{VectF}^\nabla(\mf X). 
\end{equation*}
Thus, by Proposition \ref{prop:tsuji-filtered-equiv}, we obtain an equivalence of categories 
\be\label{eq:crystalline interpretation of filtered F crystals}
\cat{VectF}^\varphi(\mf X_\crys)\isomto \cat{Vect}^\varphi(\mf X_\crys)\times_{\cat{Vect}(\mf X_\crys)}\cat{VectF}(\mf X_\crys).
\ee
For this reason, for an object $(\mc{F},\Fil^\bullet_\mc{F})$ of $\cat{Vect}^\varphi(\mf{X}_\mr{crys})$ we may consider the evaluation of $\Fil^\bullet_\mc{F}$ at objects of $(\mf{X}/W)_\crys$, by which we mean the evaluation of the associated filtered crystal.

\subsubsection{Divisibility and strong divisibility} Let us fix a base formal $W$-algebra $R$ and a framing $w$, inducing the Frobenius lift $\phi_w$ on $R$. Then, for an $F$-crystal $(\mc{F},\varphi_\mc{F})$ on $(R/W)_\crys$ we define the \emph{Nygaard filtration} on $\phi_w^\ast\mc{F}(R)$, where we write $\mc{F}(R)$ as shorthand for $\mc{F}_{\Spf(R)}(R)$, by:
\begin{equation*}
    \mr{Fil}^r_{\mr{Nyg},\mc{F}}=\bigg\{x\in \phi_w^\ast \mc{F}(R): \varphi_\mc{F}(x)\in p^r \mc{F}(R)\bigg\},
\end{equation*}
leaving $\varphi_\mc{F}$, $R$, and $\phi_w$ implicit in the notation. 

\begin{rem} There is a small discrepancy between the Nygaard filtration here and the one in the sense of \cite[Definition 1.23]{IKY2} for the crystalline prism $(R,(p))$, which we briefly call the \emph{prismatic Nygaard filtration}. Namely, when $(\mc{F},\varphi_{\mc{F}})$ is considered as a prismatic $F$-crystal on $R/p$ as in \S\ref{par:crys-syn-equiv} there is an introduced Frobenius twist (see Proposition \ref{prop: pris crys in char p})
and so the prismatic Nygaard filtration actually induces a filtration on $\mc{F}(R)$, not $\phi_w^\ast\mc{F}(R)$. That said, the Nygaard filtration introduced here is precisely the pullback under $\phi_w$ of the prismatic Nygaard filtration. Or, equivalently, it agrees with the prismatic Nygaard filtration but with respect to $(R,(p))$ thought of as an object of $R_\smallprism$ via the structure map $R\xrightarrow{\phi_w}R\to R/p$, where the last map is the natural map.
\end{rem}





\begin{defn}\label{defn:strong-divisibility} With notation as above, we say a naive filtered $F$-crystal $(\mc{F},\varphi_\mc{F},\Fil^\bullet_\mc{F})$ is:
\begin{enumerate}
    \item \emph{divisible with respect to $w$} if for all $r\in\bb{Z}$ one has $\phi_w^\ast \Fil^r_\mc{F}\subseteq \Fil^r_{\mr{Nyg},\mc{F}}$,
    \item \emph{strongly divisible with respect to $w$} if for all $r\in\bb{Z}$ one has:
    \begin{equation*}\label{eq:defn-of-strong-div}
   \varphi_\mc{F}\left(\sum_{i\in\bb{Z}}p^{-i}\phi_w^\ast \Fil^i_\mc{F}\right)=\mc{F}(R).
    \end{equation*}
\end{enumerate}
We say that $(\mc{F},\varphi_\mc{F},\Fil^\bullet_\mc{F})$ is \emph{(strongly) divisible} if it is (strongly) divisible with respect to all choices of $w$.
\end{defn}

To help better understand these definitions, we consider the following two lattices in $\phi_w^\ast \mc{F}(R)[\nicefrac{1}{p}]$:
\begin{equation*}
    \displaystyle L_{\mr{Hdg}}\defeq \sum_{i\in\bb{Z}}p^{-i}\phi_w^\ast \Fil^i_\mc{F},\qquad L_\mr{Nyg}\defeq \varphi_\mc{F}^{-1}(\mc{F}(R))=\sum_{i\in\bb{Z}}p^{-i}\Fil^i_{\mr{Nyg},\mc{F}},
\end{equation*}
leaving the various dependencies implicit in the notation. From this perspective, it is clear that $(\mc{F},\varphi_\mc{F},\Fil^\bullet_\mc{F})$ being strongly divisible with respect to $w$ is equivalent to the equality $L_\mr{Hdg}=L_\mr{Nyg}$, and it is simple to check that divisibility with respect to $w$ is equivalent to $L_\mr{Hdg}\subseteq L_\mr{Nyg}$. 

Next, we observe that when our naive filtered $F$-crystal is gtf, strong divisibility can be phrased in terms of filtered tensor products. 

\begin{prop}\label{prop:strong-divisibilty-using-filtered-tensor-product} Suppose that $(\mc{F},\varphi_\mc{F},\Fil^\bullet_\mc{F})$ is a gtf naive filtered $F$-crystal. Then, it is strongly divisible with respect to $w$ if and only if:
\begin{equation*}
    (\phi_w^\ast \mc{F}(R),\phi_w^\ast \Fil^\bullet_\mc{F})\otimes_{(R,\Fil^\bullet_\mr{triv})}(R,\Fil^\bullet_p)=(\phi_w^\ast\mc{F}(R),\Fil^\bullet_{\mr{Nyg},\mc{F}}),
\end{equation*}
i.e., for all $r\in\bb{Z}$ one has the equality
\begin{equation*}
    \phi_w^\ast \Fil^r_{\mc{F},p}\defeq \sum_{i\leqslant r}p^{r-i}\phi_w^\ast \Fil^i_\mc{F}=\Fil^r_{\mr{Nyg},\mc{F}}.
\end{equation*}
\end{prop}
\begin{proof} Observe that if $N$ is a finite-rank free $R$-module and $L,L'\subseteq N[\nicefrac{1}{p}]$ are $R$-lattices, then $L=L'$ if and only if $N\cap p^rL=N\cap p^rL'$ for all $r$. In fact, one only needs to check for some $r$ such that $p^r L$ and $p^rL'$ are contained in $N$.

From this, it suffices to verify the two equalities
\begin{equation}\label{eq:gtf-equalities}
    \phi^\ast_w \mc{F}(R)\cap p^r L_\mr{Hdg}=\sum_{i\leqslant r}p^{r-i}\phi_w^\ast \Fil^i_\mc{F},\qquad \phi^\ast_w \mc{F}(R)\cap p^r L_\mr{Nyg}=\Fil^r_{\mr{Nyg},\mc{F}}.
\end{equation}
The second equality is by definition. To verify the first equality, observe that the right-hand side is clearly included in the left-hand side. For the other inclusion, an element of the left-hand side is of the form $x=\sum_{i\in\mathbb{Z}}p^{r-i}x_i$ where $x_i\in \phi_w^\ast\Fil^i_\mc{F}$. 
As $x\in\phi_w^\ast\mc{F}(R)$ and $\sum_{i\leqslant r}p^{r-i}x_i\in \phi_w^\ast\mc{F}(R)$ one also has that $b=\sum_{i>r}p^{r-i}x_i\in\phi_w^\ast\mc{F}(R)$. Let $m>r$ be maximal with respect to $x_m\ne 0$. Then, $p^{m-r}b\in \phi_w^\ast \Fil^r_\mc{F}$, and thus by our gtf assumption, $b\in \phi_w^\ast \Fil^r_\mc{F}$. Thus, $x=b+\sum_{i\leqslant r}p^{r-i} x_i$ is in $\sum_{i\leqslant r}p^{r-i}\phi_w^\ast\Fil^i_{\mc{F}}$ as desired.
\end{proof}

Finally, we globalize the above definitions.

\begin{defn} Let $\mf{X}$ be a base formal $W$-scheme. We say that a naive filtered $F$-crystal $(\mc{F},\varphi_\mc{F},\Fil^\bullet_\mc{F})$ is \emph{\emph{(}strongly\emph{)} divisible} if it is after restriction to every affine open $\Spf(R)\subseteq \mf{X}$.
\end{defn}

We write $\cat{VectNF}^{\varphi,\mr{sd}}(\mf{X}_\mr{crys})$ for the category of strongly divisible naive filtered $F$-crystals. More generally, for any of the previously defined subcategories of naive filtered $F$-crystals, we use a superscript $\mr{sd}$ to denote those which are additionally {strongly divisible}.

\subsubsection{PD-divisibility and strong PD-divisibility} In practice strong divisibility is too strong, and so we now aim to weaken this condition. We continue to fix a base formal $W$-algebra $R$ and a framing $w$, inducing the Frobenius lift $\phi_w$ on $R$.

\begin{nota} For an integer $r\geqslant 0$, write $p^{[r]}\defeq \tfrac{p^r}{r!}$, and write $\langle r\rangle$ for the unique integer such that \begin{equation*}
    p^{\langle r\rangle}\Z_p=\bigcup_{i\geqslant r}p^{[i]}\Z_p,
\end{equation*}
i.e., $\langle r\rangle= \min \{v_p(p^{[i]}):i\geqslant r\}$. Finally, we can consider the filtered ring $(R,\Fil^\bullet_\mr{PD})$ where
\begin{equation*}
    \Fil^\bullet_\mr{PD}=\begin{cases}p^{\langle r\rangle}R & \mbox{if} \quad r\geqslant 0\\ R & \mbox{if}\quad r<0.\end{cases}
\end{equation*}
\end{nota}

\begin{defn}\label{defn:strong-PD-divisibility} We say a naive filtered $F$-crystal $(\mc{F},\varphi_\mc{F},\Fil^\bullet_\mc{F})$ of $\cat{VectNF}^\varphi(R_\crys)$ is:
\begin{enumerate}
    \item \emph{PD-divisible with respect to $w$} if for all $r\in\bb{Z}$ one has 
    the inclusion
    \begin{equation*}
    \phi_w^\ast\Fil^r_\mc{F}\subseteq\sum_{i\leqslant r}p^{[r-i]}\Fil^i_{\mr{Nyg},\mc{F}}=: \mr{Fil}^r_{\mr{Nyg},\mc{F},\mr{PD}}.
\end{equation*}
    \item \emph{strongly PD-divisible with respect to $w$} if it is PD-divisible with respect to $w$ and if, for all $r\in\mathbb{Z}$, the above inclusion induces the equality 
    \begin{equation*}
    \phi_w^\ast \Fil^r_{\mc{F},\mr{PD}}\defeq\sum_{i\leqslant r}p^{[r-i]}\phi_w^\ast \Fil^i_\mc{F}= \mr{Fil}^r_{\mr{Nyg},\mc{F},\mr{PD}}.
\end{equation*}
\end{enumerate}
We say that $(\mc{F},\varphi_\mc{F},\Fil^\bullet_\mc{F})$ is \emph{(strongly) PD-divisible} if it is (strongly) PD-divisible with respect to all choices of $w$.
\end{defn}

We write $\cat{VectNF}^{\varphi,\pddiv}(\mf{X}_\mr{crys})$ for the category of strongly PD-divisible naive filtered $F$-crystals on $\mf{X}$. More generally, for any other subcategory of naive filtered $F$-crystals, we use the superscript $\pddiv$ to denote the subcategory consisting of strongly PD-divisible objects.

We first observe that (strong) PD-divisibility can be phrased in terms of filtered tensor products. The following result is simple, and left to the reader.

\begin{lem}\label{lem:PD-div-in-terms-of-filtered-tensor-products} A naive filtered $F$-crystal $(\mc{F},\varphi_\mc{F},\Fil^\bullet_\mc{F})$ is PD-divisible \emph{(}resp.\@ strongly PD-divisible\emph{)} with respect to $w$ if and only if the identity map induces a morphism \emph{(}resp.\@ isomorphism\emph{)} from $(\phi_w^\ast \mc{F}(R),\phi_w^\ast \Fil^\bullet_\mc{F})$ to $(\phi_w^\ast\mc{F}(R),\Fil^\bullet_{\mr{Nyg},\mc{F}})$ after filtered base change along $(R,\Fil^\bullet_\mr{triv})\to (R,\Fil^\bullet_\mr{PD})$.
\end{lem}

\begin{rem}\label{rem:PD-equals-normal-in-FL-range} When $(\mc{F},\varphi_\mc{F},\Fil^\bullet_\mc{F})$ belongs to $\cat{VectNF}^{\varphi,\mr{gtf}}_{[0,p-1]}(\mf{X}_\mr{crys})$ (or more generally in $\cat{VectNF}_{[a,b]}^{\varphi,\mr{gtf}}(\mf{X}_\crys)$ with $b-a\leqslant p-1$), PD-divisibility (resp.\ strong PD-divisibility) is equivalent to divisibility (resp.\ strong divisibility). As we only require the implication that strong PD-divisibility implies strong divisibility in this range, we content ourselves with an explanation of this. 

To do this, we observe that, by definition of strong PD-divisibility, one has that $\phi_w^\ast \Fil^{p-1}_{\mc{F},\mr{PD}}=\Fil^{p-1}_{\mr{Nyg},\mc{F},\mr{PD}}$. Thus, if we can show that
\begin{equation*}
p^{p-1}L_\mr{Hdg}=\phi_w^\ast \mr{Fil}^{p-1}_{\mc{F},p}=\phi_w^\ast \mr{Fil}^{p-1}_{\mc{F},\mr{PD}},
\qquad 
p^{p-1}L_\mr{Nyg}=\mr{Fil}^{p-1}_{\mr{Nyg},\mc{F}}=\mr{Fil}^{p-1}_{\mr{Nyg},\mc{F},\mr{PD}},
\end{equation*}
then we deduce that $p^{p-1}L_\mr{Hdg}=p^{p-1}L_\mr{Nyg}$ and thus $L_\mr{Hdg}=L_\mr{Nyg}$, and thus strong divisibility would follow.

We first show that $p^{p-1}L_\mr{Hdg}=\phi_w^\ast\Fil^{p-1}_{\mc{F},p}$. As $(\mc{F},\varphi_\mc{F},\Fil^\bullet_\mc{F})$ is gtf, we know from Equation \eqref{eq:gtf-equalities}, that $\phi_w^\ast \Fil^{p-1}_{\mc{F},p}=p^{p-1}L_\mr{Hdg}\cap \phi_w^\ast\mc{F}(R)$. That said, as $\Fil^i_\mc{F}=0$ for $i>p-1$, the intersection on the right-hand side of this equality is superfluous, and so we arrive at our desired equality.

To show that $\phi_w^\ast \Fil^{p-1}_{\mc{F},p}=\phi_w^\ast \Fil^{p-1}_{\mc{F},\mr{PD}}$ means showing that
\begin{equation*}
\sum_{i\leqslant p-1}p^{p-1-i}\phi^\ast_w \Fil^i_\mc{F}=\sum_{i\leqslant p-1}p^{[p-1-i]}\phi_w^\ast\Fil^i_\mc{F}.
\end{equation*}
Note that as $\Fil^i_\mc{F}=\mc{F}(R)$ for $i\leqslant 0$, the fact that 
\begin{equation*}
\sum_{k\geqslant 0}p^{p-1+k}\phi_w^\ast \mc{F}(R)\subseteq p^{p-1}\phi_w^\ast\mc{F}(R),\qquad \sum_{k\geqslant 0}p^{[p-1+k]}\phi_w^\ast \mc{F}(R)\subseteq p^{[p-1]}\phi_w^\ast\mc{F}(R),
\end{equation*}
implies that we must only show
\begin{equation*}
\sum_{i=0}^{p-1}p^{p-1-i}\phi_w^\ast\mc{F}(R)=\sum_{i=0}^{p-1}p^{[p-1-i]}\phi_w^\ast\mc{F}(R).
\end{equation*}
But, this is clear as $v_p(p^{p-1-i})=v_p(p^{[p-1-i]})$ for $i\in[0,p-1]$. A similar argument gives the equality $\Fil^{p-1}_{\mr{Nyg},\mc{F}}=\Fil^{p-1}_{\mr{Nyg},\mc{F},\mr{PD}}$.

Finally, we show that $p^{p-1}L_\mr{Nyg}=\Fil^{p-1}_{\mr{Nyg},\mc{F}}$. 
By Equation \eqref{eq:gtf-equalities}, it suffices to show that $p^p L_\mathrm{Nyg} \subset p\phi_w^*\mc{F}(R)$, which follows from $p^p L_\mathrm{Nyg} \cap \phi_w^*\mc{F}(R) \subset p\phi_w^*\mc{F}(R)$ because if there is $x \in p^p L_\mathrm{Nyg} \setminus p\phi_w^*\mc{F}(R)$ then there is $n \geq 0$ such that $p^nx \in (p^p L_\mathrm{Nyg} \cap \phi_w^*\mc{F}(R)) \setminus p\phi_w^*\mc{F}(R)$. Then the claim follows from 
$p^p L_\mathrm{Nyg} \cap \phi_w^*\mc{F}(R)=\Fil^{p}_{\mr{Nyg},\mc{F}} \subset \Fil^{p}_{\mr{Nyg},\mc{F},\mr{PD}}=\Fil^{p}_{\mc{F},\mr{PD}} \subset p\phi_w^*\mc{F}(R)$, where we use Equation \eqref{eq:gtf-equalities}, the assumption of strong PD-divisibility and $\mathrm{Fil}^i_\mc{F}=0$ for $i>p-1$. 
\end{rem}

Next we observe that, in fact, (strong) PD-divisibility with respect to $w$ actually implies (strong) PD-divisibility if we assume that $(\mc{F},\varphi_\mc{F},\Fil^\bullet)$ is a filtered $F$-crystal.

\begin{prop}\label{prop:PD-div-Frob-invariant} Let $(\mc{F},\varphi_\mc{F},\Fil^\bullet_\mc{F})$ be a filtered $F$-crystal. Then, if $(\mc{F},\varphi_\mc{F},\Fil^\bullet_\mc{F})$ is \emph{(}strongly\emph{)} PD-divisible with respect to some $w$, then it is \emph{(}strongly\emph{)} PD-divisible.
\end{prop}
\begin{proof}This follows from Lemma \ref{lem:PD-div-in-terms-of-filtered-tensor-products} and Proposition \ref{prop:tsuji-filtered-equiv}. Indeed, the filtered module
\begin{equation*}
    (\phi_w^\ast \mc{F}(R),\phi_w^\ast \Fil^\bullet_\mc{F})\otimes_{(R,\Fil^\bullet_\mr{triv})} (R,\Fil^\bullet_\mr{PD}),
\end{equation*}
is equal to the evaluation of $\Fil^\bullet_\mc F$ (through the equivalence in Proposition \ref{prop:tsuji-filtered-equiv}) at the PD-thickening $R\to R/p$ with the structural map being the composite $R\xrightarrow{\mr{nat}} R/p\xrightarrow{\mr{Frob}}R/p$. In particular, it is independent of the choice of $w$. 
\end{proof}

\subsubsection{Comparison to the category of Faltings}

Fix a naive filtered $F$-crystal $(\mc{F},\varphi_\mc{F},\Fil^\bullet_\mc{F})$ on $\Spf(R)$, where $R$ is a base $W$-algebra with formal framing $\framew$. Consider the following module as in \cite[II.c), p.\@ 30]{Faltings89}:
    \begin{equation}\label{eq:def-of-Faltings-tilde}
    \wt{\mc{F}}_\mf{X}\defeq \colim\bigg(\cdots \xrightarrow{}\Fil^{r+1}_\mc F\xleftarrow{\cdot p} \Fil^{r+1}_\mc F \xrightarrow{} \Fil^r_\mc F \xleftarrow{\cdot p}\Fil^r_\mc F \xrightarrow{}\Fil^{r-1}_\mc F\xleftarrow{}\cdots\bigg).
    \end{equation}
The maps $\Fil^r_\mc F\to \mc F_\mf X[\nicefrac{1}{p}]$ sending $x$ to $p^{-r}x$ induce a natural map $\wt{\mc F}_\mf X\to \mc F_\mf X[\nicefrac{1}{p}]$ whose image is the sum $\sum_{r\in \Z}p^{-r}\Fil^r_\mc F$. We then have the following \emph{Faltings morphism}
\begin{equation}\label{eq: Faltings morphism}
    \phi_\framew^\ast \wt{\mc{F}}_\mf{X}\to \phi_\framew^\ast\mc{F}_\mf{X}[\nicefrac{1}{p}]\xrightarrow{\varphi_\mc{F}}\mc{F}_\mf{X}[\nicefrac{1}{p}].
\end{equation}
Observe that if $(\mc{F},\varphi_\mc{F},\Fil^\bullet_\mc{F})$ is gtf then the map $\wt{\mc{F}}_\mf{X}\to \mc{F}_\mf{X}[\nicefrac{1}{p}]$ is injective. Thus, as $\phi_\framew$ is flat, we see that if $(\mc{F},\varphi_\mc{F},\Fil^\bullet_\mc{F})$ is gtf and strongly divisible, then the Faltings morphism induces an isomorphism $\phi_\framew^\ast \wt{\mc{F}}_\mf{X} \isomto \mc{F}_\mf{X}$.

Observe that  
\be
    \wt{\mc{F}}_\mf{X}/p\cong \colim\bigg(\cdots \xrightarrow{}\Fil^{r+1}_\mc F/p\xleftarrow{0} \Fil^{r+1}_\mc F/p \xrightarrow{} \Fil^r_\mc F/p \xleftarrow{0}\Fil^r_\mc F/p \xrightarrow{}\Fil^{r-1}_\mc F/p\xleftarrow{}\cdots\bigg),
    \ee
and so is isomorphic to $\bigoplus_r \Gr^r(\Fil_\mc{F}^\bullet)/p$. Leveraging this, we show the following.

\begin{prop}[{cf.\@ \cite[Theorem 2.1]{Faltings89}, \cite[Theorem 2.9]{LiuMoonPatel}}]\label{prop:ptf-implies-loc-free} Let $(\mc{F},\varphi_\mc{F},\Fil^\bullet_\mc{F})$ be a gtf strongly divisible naive filtered $F$-crystal. Then, the filtration $\Fil^\bullet_\mc{F}\subseteq \mc{F}_\mf{X}$ is locally split. If $(\mc{F},\varphi_\mc{F},\Fil^\bullet_\mc{F})$ is further a weakly filtered $F$-crystal, then it is a strongly divisible filtered $F$-crystal.
\end{prop}
\begin{proof} We may assume that $\mf{X}=\Spf(R)$, where $R$ is a base $W$-algebra with formal framing $\framew$.

Let us first verify that the filtration $\Fil^\bullet_\mc{F}$ is locally split. In other words, we must show that $\mr{Gr}^r(\Fil^\bullet_\mc{F})$ is a locally free $R$-module for all $r$. As $(\mc{F},\varphi_\mc{F},\Fil^\bullet_\mc{F})$ is gtf and strongly divisible, the Faltings morphism is an isomorphism. This implies that $\phi_\framew^\ast\wt{\mc{F}}(R)$ is a projective $R$-module. But, as observed above, this implies that $\phi_\framew^\ast\Gr^r(\Fil^r_\mc{F}(R))/p$ is a projective $R/p$-module for all $r$, and since $\phi_\framew$ is faithfully flat this implies that $\mr{Gr}^r(\Fil^\bullet_\mc{F}(R))/p$ is a locally free $R/p$-module for all $r$. This implies that $\mr{Gr}^r(\Fil^\bullet_\mc{F}(R))$ is a projective $R$-module by the following simple lemma.

\begin{lem} Let $A$ be a Noetherian $p$-adically complete ring and $Q$ a finitely generated $p$-torsion-free $A$-module such that $Q/p$ is a projective $A/p$-module. Then, $Q$ is a projective $A$-module.
\end{lem}
\begin{proof}Take a short exact sequence
$0\to K\xrightarrow{\iota} A^n\to Q\to 0$.
As $Q$ is $p$-torsion-free, reducing modulo $p$ gives an exact sequence
$0\to K/p\xrightarrow{\overline{\iota}} (A/p)^n\to Q/p\to 0$.
As $Q/p$ is projective, there exists a retraction $\overline{\rho}\colon (A/p)^n\to K/p$ to $\overline{\iota}$. Consider the composition $A^n\to (A/p)^n\to K/p$, and lift it to $\rho\colon A^n\to K$. Note that $\rho\circ\iota$ is the identity modulo $p$. As $K$ is $p$-adically complete, this implies that $\rho\circ\iota$ is an automorphism, with inverse $\sum_{k\geqslant 0}(-1)^k((\rho\circ \iota)-\mathrm{id})^k$. So then, $(\rho\circ \iota)^{-1}\circ \rho$ is a retraction to $\iota$. Thus, $Q$ is a direct summand of a free module, so projective.
\end{proof}

Suppose further that $(\mc{F},\varphi_\mc{F},\Fil^\bullet_\mc{F})$ is a weakly filtered $F$-crystal. To show that it is a strongly divisible $F$-crystal, it remains to show that $\Fil^\bullet_\mc{F}$ satisfies Griffiths transversality with respect to $\nabla_\mc{F}$. But, let us observe that by assumption, for each $r$ we have that
\begin{equation*}
\nabla_\mc{F}(\Fil^r_\mc{F})\subseteq (\mc{F}_\mf{X}\otimes_{\mc{O}_\mf{X}}\Omega^1_{\mf{X}/W})\cap (\Fil^{r-1}_\mc{F}\otimes_{\mathcal{O}_{\mf{X}}}\Omega^1_{\mf{X}/W})[\nicefrac{1}{p}].
\end{equation*}
But, the right-hand is just $\Fil^{r-1}_\mc{F}\otimes_{\mathcal{O}_{\mf{X}}}\Omega^1_{\mf{X}/W}$. Indeed, this follows from the observation that
\begin{equation*}
(\mc{F}_\mf{X}\otimes_{\mc{O}_\mf{X}}\Omega^1_{\mf{X}/W})/(\Fil^{r-1}_\mc{F}\otimes_{\mathcal{O}_{\mf{X}}}\Omega^1_{\mf{X}/W})\cong (\mc{F}_\mf{X}/\Fil^{r-1}_\mc{F})\otimes_{\mathcal{O}_{\mf{X}}}\Omega^1_{\mf{X}/W}
\end{equation*}
is $p$-torsion-free, being the tensor product of two vector bundles on $\mf{X}$. 
\end{proof}

We may now explain the relationship between the subcategories of naive filtered $F$-crystals considered above, and the categories $\mb{MF}^\nabla_{[0,a]}(R)$ of \emph{Fontaine--Laffaille modules} as in \cite[II.d)]{Faltings89}, when $p \neq 2$ and $0\leqslant a\leqslant p-2$.\footnote{In loc.\@ cit.\@ one considers categories $\mb{MF}^\nabla_{[0,a]}(R,\Phi)$ for some Frobenius lift $\Phi$ on $R$. But, as $a$ is in $[0,p-2]$, this category is independent of $\Phi$ by \cite[Theorem 2.3]{Faltings89}; alternatively combine Remark \ref{rem:PD-equals-normal-in-FL-range} and Proposition \ref{prop:PD-div-Frob-invariant}.}

\begin{prop}\label{prop:filtered-F-crystals} Assume that $p \neq 2$. Let $R$ be a base $W$-algebra. Then, for any $a$ in $[0,p-2]$, the functor
\begin{equation*}
    \cat{VectF}^{\varphi,\mr{sd}}_{[0,a]}(R_\crys)\to \cat{MF}^\nabla_{[0,a]}(R),\qquad (\mc{F},\varphi_\mc{F},\Fil^\bullet_\mc{F})\mapsto (\mc{F}(\id \colon R \to R),\varphi_\mc{F},\Fil^\bullet_\mc{F},\nabla_\mc{F}),
\end{equation*}
is an equivalence of categories.
\end{prop}
\begin{proof}
This follows from Proposition \ref{prop:tsuji-filtered-equiv} by adding Frobenius structures to both sides and checking that the notions of divisibility match. 
\end{proof}

\subsection{The functor \texorpdfstring{$\bb{D}_\crys$}{DDcrys}}\label{ss:dcrys-definition}

We now apply the crystalline-de Rham comparison theorem to define our integral analogue $\bb{D}_\crys$ of the functor $D_\crys$ and show that it really forms a lattice in $D_\mr{crys}$. Throughout we assume that $\mc{O}_K=W$.

\subsubsection{The definition} We now come to the definition of the naive filtered $F$-crystal $\bb{D}_\crys(\mc{E},\varphi_\mc{E})$.

\begin{defn} Let $(\mathcal E,\varphi_\mathcal E)$ be an object of $\cat{Vect}^\varphi(\mf{X}_\smallprism)$. Define $\bb{D}_\crys(\mc{E},\varphi_\mc{E})$ to have underlying $F$-crystal $\underline{\bb{D}}_\crys(\mc{E},\varphi_\mc{E})$ and have the filtration induced by $\iota_\mf{X}$ from Theorem \ref{thm:crys-dR-comparison}:
\begin{equation*}
    \underline{\bb{D}}_\crys(\mc{E},\varphi_\mc{E})\supseteq \Fil^\bullet_{\bb{D}_\crys}(\mc{E},\varphi_\mc{E})\defeq \iota_\mf{X}^{-1}\bigg(\Fil^\bullet_{\bb{D}_\mr{dR}}(\bb{D}_\mr{dR}(\mc{E},\varphi_\mc{E}))\bigg).
\end{equation*} 
\end{defn}
When $(\mc{E},\varphi_\mc{E})$ is clear from context, we often omit it from the notation, writing $\Fil^\bullet_{\bb{D}_\crys}$ instead.

\begin{prop} The functor 
\begin{equation*}
    \bb{D}_\crys\colon \cat{Vect}^\varphi(\mf{X}_\smallprism)\to \cat{VectNF}^\varphi(\mf{X}_\crys),
\end{equation*}
is a $\Z_p$-linear $\otimes$-functor, which preserves duals, and maps $\cat{Vect}^\varphi_{[0,a]}(\mf{X}_\smallprism)$ into $\cat{VectNF}^\varphi_{[0,a]}(\mf{X}_\crys)$.
\end{prop}

\subsubsection{Comparison to \texorpdfstring{$D_\crys$}{Dcrys} and Griffiths transversality} We now verify (when $\mf{X}\to \Spf(W)$ is smooth) that $\bb{D}_\mr{crys}$ really does form a functorial filtered $F$-crystal lattice in $D_\mr{crys}$.

\begin{thm}\label{thm:integral-dcrys-strongly-divisible-and-rational-agreeance} If $\mf{X}\to\Spf(W)$ is smooth, then $\bb{D}_\crys$ takes values in $\cat{VectWF}^{\varphi}(\mf{X}_\crys)$, and the following diagram commutes:
\bx{
    \cat{Vect}^\varphi(\mf X_\smallprism) \ar[r]^-{\D_\crys}\ar[d]_-{T_\et}
    & \cat{VectWF}^\varphi(\mf X_\crys) \ar[d]^{(-)[\nicefrac{1}{p}]}
    \\ \cat{Loc}^\crys_{\Z_p}(X) \ar[r]^-{D_\crys}
    & \cat{IsocF}^\varphi(\mf X). 
    }\ex
\end{thm}
As $D_\crys$ takes values in $\cat{IsocF}^\varphi(\mf{X})$ it suffices to prove the commutativity.

Before we prove Theorem \ref{thm:integral-dcrys-strongly-divisible-and-rational-agreeance} we set some notational shorthand. First, shorten the notation for $D_\crys\circ T_\et$ to just $D_\crys$. Let $R$ be a framed small $W$-algebra. Then, with notation as in \cite[\S1.1.5]{IKY1}, we further set
\begin{multicols}{2}
    \begin{itemize}
    \item $\Ainf\defeq \Ainf(\wt{R})$,
    \item $\Acrys\defeq \Acrys(\wt{R})$,
    \item $\widetilde{\mathrm{A}}_\crys\defeq \Ainf[\{\nicefrac{\tilde{\xi}^n}{n!}\}]^\wedge_p$,
    \item[]
    \item $\Bdr^+\defeq \Bdr^+(\wt{R})$,
    \item $\wt{\B}_\dR^+\defeq \Ainf[\nicefrac{1}{p}]^\wedge_{(\tilde{\xi})}$.
\end{itemize}
\end{multicols}
\noindent 
These rings are arranged in the following diagram
\begin{equation*}
\begin{tikzcd}
	& {\mathrm{A}_\mathrm{crys}} & {\tilde{\mathrm{A}}_\mathrm{crys}} & {\mathrm{A}_\mathrm{crys}} & {\mathrm{B}_\mathrm{dR}^+} \\
	{\mathrm{B}_\mathrm{dR}^+} && {\widetilde{\mathrm{B}}_\mathrm{dR}^+.}
	\arrow["\phi", from=1-2, to=1-3]
	\arrow[hook, from=1-3, to=2-3]
	\arrow[hook, from=1-3, to=1-4]
	\arrow[hook, from=1-4, to=1-5]
	\arrow["{\widetilde{\psi}}"', from=1-2, to=2-3]
	\arrow["\psi", curve={height=-24pt}, from=1-2, to=1-5]
	\arrow[hook, from=1-2, to=2-1]
	\arrow["{\phi_\mathrm{dR}}"', from=2-1, to=2-3]
\end{tikzcd}
\end{equation*}
Here we use the following notation:
\begin{itemize}[leftmargin=.3in]
    \item $\widetilde{\mathrm{A}}_\crys\hookrightarrow \Acrys$ is the natural inclusion, 
    \item $\phi_\dR\colon \B_\dR^+\isomto \wt{\B}_\dR^+$ and $\phi\colon \Acrys\isomto \widetilde{\mathrm{A}}_\crys$ are the maps induced by $\phi\colon \Ainf\to \Ainf$,
    \item and the maps $\psi$ and $\wt{\psi}$ are defined uniquely to make the diagram commute.
\end{itemize}
Set $\wt{\B}_\dR\defeq \Ainf[\nicefrac{1}{p}]^\wedge_{(\tilde{\xi})}[\nicefrac{1}{\wt{\xi}}]$. Then, a \emph{$\varphi$-module} over $\B_\dR^+$ is a triple $(M,\wt M,\varphi_M)$ with $M$ (resp.\ $\wt M$) a finitely generated projective $\B_\dR^+$-module (resp.\ $\wt{\B}^+_\dR$-module) and $\varphi_M$ a $\wt{\B}_\dR$-linear isomorphism $(\phi_\dR^*M)[\nicefrac{1}{\tilde\xi}]\to \wt M[\nicefrac{1}{\tilde\xi}]$. Let us denote the category of $\varphi$-modules over $\Bdr^+$ by $\cat{Vect}^\varphi(\B_\dR^+)$. 

We now observe that there is a natural functor 
\begin{equation*}
    \mc M\colon \cat{IsocF}^\varphi(R)\to \cat{Vect}^\varphi(\B_\dR^+),\quad (\mc{F},\varphi_\mc{F},\Fil^\bullet_F)\mapsto (\mc{F}(\Acrys\twoheadrightarrow \wt{R})\otimes_{\Acrys,\psi}\B_\dR^+,\wt{\Fil}^0_{F},\varphi_\mc{F}\otimes 1).
\end{equation*}
Here we define the filtration $\wt{\Fil}_{F}^\bullet$ on $\mc{F}(\Acrys\twoheadrightarrow \wt{R})\otimes_{\Acrys,\tilde{\psi}}\wt{\B}_\dR$ as follows: 
\begin{equation}\label{eq: tilde filtration}
\wt{\Fil}^r_{F}=\sum_{i+j=r}\Fil^i_F(\Acrys\twoheadrightarrow\wt{R})\,\otimes_{\Acrys,\tilde{\psi}}\,\tilde\xi^j\wt{\B}_\dR^+.
\end{equation}
Note that the Frobenius map $\varphi_\mc{F}\otimes 1$ is sensible as $\wt{\Fil}^0_{F}$ is a lattice in $\mc{F}(\Acrys\twoheadrightarrow \check{R})\otimes_{\Acrys,\tilde{\psi}}\wt{\B}_\dR$. On the other hand, consider the functor 
\begin{equation*}
    \mr{Nyg}_{\dR}\colon \cat{Vect}^\varphi(\B_\dR^+)\to \cat{MF}(\wt{\B}_\dR^+,\Fil^\bullet_{\tilde\xi}),\quad (M,\wt M,\varphi_M)\mapsto (\phi_\dR^\ast M,\Fil^\bullet_\mathrm{Nyg}(\phi^\ast_\dR M)),
\end{equation*}
where $\Fil^r_\mathrm{Nyg}(\phi^\ast_\dR M)\defeq \phi_\dR^*M\cap \varphi_M^{-1}(\tilde\xi^r\wt M)$. Then we have the following result.

\begin{lem}\label{lem: D crys on the ring B}
    The following diagram commutes.
    \xymatrixrowsep{1pc}\xymatrixcolsep{4pc}\bx{
    \cat{Vect}^\varphi(R_\smallprism) \ar[r]^-{D_\crys} \ar[d]_-{\mr{ev}_{\Ainf}}
    & \cat{IsocF}^\varphi(R) \ar[d]^-{\mc M} \ar[r]
    & \cat{MF}(\Acrys[\nicefrac{1}{p}],\Fil^\bullet_\xi) \ar[d]
    \\ \cat{Vect}^\varphi(\Ainf) \ar[r]^-{-\otimes_{\Ainf}\Bdr^+}
    & \cat{Vect}^\varphi(\B_\dR^+) \ar[r]^-{\mr{Nyg_{\dR}}}
    & \cat{MF}(\wt{\B}_\dR^+,\Fil^\bullet_{\tilde\xi}). 
    }\ex 
\end{lem}

The top-right arrow is the evaluation of $(\mc{F},\varphi_\mc{F},\Fil^\bullet_{F})$ at $\Acrys\twoheadrightarrow\wt{R}$ and the right vertical arrow is scalar extension along the map of filtered rings $\wt{\psi}\colon (\Acrys[\nicefrac{1}{p}],\Fil^\bullet_\xi)\to (\wt{\B}_\dR^+,\Fil^\bullet_{\tilde\xi})$. 

\begin{proof}[Proof of Lemma \ref{lem: D crys on the ring B}]
    The left square commutes by the proof of \cite[Theorem 4.8]{GuoReinecke} (cf.\ the proof of \cite[Lemma 2.27]{IKY1}). Let $(\mc F,\varphi_\mc F,\Fil^\bullet_F)$ be an object of $\cat{IsocF}^\varphi(R)$. As the category $\cat{MF}(\wt{B}_\dR^+,\Fil^\bullet_{\tilde\xi})$ satisfies descent (cf.\ \cite[Corollary 17.1.9]{ScholzeBerkeley}), to show the commutativity of the right square, we are free to localize on $\Spf(R)$. But, note that as the graded pieces of the filtration on $F$ are locally free, we may localize on $\Spf(R)$ to assume that they are free.\footnote{Indeed, suppose that $\mc{E}$ is a vector bundle on $\Spa(R[\nicefrac{1}{p}])$. Then, $\mc{E}(R[\nicefrac{1}{p}])$ is a  projective $R[\nicefrac{1}{p}]$-module (see \cite[Theorem 1.4.2]{KedlayaAWS}, and \cite{Kiehl}). Thus, there exists an open cover $\Spec(R[\nicefrac{1}{f_i p}])$ of $\Spec(R[\nicefrac{1}{p}])$ such that $\mc{E}(R[\nicefrac{1}{p}])\otimes_{R[\nicefrac{1}{p}]}R[\nicefrac{1}{p f_i}]$ is trivial, where $f_i$ is a collection of elements of $R$. So, replacing $\Spf(R)$ by $\Spf(R[\nicefrac{1}{ f_i}]^\wedge_p)$, one may Zariski localize on $\Spf(R)$ to assume that $\mc{E}$ is free.} This implies that the evaluation of $(\mc F,\Fil^\bullet_F)$ at any object $A\twoheadrightarrow R'$ of $(R/W)_\crys$ admitting a map to $\id_R\colon R\to R$ is free as a filtered module over $(A[\nicefrac{1}{p}],\Fil_\mr{PD}^\bullet[\nicefrac{1}{p}])$. 
    
    Fix a filtered basis $(e_\nu,r_\nu)_{\nu=1}^n$ of $\mc F(\Acrys\twoheadrightarrow\wt{R})[\nicefrac{1}{p}]$, so $\mc M(\mc F,\varphi_\mc{F},\Fil^\bullet_F)=(M,\wt M,\varphi_M)$ where: 
    \begin{equation*}
        M=\mc F(\Acrys\twoheadrightarrow\wt{R})\otimes_{\Acrys,\psi}\B^+_\dR,\quad \wt M=\sum_{\nu=1}^n\tilde\xi^{-r_\nu}\wt \B_\dR^+\cdot e_\nu\subset \bigoplus_{\nu=1}^n\wt \B_\dR\cdot e_\nu,
    \end{equation*} and $\varphi_M$ is the scalar extension of $\varphi_{\mc F}$. Thus, the Nygaard filtration on $\phi_\dR^*M$ is 
    \be \Fil^r_{\mr{Nyg}}(\phi^\ast_\dR M)=\phi_\dR^*M\cap \varphi_M^{-1}\left(\sum_{\nu=1}^n\tilde\xi^{r-r_\nu}\wt{\B}_\dR^+\cdot e_\nu\right).
    \ee
    On the other hand, if the object $(M',\Fil_{M'}^\bullet)$ denotes the image under the other composition in the right-hand square, then it may be described as follows:
    \begin{equation*}
        M'=\mc F(\Acrys\twoheadrightarrow\wt{R})\otimes_{\Acrys,\wt\psi}\wt{\B}_\dR^+,\qquad \Fil_{M'}^r=\sum_{\nu=1}^n\Fil^{r-r_\nu}_{\tilde\xi}\cdot e_\nu.
    \end{equation*} 
    Thus, $\varphi_{M}$ 
    induces an isomorphism $(\phi_\dR^*M,\Fil^\bullet_{\mr{Nyg}}(\phi^\ast_\dR M))\isomto (M',\Fil^\bullet_{M'})$. 
\end{proof}

\begin{proof}[Proof of Theorem \ref{thm:integral-dcrys-strongly-divisible-and-rational-agreeance}]
    Let $(\mc E,\varphi_\mc E)$ be an object of $\cat{Vect}^\varphi(\mf{X}_\smallprism)$. We will show that there is a canonical isomorphism $\underline{\D}_\crys(\mc E,\varphi_\mc{E})[\nicefrac{1}{p}]\isomto \underline{D}_\crys(\mc E,\varphi_\mc{E})$ in $\cat{Isoc}^\varphi(\mf{X})$ respecting filtrations. Note that by the proof of \cite[Theorem 4.8]{GuoReinecke}, the underlying $F$-isocrystals are identified, and so it suffices to show that the filtrations are matched. So, we may assume that $\mf X=\Spf(R)$ where $R$ is a framed small $W$-algebra. For an object $(\mc E,\varphi_\mc{E})$ of $\cat{Vect}^\varphi(R_\smallprism)$, let $\D_{\crys,R}[\nicefrac{1}{p}](\mc E,\varphi_\mc{E})$ (resp.\ $\underline{\D}_{\crys,R}[\nicefrac{1}{p}](\mc E,\varphi_\mc{E})$) denote the filtered $R[\nicefrac{1}{p}]$-module $\D_{\crys}(\mc E,\varphi_\mc{E})(R)[\nicefrac{1}{p}]$ (resp.\ its underlying $R[\nicefrac{1}{p}]$-module). Define $D_{\crys,R}(\mc E,\varphi_\mc{E})$ and $\underline{D}_{\crys,R}(\mc E,\varphi_\mc{E})$ similarly. 

    Define the functor 
    \begin{equation*}
        \mr{Nyg}_{\mf S_R}\colon \cat{Vect}^\varphi(\mf S_R,(E))\to \cat{MF}(\mf S_R,\Fil^\bullet_E),\quad (\mf{M},\varphi_\mf{M})\mapsto (\phi^\ast\mf{M} ,\Fil^\bullet_\mr{Nyg}).
    \end{equation*} 
    Then we have the following diagram.
    \xymatrixrowsep{3pc}\xymatrixcolsep{8.5pc}\bx{
        \cat{Vect}^\varphi(R_\smallprism) \ar[d]^-{\mr{ev}_{\mf{S}_R}} \ar[rr]^-{\D_{\crys,R}[\nicefrac{1}{p}]} 
        && \cat{MF}(R[\nicefrac{1}{p}],\Fil^\bullet_\triv) \ar@{=}[d]
        \\\cat{Vect}^\varphi(\mf S_R) \ar[d]^-{(-)\otimes_{{\mf S}_R}\B_\dR^+} \ar[r]^{\mr{Nyg}_{\mf S_R}} 
        & \cat{MF}(\mf S_R,\Fil^\bullet_{E}) \ar[d]^-{(-)\otimes_{\mf S_R}\wt{\B}_\dR^+} \ar[r]^-{(-)\otimes_{\mf S_R}R[\nicefrac{1}{p}]}
        & \cat{MF}(R[\nicefrac{1}{p}],\Fil^\bullet_\triv) \ar[d]^-{(-)\otimes_{R}\wt R}
        \\ \cat{Vect}^\varphi(\B_\dR^+) \ar[r]^-{\mr{Nyg}_{\dR}}  
        & \cat{MF}(\wt{\B}_\dR^+,\Fil^\bullet_{\tilde\xi}) \ar[r]^-{\mod \tilde\xi}
        & \cat{MF}(\wt R[\nicefrac{1}{p}],\Fil^\bullet_\triv),
    }\ex 
    with the maps $\mf{S}_R\to \Bdr^+$  and $\mf S_R\to \wt B_\dR^+$ given by the compositions $\mf{S}_R\xrightarrow{\alpha_\crys}\Acrys\xrightarrow{\psi}\Bdr^+$ and $\mf{S}_R\xrightarrow{\alpha_\crys}\Acrys\xrightarrow{\wt\psi}\wt{\B}_\dR^+$, respectively.
    
    The upper rectangle and the right lower square commute by definition. Noting that $\mf S_R\to \Bdr^+$ is flat as $\mf S_R$ is Noetherian and it is $E$-adically flat (indeed, the image of $E$, which is $\xi$, is a nonzerodivisor and the map is flat mod $E$ by \cite[Lemma 1.15]{IKY1}), we get that the left lower square commutes as well. Thus,  we have a canonical identification 
    \begin{equation*}
        \D_{\crys,R}[\nicefrac{1}{p}](\mc{E},\varphi_\mc{E})\otimes_R\wt R\cong  \mr{Nyg}_{\dR}\left(\mc{E}(\mf{S}_R,(E))\otimes_{{\mf S}_R}\B_\dR^+\right)\otimes_{\wt{\B}_\dR^+}\wt R[\nicefrac{1}{p}]
    \end{equation*}
    On the other hand, by Lemma \ref{lem: D crys on the ring B}, we also have a canonical identification 
    \begin{equation*}
        D_{\crys,R}(\mc{E},\varphi_\mc{E})\otimes_R\wt R\cong  \mr{Nyg}_{\dR}\left(\mc{E}(\mf{S}_R,(E))\otimes_{{\mf S}_R}\B_\dR^+\right)\otimes_{\wt{\B}_\dR^+}\wt R[\nicefrac{1}{p}].
    \end{equation*}
    These identifications induce an isomorphism $\underline{\D}_{\crys,R}[\nicefrac{1}{p}](\mc E)\otimes_R\wt{R}\isomto \underline{D}_{\crys,R}(\mc E)\otimes_R\wt{R}$ which agrees with that from \cite[Theorem 4.8]{GuoReinecke}. Thus, $\underline{\D}_{\crys,R}[\nicefrac{1}{p}](\mc E)\isomto \underline{D}_{\crys,R}(\mc E)$ preserves filtrations after base change along the faithfully flat map $R[\nicefrac{1}{p}]\to \wt{R}[\nicefrac{1}{p}]$, and so preserves filtrations.
\end{proof} 

\begin{example}\label{ex:Kisin-Dcrys-II}Similar to Example \ref{ex:Kisin-Dcrys}, when $\mf{X}=\Spf(W)$, we define
\begin{equation*}
    \bb{D}_\crys\colon \cat{Rep}_{\bb{Z}_p}^\crys(\Gal(\ov{K}_0/K_0))\to \cat{VectWF}^{\varphi}(W),\qquad \Lambda\mapsto \bb{D}_\crys(T^{-1}_{\et}(\Lambda)). 
\end{equation*}
By Theorem \ref{thm:integral-dcrys-strongly-divisible-and-rational-agreeance}, we have an identification of filtered $F$-isocrystals $\D_\crys(\Lambda)[\nicefrac{1}{p}]\isomto D_\crys(\Lambda)$. This agrees with the composition
\begin{equation*}
    \underline{\bb{D}}_\crys(\Lambda)[\nicefrac{1}{p}]\isomto (\phi^\ast(\mf{M}(\Lambda))/u)[\nicefrac{1}{p}]\isomto (\mf{M}(\Lambda)/(u))[\nicefrac{1}{p}]\isomto D_\crys(\Lambda),
\end{equation*}
where the first isomorphism is that in Example \ref{ex:Kisin-Dcrys}, the second is from the Frobenius structure, and the last is from the definition of $\mf{M}$ (see \cite[Theorem (1.2.1) (1)]{KisIntShab}).
\end{example}


\subsection{Relationship to locally filtered free prismatic \texorpdfstring{$F$}{F}-crystals} Recall (see \cite[Definition 1.24]{IKY2}), that a prismatic $F$-crystal $(\mc{E},\varphi_\mc{E})$ on $\mf{X}_\smallprism$ is called \emph{locally filtered free (lff)} if $(\phi^\ast\mc{E},\Fil^\bullet_\mr{Nyg})$ is locally filtered free over $(\mc{O}_\smallprism,\Fil^\bullet_{\mc{I}_\smallprism})$ in the sense of \cite[Definition 1.1]{IKY2}. In this section we discuss the interplay between the functor $\bb{D}_\crys$ and the notion of being lff.

\subsubsection{Lffness in terms of \texorpdfstring{$\bb{D}_\mr{crys}$}{DDcrys}} We begin by observing that $\bb{D}_\mr{crys}$ not only enjoys nicer properties on lff prismatic $F$-crystals but, in fact, can detect lffness.

\begin{prop}\label{prop:filtered-equiv} Let $\mf{X}$ be a base formal $W$-scheme, and $(\mc{E},\varphi_\mc{E})$ a prismatic $F$-crystal on $\mf{X}$. Let $\{\Spf(R_i)\}$ be an open cover with each $R_i$ a formally framed base $W$-algebra, and for each $i$ let $\mf{M}_i\defeq \mc{E}(\mf{S}_{R_i},(E))$. Consider the following conditions:
\begin{enumerate}
    \item the prismatic $F$-crystal $(\mc{E},\varphi_\mc{E})$ is locally filtered free,
    \item the filtration $\Fil^\bullet_\mr{Nyg}(\phi^\ast\mf{M}_i)$ is locally free over $(\mf{S}_{R_i},\Fil^\bullet_E)$,
    \item the filtration $\mathrm{Fil}^\bullet_{\bb{D}_\crys}\subseteq \bb{D}_\crys(\mc{E},\varphi_\mc{E})$ is locally free over $(\mc{O}_\mf{X},\Fil^\bullet_\mr{triv})$,
    \item $\bb{D}_\crys(\mc{E},\varphi_\mc{E})$ is an object of $\cat{VectNF}^{\varphi,\mr{sd},\mr{gtf}}(\mf{X}_\crys)$,
    \item[(4')] $\bb{D}_\crys(\mc{E},\varphi_\mc{E})$ is an object of $\cat{VectNF}^{\varphi,\pddiv,\mr{gtf}}(\mf{X}_\crys)$,
    \item $\bb{D}_\crys(\mc{E},\varphi_\mc{E})$ is an object of $\cat{VectF}^{\varphi,\mr{sd}}(\mf{X}_\crys)$.
\end{enumerate}
Then, (1), (2), and (3) are equivalent, (4) implies (1), (1) implies (4'), and if $\mf{X}/W$ is smooth, then (4) and (5) are equivalent.
\end{prop}

We begin by establishing the following filtered-freeness lifting lemma.

\begin{lem}\label{lem: free BK mod}
    Let $(A,(d))$ be a bounded prism, and $(M,\varphi_M)$ an object of $\cat{Vect}^\varphi(A,(d))$ with $M$ free over $A$. Set $\Fil^\bullet\defeq \Fil^\bullet_\mr{Nyg}(\phi^\ast M)$ and $\ov{\Fil}^\bullet\defeq \ov{\Fil}^\bullet_\mr{Nyg}(\phi^\ast(M)/(d))$. Let $(f_\nu,r_\nu)_{\nu=1}^n$ be a filtered basis of $(\phi^\ast(M)/(d),\ov{\Fil}^\bullet)$ over $(A/(d),\Fil^\bullet_\triv)$. Choose $e_\nu$ in $\Fil^{r_\nu}$ such that $\ov{e_\nu}=f_\nu$ (which exist as $f_\nu$ is in $\ov{\Fil}^{r_\nu}$). Then, the following is true:
    \begin{enumerate}
        \item $(e_\nu,r_\nu)_{\nu=1}^n$ is a filtered basis of $(\phi^\ast M,\Fil^\bullet)$ over $(A,\Fil^\bullet_d)$,
        \item $\varphi_M(e_\nu)$ is in $d^{r_\nu}M$, and $(\tfrac{\varphi_M(e_\nu)}{d^{r_\nu}})_\nu$ is a basis of $M$. 
    \end{enumerate}
\end{lem}
\begin{proof} To see that $(e_\nu)$ is a basis of $\phi^\ast M$ we observe that the map $A^n\to \phi^\ast M$ sending $(a_1,\ldots,a_n)$ to $\sum_{\nu=1}^n a_\nu e_\nu$ is surjective modulo $d$ by assumption, and so surjective by Nakayama's lemma as $A$ is $d$-adically complete. It is then an isomorphism as the source and target are finite projective $A$-modules of the same rank. 

To prove the first assertion we consider the natural map $(A,\Fil_d^\bullet)\to (A/(d),\Fil^\bullet_\mr{triv})$ of filtered rings and the corresponding map of Rees algebras $\mr{Rees}(\Fil^\bullet_d)\to \mr{Rees}(\Fil^\bullet_\mr{triv})\simeq A/(d)[t]$. 
We observe that, by Lemma \ref{lem:derived Rees tensor identification}, the filtration $\ov{\Fil}^\bullet$ on $\phi^*M/(d)$ defines via the Rees construction the $\mr{Rees}(\Fil^\bullet_\mr{triv})$-module $\mr{Rees}(\Fil^\bullet)\otimes_{\mr{Rees}(\Fil^\bullet_d)}\mr{Rees}(\Fil^\bullet_\mr{triv})$, more precisely, that the natural map 
\begin{equation}\label{eq:image filtration as Rees tensor product}
\mr{Rees}(\Fil^\bullet)\otimes_{\mr{Rees}(\Fil^\bullet_d)}\mr{Rees}(\Fil^\bullet_\mr{triv})\to \mr{Rees}(\ov{\Fil}^\bullet)    
\end{equation}
is an isomorphism. Thus, we can apply \cite[Remark 1.9]{IKY2} to obtain the first assertion. 

We now prove the second assertion. We take an integer $r$ large enough so that $r\geqslant r_\nu$ for any $\nu$ and that $\varphi_M^{-1}(d^rM)\subset \phi^*M$. In particular, $\varphi_M$ induces $\Fil^r=\sum_\nu d^{r-r_\nu}A\cdot e_\nu\isomto d^rM$, which implies that $d^{r-r_\nu}\varphi(e_\nu)$ forms a basis of $d^rM$. Dividing by $d^r$ we obtain assertion (2).   
\end{proof}

We next show that for `filtered free' prismatic $F$-crystals $(\mc{E},\varphi_\mc{E})$, the strong PD-divisibility condition on $\bb{D}_\crys(\mc{E},\varphi_\mc{E})$ is automatic. To make this precise, fix a base $W$-algebra $R$. Define the category $\cat{Vect}^\varphi_{\mr{free}}(R_\smallprism)$ to be the full subcategory $\cat{Vect}^\varphi(R_\smallprism)$ consisting of those $(\mc{E},\varphi_\mc{E})$ with $(\bb{D}_\dR(\mc{E},\varphi_\mc{E}),\Fil^\bullet_{\bb{D}_\dR}(\mc{E},\varphi_\mc{E}))$ (equiv.\@ $\bb{D}_\crys(\mc{E},\varphi_\mc{E})$) filtered free over $(R,\Fil^\bullet_\mr{triv})$ in the sense of \cite[Definition 1.1]{IKY2}.

\begin{prop}\label{lff implies pd-divisible}
    Let $R$ be a base $W$-algebra with framing $w$ and let $(\mc{E},\varphi_\mc{E})$ be an object of $\cat{Vect}^\varphi_\mr{free}(R_\smallprism)$. Then, $\mbb{D}_\crys(\mc{E},\varphi_\mc{E})$ is strongly PD-divisible. 
    If $(\mc E,\varphi_\mc E)$ has level in $[a,b]$ with $b-a\leqslant p-1$, then it is strongly divisible. 
\end{prop}

\begin{proof} Note that the second claim follows from the first by Remark \ref{rem:PD-equals-normal-in-FL-range}. Thus, we focus only on the first statement.

Set $(\mc F,\varphi_\mc{F},\Fil^\bullet_\mc{F})=\D_\crys(\mc E,\varphi_\mc E)$. Note that the Nygaard filtration on $\phi_w^*\mc F(R)$ is induced from the Nygaard filtration $\Fil^\bullet_\mr{Nyg,\mc E}$ on $\phi_w^*\mc E(R)$ through the identification $\phi_w^*\mc E(R)=\mc F(R)$ and the scalar extension along $\phi_R$. So, by Lemma \ref{lem:PD-div-in-terms-of-filtered-tensor-products}, strong PD-divisibility is equivalent to the identity map $\mc F(R)\isomto\mc F(R)$ induces a filtered isomorphism
    \begin{equation}\label{the filtered isomorphism of PD filtered modules}
        (\mc F(R),\Fil^\bullet_\mc F(R))\otimes_{(R,\Fil^\bullet_\mr{triv})}(R,\Fil^\bullet_\mr{PD})\isomto (\mc F(R),\Fil^\bullet_\mr{Nyg,\mc E})\otimes_{(R,\Fil^\bullet_p)}(R,\Fil^\bullet_\mr{PD}).
    \end{equation}
    To show this, we consider the following commutative diagram of filtered rings,
    \begin{equation*}
        \begin{tikzcd}
            (R,\Fil^\bullet_\mr{triv}) 
            &   (S_R,\Fil^\bullet_\mr{PD})
            &   (R,\Fil^\bullet_\mr{PD})
            \\ 
            &   (\mf S_R,\Fil^\bullet_E)
            &   (R,\Fil^\bullet_p).
            \arrow["{\mr{nat}}", from=1-1, to=1-2]
            \arrow["{\mr{sp}_\dR}", from=1-2, to=1-3]
            \arrow["{\mr{sp}_\dR}", from=2-2, to=2-3]
            \arrow["\mr{nat}", from=2-2, to=1-2]
            \arrow["\id", from=2-3, to=1-3]
        \end{tikzcd}
    \end{equation*}
    The left-hand side in (\ref{the filtered isomorphism of PD filtered modules}) is then the filtered scalar extension of the filtered module $(\mc F(R),\Fil^\bullet_\mc F(R))$ along the top horizontal composition. 
    On the other hand, we claim that $(\mc F(R),\Fil^\bullet_\mr{Nyg,\mc E})$ 
    is identified with the filtered scalar extension of the $(\mf S_R,\Fil^\bullet_E)$-module $(\phi^*\mc E(\mf S_R,(E)),\Fil^\bullet_\mr{Nyg,\mc E})$ along the bottom horizontal map. 
    To see this, we take a filtered basis $(e_\nu,r_\nu)_\nu$ of the Nygaard filtered module over $(\mf S_R,\Fil^\bullet_E)$ such that $(\varphi(e_\nu)/E^{r_\nu})_\nu$ is a basis of the $\mf S_R$-module $\mc E(\mf S_R,(E))$, as in Lemma \ref{lem: free BK mod}. 
    Base changing this basis along the map of prisms $(\mf S_R,(E))\to (R,(p))$ with the underlying ring map sending $u$ to $0$, we see that the images $\bar e_\nu$ of $e_\nu$ in $\phi^*\mc E(R,(p))$ form a basis such that $\varphi(e_\nu)$ is divisible by $p^{r_\nu}$ and $(\varphi(\bar e_\nu)/p^{r_\nu})_\nu$ is a basis of the $R$-module $\mc E(R,(p))$. 
    This shows that the Nygaard filtration for $\phi^\ast\mc E(\mf S_R,(E))$ base changes to the Nygaard filtration for $\phi^\ast{\mc E(R,(p))}$. 
    Then the assertion follows from Proposition \ref{filtered commutativity of Breuil prism diagram} (2).
\end{proof}
\begin{rem}
    We make the following remarks. 
    \begin{enumerate}
        \item Proposition \ref{lff implies pd-divisible} can be shown using \cite[Theorem 1 in \S 3]{TVX} instead of Proposition \ref{filtered commutativity of Breuil prism diagram}. 
        The authors thank Qixiang Wang for pointing this out. 
        \item Proposition \ref{lff implies pd-divisible} gives the PD-divisibility for the Frobenius structure on crystalline cohomology, which Mazur obtains in \cite[p. 665]{Mazur1972} (whose proof is in \cite{mazur1973frobenius}). For more information on this material, see \cite[Remark 4.6]{TVX}. 
    \end{enumerate}
\end{rem}

We are now ready to prove Proposition \ref{prop:filtered-equiv}.

\begin{proof}[Proof of Proposition \ref{prop:filtered-equiv}] The equivalence of (1) and (2) follows from \cite[Proposition 1.16]{IKY1}. The equivalence of (2) and (3) follows from Lemma \ref{lem: free BK mod}. To show that (3) implies (4'), it suffices to consider a cover $\{\Spf(R_i)\}$ of $\mf{X}$ where each $R_i$ is a base $W$-algebra, and the restriction of $(\mc{E},\varphi_\mc{E})$ to $(R_i)_\smallprism$ belongs to $\cat{Vect}_\mr{free}^\varphi((R_i)_\smallprism)$. The desired implication then follows from Proposition \ref{lff implies pd-divisible}. That (4) implies (3) is given by Proposition \ref{prop:ptf-implies-loc-free}.
Finally, when $\mf{X}\to\Spf(W)$ is smooth, the equivalence of (4) and (5) follows by combining Proposition \ref{prop:ptf-implies-loc-free} and Theorem \ref{thm:integral-dcrys-strongly-divisible-and-rational-agreeance}.
\end{proof}

\subsubsection{Exactness in the strongly divisible lff case}
We now make the observation that exactness of $\bb{D}_\crys$ holds when restricted to the category $\cat{Vect}^{\varphi,\mr{lff},\mr{sd}}(\mf{X}_\smallprism)$ of strongly divisible lff prismatic $F$-crystals where, by definition, we call an lff prismatic $F$-crystal \emph{strongly divisible} if $\bb{D}_\mr{crys}(\mc{E},\varphi_\mc{E})$ is a strongly divisible filtered $F$-crystal.

\begin{prop}\label{prop:dcrys-exact} The functor 
\begin{equation*}
    \bb{D}_\crys\colon \cat{Vect}^{\varphi,\mr{lff},\mr{sd}}(\mf{X}_\smallprism)\to \cat{VectNF}^\varphi(\mf{X}_\crys),
\end{equation*}
is exact.
\end{prop}
\begin{proof}For this we may reduce to the case when $\mf{X}=\Spf(R)$, where $R$ is a formally framed base $W$-algebra. Consider an exact sequence 
\begin{equation*}
    0\to (\mc{E}_1,\varphi_{\mc{E}_1})\to (\mc{E}_2,\varphi_{\mc{E}_2})\to (\mc{E}_3,\varphi_{\mc{E}_3})\to 0
\end{equation*}
in $\cat{Vect}^{\varphi,\mr{lff},\mr{sd}}(R_\smallprism)$. Set $\mf{M}_i\defeq \mc{E}_i(\mf{S}_R,(E))$ and write $\Fil^\bullet_i\subseteq \phi^\ast\mf{M}_i/u$ for $\iota_i^{-1}(\ov{\Fil}^\bullet_\mr{Nyg}(\phi^\ast \mf{M}_i/E))$, with $\iota_i$ as in Construction \ref{construction: crys dR isom with Breuil prism}.
Then, as evaluation at $(\mf{S}_R,(E))$ is exact (see \cite[Lemma 1.18]{IKY1}), and an exact sequence of vector bundles is universally exact, we see from Propositions \ref{prop:de-Rham-realization-BK-relationship} and \ref{prop:crys-dR-comparison-using-prisms} that it suffices to show that the sequence of filtered modules over $(R,\Fil^\bullet_\triv)$
\begin{equation*}
   0\to (\phi^\ast \mf{M}_1/u,\Fil^\bullet_1)\to (\phi^\ast \mf{M}_2/u,\Fil^\bullet_2)\to (\phi^\ast \mf{M}_3/u,\Fil^\bullet_3)\to 0,
\end{equation*}
is short exact. As this sequence is short exact on the underlying module and each $(\phi^\ast \mf{M}_i/u,\Fil^\bullet_i)$ is finitely supported, it suffices by \cite[Lemma 1.3]{IKY2} to show that the morphisms $(\phi^\ast \mf{M}_1/u,\Fil^\bullet_1)\to (\phi^\ast \mf{M}_2/u,\Fil^\bullet_2)$ and $(\phi^\ast \mf{M}_2/u,\Fil^\bullet_2)\to (\phi^\ast \mf{M}_3/u,\Fil^\bullet_3)$ are strict. But, as each $(\mc{E}_i,\varphi_{\mc{E}_i})$ is in $\cat{Vect}^{\varphi,\mr{lff},\mr{sd}}(R_\smallprism)$, the Faltings morphism \eqref{eq: Faltings morphism} is an isomorphism, and thus, the claim follows from \cite[Theorem 2.1 (3)]{Faltings89} (cf.\@ \cite[Theorem 2.9 (3)]{LiuMoonPatel}). 
\end{proof}

\subsection{Bi-exactness of forgetful functor} In this final short subsection we observe that one can use Proposition \ref{prop:dcrys-exact} to prove that the forgetful functor $\mr{R}_{\mf{X}}$ is bi-exact when restricted to the full subcategory $\cat{Vect}^\mr{sd}(\mf{X}^\syn)$ defined to be the essential preimage of $\cat{Vect}^{\varphi,\mr{lff},\mr{sd}}(\mf{X}_\smallprism)$ under $\mr{R}_\mf{X}$.

\begin{prop}\label{prop:G_X-bi-exact} Let $\mf{X}$ be a base formal $W$-scheme. Then, the $\Z_p$-linear $\otimes$-equivalence 
\begin{equation*}
    \mathrm{R}_\mf{X}\colon \cat{Vect}^\mr{sd}(\mf{X}^\mr{syn})\isomto\cat{Vect}^{\varphi,\mr{lff},\mr{sd}}(\mf{X}_\smallprism)
\end{equation*}
is bi-exact. 
\end{prop}
\begin{proof} The exactness of $\mathrm{R}_\mf{X}$ is clear. Conversely, using \cite[Corollary 1.18]{IKY2}, \cite[Proposition 1.5]{IKY2}, and the method of proof in \cite[Proposition 1.28]{IKY2} (which shows that the Rees module over $\mr{Rees}(\Fil^\bullet_\mr{Nyg}(\Prism_R))$ is a descent of $\mr{Rees}(\Fil^\bullet_\mr{Nyg}(\phi^\ast\mc{E}))$ along the map on Rees algebras induced by $\phi$) we are thus reduced to showing that if 
\begin{equation*}
    0\to (\mc{E}_1,\varphi_{\mc{E}_1})\to (\mc{E}_2,\varphi_{\mc{E}_2})\to (\mc{E}_3,\varphi_{\mc{E}_3})\to 0,
\end{equation*}
is an exact sequence in $\cat{Vect}^{\varphi,\mr{lff},\mr{sd}}(\mf{X}_\smallprism)$ then for all $r$ in $\Z$ we have that
\begin{equation*}
    0\to \Fil^r(\phi^\ast\mc{E}_1)\to \Fil^r(\phi^\ast\mc{E}_2)\to \Fil^r(\phi^\ast\mc{E}_3)\to 0,
\end{equation*}
is exact. As this is a local condition, \cite[Proposition 1.16]{IKY1} reduces us to showing exactness after evaluation on the Breuil--Kisin prism associated to each open subset $\Spf(R)\subseteq\mf{X}$, where $R$ is a (formally framed) base $W$-algebra. But, this follows from Proposition \ref{prop:dcrys-exact} and its proof.
\end{proof}

\section{Relationship to Fontaine--Laffaille theory}\label{ss:FL-theory} We now relate $\bb{D}_\crys$ to relative Fontaine--Laffaille theory as first developed in \cite{Faltings89}, when $\mf{X}$ is smooth over $W$.  Throughout we use notation and terminology from \hyperref[notation-and-terminology]{Notation and terminology} without comment but now assume further that $p$ is odd.

\subsection{Statement of main result and reduction steps}\label{ss:statement-of-main} We are interested in understanding the following diagram of $\mathbb{Z}_p$-linear functors when $\mf{X}$ is a smooth formal $W$-scheme:
\begin{equation}\label{eq:big-diagram}
  \begin{tikzcd}[sep=large]
	{\cat{Vect}^{\varphi,\mr{lff}}_{[0,p-2]}(\mf{X}_\smallprism)} && {\cat{Vect}_{[0,p-2]}(\mf{X}^\mr{syn})} \\
	{\cat{Vect}^\varphi_{[0,p-2]}(\mf{X}_\smallprism)} \\
	{\cat{Vect}^{\varphi,\mr{an}}_{[0,p-2]}(\mf{X}_\smallprism)} & {\cat{Loc}^\mr{crys}_{\mathbb{Z}_p,[0,p-2]}(X)} & {\cat{VectF}^{\varphi,\mr{sd}}_{[0,p-2]}(\mf{X}_\mr{crys})}
	\arrow[hook, from=1-1, to=2-1]
	\arrow["{\bb{D}_\mr{crys}}"{description}, from=1-1, to=3-3]
	\arrow["{\mr{R}_\mf{X}}"{description}, from=1-3, to=1-1]
	\arrow["{\bb{D}_\mr{crys}}"{description}, from=1-3, to=3-3]
	\arrow[hook, from=2-1, to=3-1]
	\arrow["{T_\et}"{description}, from=3-1, to=3-2]
	\arrow["{T_\mr{crys}}"{description}, from=3-3, to=3-2]
\end{tikzcd}\
\end{equation}
Each of these functors, except the diagonal $\bb{D}_\mr{crys}$ and $T_\mr{crys}$, has a self-evident definition or has been defined in \S\ref{s:integral-Dcrys}. We explain the definition of these remaining two functors below.

First, we have a $\bb{Z}_p$-linear $\otimes$-equivalence
\begin{equation*}
    \mathrm{R}_\mf{X}\colon \cat{Vect}(\mf{X}^\mr{syn})\isomto \cat{Vect}^{\varphi,\mr{lff}}(\mf{X}_\smallprism).
\end{equation*}
Composing this with the functor $\bb{D}_\crys$ and applying Proposition \ref{prop:filtered-equiv} gives us an exact $\mathbb{Z}_p$-linear $\otimes$-functor
\begin{equation}\label{eq:Dcrys-FGauge}
    \cat{Vect}(\mf{X}^\mr{syn})\to \cat{VectF}^{\varphi,\pddiv}(\mf{X}_\mr{crys}),
\end{equation}
which we abusively also denote $\bb{D}_\mr{crys}$. This, together with Remark \ref{rem:PD-equals-normal-in-FL-range}, gives rise to the diagonal arrow in \eqref{eq:big-diagram}, where we write $\cat{Vect}_{[a,b]}(\mf{X}^\mr{syn})$ to be the full subcategory of $\cat{Vect}(\mf{X}^\syn)$ consisting of those $\mc{V}$ such that $\mathrm{R}_\mf{X}(\mc{V})$ is an object of $\cat{Vect}^\varphi_{[a,b]}(\mf{X}_\smallprism)$.\footnote{See \cite[\S8.4.1]{GMM} for a stack-theoretic definition of these \emph{Hodge--Tate weights}.} 

Second, the functor
\be
T^*_\mathrm{crys}\colon \cat{VectF}^{\varphi,\mr{sd}}_{[0,p-2]}(\mf X_\crys)\to \cat{Loc}_{\Z_p}(X),
\ee
constructed by Faltings in \cite[II.e), pp.\@ 35--37]{Faltings89} (see also \cite[\S4]{Tsu20}), is described as follows. For an object $\mc F=(\mc F,\varphi,\Fil_\mc{F}^\bullet)$ of $\cat{VectF}^{\varphi,\mr{sd}}_{[0,p-2]}(\mf X_\crys)$, the reductions $\mc{F}_m$ of $\mc{F}$ modulo $p^m$ define a projective system of objects of the category $\mathfrak{MF}^\nabla(\mf X)$ defined in \cite[2.c)--d), pp.\@ 30--33]{Faltings89}. With the notation of loc.\ cit., $T^*_\crys(\mc F)$ is then given as the inverse limit $\varprojlim_{m}\mathbf D(\mc F_m)$. We then set $T_\mr{crys}(-)=T^\ast_\mr{crys}(-)^\vee$, where $(-)^\vee$ is the dual in $\cat{Loc}_{\mathbb{Z}_p}(X)$, which gives the last undefined arrow in \eqref{eq:big-diagram}.

Our main goal in this section is to prove the following.

\begin{thm}\label{thm:big-equiv-diagram} Suppose that $\mf{X}$ is a smooth formal $W$-scheme. Then,  \eqref{eq:big-diagram} is a $2$-commutative diagram where every arrow is an equivalence.
\end{thm}

Many parts of Theorem \ref{thm:big-equiv-diagram} are either explicitly or implicitly contained in \S\ref{s:integral-Dcrys} or other references. Thus, we single out below the two remaining results to be demonstrated.

\begin{prop}
\label{prop:D crys and FL}    
    
Let $\mf{X}$ be a smooth formal $W$-scheme and set $T_\et^\ast(-)\defeq T_\et(-)^\vee$. Then, the following diagram commutes:
    \bx{
    &\cat{Vect}^{\varphi,\mr{lff}}_{[0,p-2]}(\mf X_\smallprism)\ar[rd]^-{\D_\crys}\ar[ld]_-{T^*_\et}
    \\\cat{Loc}_{\mathbb Z_p}(\mf X_\eta)
    &&\cat{VectF}^{\varphi,\mr{sd}}_{[0,p-2]}(\mf X_\crys)\ar[ll]^-{T^*_\crys}.
    }\ex
\end{prop}

 \begin{prop}\label{prop:lff-FL-range} Suppose that $\mf{X}$ is a base formal $W$-scheme, and that $(\mc{E},\varphi_\mc{E})$ is an effective prismatic $F$-crystal on $\mf{X}$ of height at most $p-2$. Then, $(\mc{E},\varphi_\mc{E})$ is locally filtered free.
\end{prop}

We claim that these two propositions are sufficient to prove Theorem \ref{thm:big-equiv-diagram}.

\begin{prop} \emph{Propositions \ref{prop:D crys and FL}} and \ref{prop:lff-FL-range} imply \emph{Theorem \ref{thm:big-equiv-diagram}}.
\end{prop}
\begin{proof} Note $R_\mf{X}$ is an equivalence by \cite[Proposition 1.28]{IKY2}, the inclusion $\cat{Vect}^\varphi_{[0,p-2]}(\mf{X}_\smallprism)\hookrightarrow \cat{Vect}^{\varphi,\mr{an}}_{[0,p-2]}(\mf{X}_\smallprism)$ is an equivalence by \cite[Remark 3.38]{DLMS}, $T_\et$ is an equivalence by \cite[Theorem A]{GuoReinecke}, and $T_\crys$ is fully faithful by \cite[Theorem 2.6]{Faltings89}. As we also know this diagram commutes by Proposition \ref{prop:D crys and FL}, it is simple to check we are reduced to Proposition \ref{prop:lff-FL-range}.
\end{proof}

We spend the next few sections preparing for, and then proving, Propositions \ref{prop:D crys and FL} and \ref{prop:lff-FL-range}.

\subsection{Several functors of Tsuji}

A key to proving Proposition \ref{prop:D crys and FL} are certain results in \cite{Tsu20}, which we now describe. Fix a small framed $W$-algebra $R$. Let us write $\Gamma_R=\pi_1^\et(R[\nicefrac{1}{p}])$ (with respect to some base point), so that $\cat{Rep}^\mr{cont.}_{\mathbb{Z}_p}(\Gamma_R)=\cat{Loc}_{\bb{Z}_p}(\Spf(R[\nicefrac{1}{p}])_\eta)$ (cf.\@ \cite[\S2.1.4]{IKY1}).

To describe Tsuji's results, we first must describe certain subcategories of $\cat{Vect}^{\varphi,\mr{lff}}_{[0,p-2]}(R_\smallprism)$ and $\cat{VectF}^{\varphi,\mr{sd}}_{[0,p-2]}(R_\crys)$. First, recall (see Proposition \ref{prop:filtered-F-crystals}) that there is an equivalence of categories
\begin{equation*}
\cat{VectF}^{\varphi,\mr{sd}}_{[0,p-2]}(R_\crys)\isomto \cat{MF}^\nabla_{[0,p-2]}(R),\qquad (\mc{F},\varphi_\mc{F},\Fil^\bullet_\mc{F})\mapsto (\mc{F}_\mf{X}(R),\varphi_{\mc{F}_\mf{X}(R)},\nabla_\mc{F},\Fil^\bullet_\mc{F}). \end{equation*}
Following \cite[\S4]{Tsu20}, we consider the full subcategory $\cat{MF}^\nabla_{[0,p-2],\text{free}}(R)$ consisting of those $(M,\varphi_M,\nabla_M,\Fil^\bullet _M)$ such that $\Gr^r(\Fil^\bullet_M)$ is free over $R$ for every $r\in \Z$. 
Then, unraveling the definition of $\D_\crys$ we obtain a functor 
\begin{equation*}
    \cat{Vect}^\varphi_{[0,p-2],\mr{free}}(R_\smallprism)\to \cat{MF}^\nabla_{[0,p-2],\text{free}}(R),
\end{equation*}
which we denote again by $\D_\crys$.

Now, for notational simplicity, we again use abbreviations:
\begin{equation*}
    \Ainf\defeq \Ainf(\check{R}),\quad \Acrys\defeq \Acrys(\check{R}),
\end{equation*}
with notation as in \cite[\S1.1.5]{IKY1}.  We then further consider the categories
\begin{equation*}
     \cat M^{\tilde\xi,\mr{cont}}_{[0,p-2],\text{free}}(\Ainf,\varphi,\Gamma_R), \quad
    \cat{MF}^{\tilde\xi,\mr{cont}}_{[0,p-2],\text{free}}(\Ainf,\varphi,\Gamma_R),\quad
    \cat{MF}^{p,\mr{cont}}_{[0,p-2],\text{free}}(\Acrys,\varphi,\Gamma_R),
\end{equation*} 
from \cite[Definition 51 and \S8]{Tsu20}.\footnote{In \cite[\S2 and \S8]{Tsu20}, the element $\tilde\xi$ of $\Ainf$ is denoted by $q$.} By \cite[Equation (49) and Proposition 59]{Tsu20}, we have
\begin{equation}\label{eq: filtered phi modules with Galois}
\cat M^{\tilde\xi,\mr{cont}}_{[0,p-2],\text{free}}(\Ainf,\varphi,\Gamma_R) 
\isomfrom \cat{MF}^{\tilde\xi,\mr{cont}}_{[0,p-2],\text{free}}(\Ainf,\varphi,\Gamma_R)
\isomto\cat{MF}^{p,\mr{cont}}_{[0,p-2],\text{free}}(\Acrys,\varphi,\Gamma_R).    
\end{equation}
The first functor is that forgetting the filtration, and the second is defined by 
\begin{equation*}
    (M,\Fil^\bullet_M,\varphi_M)\mapsto \left((M,\Fil^\bullet_M)\otimes_{(\Ainf,\Fil^\bullet_\xi)}(\Acrys,\Fil^\bullet_\mr{PD}),\varphi_M\otimes 1\right)
\end{equation*}
with the semi-linearly extended action of $\Gamma_R$.

\begin{construction}[{\cite[Equations (23) and (36)]{Tsu20}}]Define functors
\begin{equation*}
\begin{aligned}T\Acrys &\colon \cat{MF}^\nabla_{[0,p-2],\text{free}}(R)\to \cat{MF}^{p,\mr{cont}}_{[0,p-2],\text{free}}(\Acrys,\varphi,\Gamma_R),
\\T\Ainf &\colon\cat{MF}^\nabla_{[0,p-2],\text{free}}(R)\to \cat{M}^{\tilde\xi,\mr{cont}}_{[0,p-2],\text{free}}(\Ainf,\varphi,\Gamma_R),\end{aligned}
\end{equation*}
as follows. 
Let $(M,\varphi_M,\nabla_M,\Fil^\bullet _M)$ be an object of $\cat{MF}^\nabla_{[0,p-2]\text{free}}(R)$, and $(\mc{F},\varphi_\mc{F},\Fil^\bullet_\mc{F})$ denote the corresponding object of $\cat{VectF}^\varphi(\mf{X}_\crys)$. Then, $T\Acrys(M,\varphi_M,\nabla_M,\Fil^\bullet_M)$ is defined as the evaluation $(\mc{F},\varphi_\mc{F},\Fil^\bullet_\mc{F})(\Acrys\twoheadrightarrow\check{R})$, and $T\Ainf$ is the result of translating $T\Acrys$ through \eqref{eq: filtered phi modules with Galois}. 
\end{construction}

We make analogous constructions for prismatic $F$-crystals as follows.

\begin{construction} 
We have a natural functor
\be
\D \Ainf\colon \cat{Vect}^\varphi_{[0,p-2],\text{free}}(R_\smallprism)\to \cat M^{\tilde\xi,\mr{cont}}_{[0,p-2],\text{free}}(\Ainf,\varphi,\Gamma_R),\quad (\mc{E},\varphi_\mc{E})\mapsto (\mc{E}(\Ainf,(\tilde{\xi})),\varphi_\mc{E}),
\ee
where we equip $\mc{E}(\Ainf,(\tilde{\xi}))$ with the $\Gamma_R$-action induced by the $\Gamma_R$-action on $(\Ainf,(\tilde{\xi}))$. That $(\mc{E}(\Ainf,(\tilde{\xi})),\varphi_\mc{E})$ is an object of $M^{\tilde\xi,\mr{cont}}_{[0,p-2],\text{free}}(\Ainf,\varphi,\Gamma_R)$ follows from Lemma \ref{lem: free BK mod}.

Let $\D^\mathrm F\Ainf$ and $\D \Acrys$ denote the compositions
\begin{eqnarray*}
\cat{Vect}^\varphi_{[0,p-2],\mr{free}}(R_\smallprism)\xrightarrow{\bb{D}\Ainf} \cat{M}^{\tilde\xi,\mr{cont}}_{[0,p-2],\text{free}}(\Ainf,\varphi,\Gamma_R)\isomto \cat{MF}^{\tilde\xi,\mr{cont}}_{[0,p-2],\text{free}}(\Ainf,\varphi,\Gamma_R),
\\
\cat{Vect}^\varphi_{[0,p-2],\mr{free}}(R_\smallprism)\xrightarrow{\bb{D}\Ainf}  \cat{M}^{\tilde\xi,\mr{cont}}_{[0,p-2],\text{free}}(\Ainf,\varphi,\Gamma_R)\isomto \cat{MF}^{p,\mr{cont}}_{[0,p-2],\text{free}}(\Acrys,\varphi,\Gamma_R),     
\end{eqnarray*}
respectively. 
\end{construction}

Note that for an object $(\mc E,\varphi_\mc E)$ of $\cat{Vect}^\varphi_{[0,p-2],\mr{free}}(R_\smallprism)$, the underlying $\varphi$-module of $\D \Acrys(\mc E,\varphi_\mc{E})$ is given by $\mc E(\Acrys,(p))$. The action of $\Gamma_R$ is induced by that on the object $(\Acrys,(p),\wt{\mr{nat}}.)$ of $R_\smallprism$. The filtration on these objects is described as follows.

\begin{lem}\label{lem: filtration on DA crys}
    Let $(e_\nu,r_\nu)_{\nu=1}^n$ be a filtered basis of $(\phi^*\mc E(\mf S_R,(E)),\Fil^\bullet_\mr{Nyg})$ over the filtered ring $(\mf S_R,\Fil^\bullet_E)$. 
    Then we have:
    \begin{enumerate}
        \item $(\alpha^*_\mr{inf}(e_\nu),r_\nu)_{\nu=1}^n$ is a filtered basis of $\D^\mathrm F \Ainf(\mc E,\varphi_\mc{E})$ over the filtered ring $(\Ainf,\Fil^\bullet_{\xi})$, 
        \item $(\alpha^*_\crys(e_\nu),r_\nu)_{\nu=1}^n$ is a filtered basis of $\D \Acrys(\mc E,\varphi_\mc{E})$ over the filtered ring $(\Acrys,\Fil^\bullet_{\mr{PD}})$. 
    \end{enumerate}
\end{lem}

\begin{proof}
    Claim (2) follows from (1), and (1) follows from the description of the equivalence 
    \be\cat{M}^{\tilde\xi,\mr{cont}}_{[0,p-2],\text{free}}(\Ainf,\varphi,\Gamma_R)\isomto \cat{MF}^{\tilde\xi,\mr{cont}}_{[0,p-2],\text{free}}(\Ainf,\varphi,\Gamma_R)\ee 
    given in the proof of \cite[Lemma 46]{Tsu20} combined with Lemma \ref{lem: free BK mod}.
\end{proof}

\subsection{Proof of Proposition \ref{prop:D crys and FL}} To begin, we consider the functor
\begin{equation*}
    T^\ast_\inf\colon \cat{M}^{\tilde{\xi,\mr{cont}}}_{[0,p-2],\text{free}}(\Ainf,\varphi,\Gamma_R)\to \cat{Rep}^\mr{cont.}_{\Z_p}(\Gamma_R),\quad (M,\varphi_M)\mapsto \text{Hom}((M,\varphi_M),(\Ainf,\phi)),
\end{equation*}
where $\Gamma_R$ acts on $\text{Hom}((M,\varphi_M),(\Ainf,\phi))$ via its action on $(M,\varphi_M)$. By \cite[Theorem 63 (2)]{Tsu20}, the composition $T^*_{\inf}\circ T\Ainf$ is identified with $T^*_\crys$. Thus, the proof of Proposition \ref{prop:D crys and FL} is reduced to showing that the following diagram of categories commutes:
\begin{equation}\label{eq:D crys T crys}
    \xymatrix{
&\cat{Vect}^\varphi_{[0,p-2],\text{free}}(R_\smallprism)
\ar[dl]_-{T^*_\et}\ar[d]^-{\D A_\mathrm{inf}}\ar[dr]^-{\D_\crys}
\\ \cat{Rep}_{\Z_p}(\Gamma_R) 
&\cat{M}^{\tilde{\xi},\mr{cont}}_{[0,p-2],\text{free}}(\Ainf,\varphi,\Gamma_R)
\ar[l]^-{T^*_{\inf}} \ar[d]^-\wr
& \cat{MF}^\nabla_{[0,p-2],\text{free}}(R)
\ar[l]^-{T\Ainf}\ar[ld]^-{T\Acrys}
\\ &\cat{MF}^{p,\mr{cont}}_{[0,p-2],\text{free}}(\Acrys,\varphi,\Gamma_R).
}
\end{equation}
Moreover, the lower-right triangle of \eqref{eq:D crys T crys} commutes by the definition of $T\Ainf$.

\begin{lem}\label{lem:right commutes}
    The upper-right triangle of \eqref{eq:D crys T crys} commutes.
\end{lem}
\begin{proof}
    It suffices to show that the right large triangle in \eqref{eq:D crys T crys} commutes. Moreover, note that by definition the composition of the two vertical arrows is precisely $\bb{D}\Acrys$. 
    
    So, let $(\mc E,\varphi_\mc{E})$ be an object of $\cat{Vect}^\varphi_{[0,p-2],\text{free}}(R_\smallprism)$. 
    We first check that the underlying $(\varphi,\Gamma_R)$-modules of $(T\Acrys\circ\D_\crys)(\mc E,\varphi_\mc{E})$ and $\D \Acrys(\mc E,\varphi_\mc{E})$ are canonically identified. 
    The latter is by definition $\mc E(\Acrys,(p))$ with $\Gamma_R$-action induced by that on $(\Acrys,(p),\wt{\mr{nat}}.)$ via its normal action on $\Acrys$. 
    To describe $(T\Acrys\circ\D_\crys)(\mc E,\varphi_\mc{E})$, recall that the underlying $F$-crystal of $\bb{D}_\crys(\mc{E},\varphi_\mc{E})$ is $(\mc E^\crys,\varphi_{\mc{E}^\crys})$. 
    So, by definition, the underlying $(\varphi,\Gamma_R)$-module of $(T\Acrys\circ\D_\crys)(\mc E,\varphi_\mc{E})$ is $\mc E^\crys(\Acrys\twoheadrightarrow \check{R}/p)$  with $\Gamma_R$-action induced by that on $\Acrys\twoheadrightarrow \check{R}/p$ via its normal action on $\Acrys$. But, these two coincide by the definition of $(-)^\crys$.

    We now compare the filtrations. We put $\mc M\defeq\mc E(S_R,(p))$ which is naturally isomorphic to $\mc E^\crys(S_R\twoheadrightarrow R)$ (cf.\@ Proposition \ref{prop: pris crys in char p}). Let $\Fil^\bullet_1$ denote the filtration on $\mc M$ obtained from the Nygaard filtration on $\phi^*\mc E(\mf S_R,(E))$ by filtered scalar extension along $(\mf S_R,\Fil^\bullet_E)\to (S_R,\Fil^\bullet_\mr{PD})$, where $\mf S_R\to S_R$ is the natural inclusion. On the other hand, consider the filtration $\Fil^\bullet_2$ on $\mc M$ obtained from the filtration on $\D_\crys(\mc E)(R)$ via the filtered map $(R,\Fil^\bullet_\triv)\to (S_R,\Fil^\bullet_\mr{PD})$, where $R\to S_R$ is the natural inclusion. 
    Note that the filtration on $\D A_{\crys}(\mc E,\varphi_\mc{E})=\mc E(A_{\crys},(p))$ (resp.\@ $(T\Acrys\circ\D_\crys)(\mc E,\varphi_\mc{E})=\mc E(\Acrys\twoheadrightarrow\check{R})$) is obtained from $\Fil^\bullet_1$ (resp.\@ $\Fil^\bullet_2$) by scalar extension along the faithfully flat map $S_R\to \Acrys$. {But we have the equality $\Fil^\bullet_1=\Fil^\bullet_2$ by Proposition \ref{filtered commutativity of Breuil prism diagram}. Thus, the assertion follows.}
\end{proof}

\begin{rem}\label{rem:filtered-tensor-product-computation}The method of proof in Lemma \ref{lem:right commutes} shows that if $R$ is a base $W$-algebra, and $(\mc{E},\varphi_\mc{E})$ an object of $\cat{Vect}^{\varphi,\mr{lff}}(R_\smallprism)$, then there is an identification of filtered Frobenius modules
\begin{equation*}
\begin{aligned} \bb{D}_\crys(\mc{E},\varphi_\mc{E})(\Acrys(\wt{R})\to\wt{R}) &\isomto (\phi^\ast\mc{E}(\Ainf(\wt{R}),(\xi)),\Fil^\bullet_\mr{Nyg})\otimes_{(\Ainf(\wt{R}),\Fil^\bullet_\xi)}(\Acrys(\wt{R}),\Fil^\bullet_\mr{PD}),\\
\bb{D}_\crys(\mc{E},\varphi_\mc{E})(S_R,(p)) &\isomto (\phi^\ast\mc{E}(\mf{S}_R,(E)),\Fil^\bullet_\mr{Nyg})\otimes_{(\mf{S}_R,\Fil^\bullet_E)}(S_R,\Fil^\bullet_\mr{PD}).
\end{aligned}
\end{equation*}
\end{rem}

\begin{lem}\label{lem:left commutes}
    The upper left triangle of the diagram \eqref{eq:D crys T crys} commutes. 
\end{lem}
\begin{proof}
    Let $(\mc E,\varphi_\mc{E})$ be an object of $\cat{Vect}^\varphi_{[0,p-2],\text{free}}(R_\smallprism)$. By definition, we have
    \be
    (T^*_{\inf}\circ \bb{D}\Ainf)(\mc{E},\varphi_{\mc{E}})\cong \Hom((\mc{E}(\Ainf,(\tilde{\xi})),\varphi_\mc{E}),(\Ainf,\phi)).
    \ee
    On the other hand, by \cite[Example 2.15]{IKY1}, we have 
    \be
    T^*_\et(\mc E,\varphi_\mc{E})\cong \Hom((\mc E(\Ainf,(\tilde{\xi})),\varphi_\mc{E}),(\Ainf[\nicefrac{1}{\widetilde{\xi}}]^\wedge_p,\phi)).
    \ee
     The obvious map $(T^*_{\inf}\circ \bb{D}\Ainf)(\mc E,\varphi_\mc{E})\to T^*_\et(\mc E,\varphi_\mc{E})$ is an isomorphism. Indeed, it is an injective map between free $\Z_p$-modules of the same rank by \cite[Proposition 66]{Tsu20} and Lemma \ref{lem:right commutes}. 
     So, the cokernel is killed by a power of $p$, but also embeds into $\Hom(\mc{E}(\Ainf,(\tilde{\xi})),\Ainf[\nicefrac{1}{\widetilde{\xi}}]^\wedge_p/\Ainf)$, which is $p$-torsionfree as $\Ainf[\nicefrac{1}{\widetilde{\xi}}]^\wedge_p/\Ainf$ is, and so the cokernel is zero as desired.
\end{proof}

With these observations, the proof of Proposition \ref{prop:D crys and FL} is now an exercise in parts assembly.

\begin{proof}[Proof of Proposition \ref{prop:D crys and FL}]
    Let $(\mc E,\varphi_\mc{E})$ be an object of $ \cat{Vect}^{\varphi,\mr{lff}}_{[0,p-2]}(\mf X_\smallprism)$.  
     By taking an open covering $(\mf U_i)_i$ by small affine opens of $\mf X$ such that $\mc E_{\mf U_{i,\smallprism}}$ is in $\cat{Vect}^\varphi_{[0,p-2],\text{free}}(\mf U_{i,\smallprism})$, the proof is reduced to constructing, for a small affine formal scheme $\mf X=\Spf(R)$ over $W$ and an object $(\mc E,\varphi_\mc{E})$ of $\cat{Vect}^\varphi_{[0,p-2],\text{free}}(R_\smallprism)$, an isomorphism
    \be
     (T^*_\crys\circ\D_\crys)(\mc E,\varphi_\mc{E})\isomto T^*_\et(\mc E,\varphi_\mc{E})
    \ee
    functorial in $(\mc E,\varphi_\mc{E})$ and compatible with open immersions $\Spf (R')\to \Spf (R)$. 
    Such an isomorphism is obtained using Lemmas \ref{lem:right commutes} and \ref{lem:left commutes}, together with \cite[Theorem 63.(2)]{Tsu20}, which is seen to satisfy the desired compatibility.
\end{proof}

\subsection{Proof of Proposition \ref{prop:lff-FL-range}} We now proceed to the proof of Proposition \ref{prop:lff-FL-range}. Our approach is inspired by techniques developed in \cite[\S3.3]{Hokaj}, but is significantly simpler (e.g., does not require reduction to $k_g$ as in loc.\@ cit.\@).

\begin{proof}[Proof of Proposition \ref{prop:lff-FL-range}] By Proposition \ref{prop:filtered-equiv} we may assume that $\mf{X}=\Spf(R)$ for a formally framed $W$-algebra $R$. Moreover, by the same result it suffices to show that $\mathrm{Gr}^i(\Fil^\bullet_{\mathbb{D}_\mr{crys}})$ is a locally free $R$-module for all $i$. As $\mathbb{D}_\mr{crys}$ is compatible with flat base change, and one may check local freeness at $\wh{R}_\mf{p}$ for all maximal ideals $\mf{p}$ of $R/p$, we may further reduce to the power-series case, i.e., when $R=W\llbracket t_1,\ldots,t_d\rrbracket$ for some $d$.

Set $(\mf{M},\varphi_\mf{M})=\mc{E}(\mf{S}_R,(E))$ and let $\ol{\mathrm{Fil}}^\bullet\subseteq \phi_{\mf{S}_R}^\ast\mf{M}/E$ be the image of the Nygaard filtration $\mathrm{Fil}^\bullet\subseteq \phi_{\mf{S}_R}^\ast\mf{M}$. We must show that $\mr{Gr}^i(\ol{\Fil}^\bullet)$ is a free $R$-module for all $i$. The argument in \cite[Lemma 52]{Hokaj} shows that the map $f_i\colon \phi_{\mf{S}_R}^\ast\mf{M}\to\mf{M}/E^i$ induced by the relative Frobenius map $\varphi_\mf{M}$ has cokernel which is a finite free $R$-module. Using this, we claim that if $\mf{M}_0$ is the Breuil--Kisin module over $W$ given by $\mf{M}\otimes_{\mf{S}_R} \mf{S}_W$, with filtration $\Fil_0^\bullet\subseteq  \phi_{\mf{S}_W}^\ast \mf{M}_0$, then the natural map $\mr{Fil}^\bullet\otimes_{\mf{S}_R}\mf{S}_W\to \Fil^\bullet_0$ is an isomorphism. Indeed, for each $i$ we have the short exact sequence
\begin{equation*}
0\to \phi_{\mf{S}_R}^\ast\mf{M}/\mr{Fil}^i\xrightarrow{f_i} \mf{M}/E^i\to \mr{coker}(f_i)\to 0.
\end{equation*}
As $\mr{coker}(f_i)$ is a free $R$-module, we may tensor this along $R\to W$ to obtain the exact sequence
\begin{equation*}
0\to (\phi_{\mf{S}_R}^\ast\mf{M}/\mr{Fil}^i)\otimes_R W\to (\mf{M}/E^i)\otimes_R W\to \mr{coker}(f_i)\otimes_R W\to 0.
\end{equation*}
The injectivity in the above exact sequence implies the claimed isomorphism $\mr{Fil}^\bullet\otimes_{\mf{S}_R}\mf{S}_W\cong \Fil^\bullet_0$. 
By \cite[Lemma 3.8]{GaoFLstr}, $\mr{Gr}^i(\ol{\Fil}_0^\bullet)$ is finite free over $W(k)$. Let $r_i$ be the rank of $\mr{Gr}^i(\ol{\Fil}_0^\bullet)$ over $W(k)$. By Nakayama's lemma, we can lift $W(k)^{r_i} \cong \mr{Gr}^i(\ol{\Fil}_0^\bullet)$ to a surjection $R^{r_i} \twoheadrightarrow \mr{Gr}^i(\ol{\Fil}^\bullet)$. 
Lifting those surjections to $R^{r_i} \to \ol{\Fil}^i$, we obtain a surjection $\bigoplus_{j \geq i} R^{r_j} \twoheadrightarrow \ol{\Fil}^i$. In particular, we have $\bigoplus_{j \geq 0} R^{r_j} \twoheadrightarrow \phi_{\mf{S}_R}^\ast\mf{M}/E$. Since the source and target of the last surjection have the same rank, it is an isomorphism. This implies the injectivity of $\bigoplus_{j \geq i} R^{r_j} \twoheadrightarrow \ol{\Fil}^i$. Hence this is an isomorphism, which implies $R^{r_i} \cong \mr{Gr}^i(\ol{\Fil}^\bullet)$. 
\end{proof} 

\subsection{Cohomological applications} Finally, we record some cohomological applications of the discussion above. To this end, fix $\ms{X}\to\Spec(W)$ to be smooth and projective.

For $(\mc{E},\varphi_\mc{E})$ in $\cat{Vect}^{\varphi,\mr{lff}}(\wh{\ms{X}}_\smallprism)$ consider the object $T_\et(\mc{E},\varphi_\mc{E})^\alg$ of $\cat{Loc}_{\Z_p}(\ms{X}_K)$ (see \cite[\S2.1.3]{IKY1}). We make the following assumption:
\begin{itemize}
    \item[] \textbf{(Assumption TF)} The $W$-module $H^i((\mathscr{X}_k/W)_\mr{crys},\bb{D}_\crys(\mc{E},\varphi_\mc{E}))$ is $p$-torsionfree. 
\end{itemize}
Note then that by \cite[Theorem 1.1]{BMSI} we also have that $H^i_\mathrm{et}(X_{\overline{K}},T_\et(\mc{E},\varphi_\mc{E}))$ is $p$-torsionfree. Then, for $i\geqslant 0$, we have a diagram
\begin{equation}\label{eq:Dcrys-comparison-isom}
    \xymatrixrowsep{1pc}\xymatrixcolsep{4pc}\xymatrix{D_\crys\left(H_\et^i(\ms{X}_{\ov{K}},T_\et(\mc{E},\varphi_\mc{E})^\alg[\nicefrac{1}{p}])\right)\ar[r]^-{c_{(\mc{E},\varphi_\mc{E})}} & H^i\left((\ms{X}_k/W)_\crys, \bb{D}_\crys(\mc{E},\varphi_\mc{E})[\nicefrac{1}{p}]\right)\\ \bb{D}_\crys\left(H_\et^i(\ms{X}_{\ov{K}},T_\et(\mc{E},\varphi_\mc{E})^\alg)\right)\ar@{}[u]|{\rotatebox{90}{$\subseteq$}} & H^i\left((\ms{X}_k/W)_\crys, \bb{D}_\crys(\mc{E},\varphi_\mc{E})\right)\ar@{}[u]|{\rotatebox{90}{$\subseteq$}}}
\end{equation}
where $c_{(\mc{E},\varphi_\mc{E})}$ is the isomorphism obtained by combining \cite[Theorem 5.5]{TanTong}, Theorem  \ref{thm:integral-dcrys-strongly-divisible-and-rational-agreeance}, and \cite[Theorem 3.7.2]{HuberEC}, and we are using notation as in Example \ref{ex:Kisin-Dcrys-II}. 

From Proposition \ref{prop:D crys and FL}, Proposition \ref{prop:lff-FL-range}, and \cite[Proposition 5.4]{Faltings89} we deduce the following. 

\begin{cor}\label{cor: FL coh}
    Let $\ms{X}\to\Spec(W)$ be smooth and projective. Then, for an object $(\mc E,\varphi_\mc{E})$ of $\cat{Vect}^{\varphi}_{[0,a]}(\widehat{\ms X}_\smallprism)$ and $b$ in $\mbb N$ with $a+b< p-2$, 
    the Galois representation $H^b_\et(\ms{X}_{\ov{K}},T_\et(\mc{E},\varphi_\mc{E})^\alg)$ and the Fontaine--Laffaille module  $H^b\left((\ms{X}_k/W)_\crys, \bb{D}_\crys(\mc{E},\varphi_\mc{E})\right)$ \emph{(}which are possibly torsion\emph{)} correspond via the functor constructed in \emph{\cite[Th\'eor\`eme (0.6)]{FL82}}.    
\end{cor}

We end this section by recording a related compatibility of cohomology. 

\begin{prop}
    Let $f\colon \mf{X}\to\Spf(W)$ be smooth and proper. Then, for an object $\mc{F}$ of $\cat{Perf}({\mf X}^\syn)$, there is a canonical quasi-isomorphism
    \begin{equation*}
         Rf_{K,*}T_\et(\mc{F})\isomto T_\et\left(Rf^\syn_\ast\mc{F}\right)
    \end{equation*}
of complexes of Galois representations. 
\end{prop}

In the proof, we apply \cite[Appendices A and B]{Hauck} to $f\colon \mf X\to \Spf(W)$ even though in loc.\ cit., the target is assumed to be $\Spf(\Z_p)$: replacing $\Z_p$ with $W$ does not affect the argument there.  
\begin{proof}
    Recall first from \cite[Remark 6.3.4]{BhattNotes} that we have a functor $\cat{Perf}(\mf X^\syn)\to \cat{Perf}^\varphi(\mf X_\smallprism)$. 
    Note that the push-forward $Rf^\syn_*$ preserves perfect complexes by \cite[Proposition B.0.1]{Hauck}. 
    Then, by \cite[Theorem 1.10 (i)]{GuoReinecke} (\cf the diagram (1) in loc.\ cit.),
    the proof is reduced to showing that the push-forward operations for prismatic $F$-gauges and prismatic $F$-crystals are compatible: more precisely, we need to show that the diagram 
    \bx{
    \cat{Perf}(\mf X^\syn) \ar[r]^-{}\ar[d]_-{Rf^\syn_*}
    & \cat{Perf}^\varphi(\mf X_\smallprism) \ar[d]^{Rf_{\smallprism,*}}
    \\ \cat{Perf}(W^\syn) \ar[r]^-{}
    & \cat{Perf}^\varphi(W_\smallprism) 
    }\ex
    commutes. 
    By \cite[Corollary A.0.4 and Lemma A.0.6]{Hauck}, 
    we know that the diagram commutes after forgetting the Frobenius structures. 
    
    We now check that the two Frobenius structures agree. Let $\mc F$ be an object of $\cat{Perf}(\mf X^\syn)$, and $\mc F^\mc N$ (resp.\ $\mc F^\smallprism$) denote its pullback to the stack $\mf X^\mc N$ (resp.\ $\mf X^\smallprism$). 
    Recall that the push-forward of the prismatic $F$-crystals associated to $\mc F$ is defined as follows. 
    Firstly, the underlying prismatic $F$-crystal is defined to be (what corresponds to) the push-forward $Rf^\smallprism_*\mc F^\smallprism$. 
    The Frobenius structure is defined as the composite of the base change map $\phi^*Rf^\smallprism_*\mc F^\smallprism \to Rf^\smallprism_*\phi^*\mc F^\smallprism$ with $\mc I_\smallprism$ inverted (on the prismatic site $W_\smallprism$), which is shown to be isomorphic in the proof of \cite[Theorem 8.1]{GuoReinecke}, with the push-forward of the Frobenius structure $Rf_{\smallprism,*}(\varphi_\mc E)$. 
    
    Thus, by construction of the associated prismatic $F$-crystal in \cite[Remark 6.3.4]{BhattNotes} (and \cite[Lemma A.0.6]{Hauck}), 
    it suffices to show the commutativity of the diagram in $\cat{Perf}(W^\smallprism)$,
    \bx{
        j^*_\dR Rf^\mc N_*\mc F^\mc N \ar[r]\ar[d]
        & j^*_\mr{HT} Rf^\mc N_*\mc F^\mc N \ar[d]
        & \phi^*R\pi_* Rf^\mc N_*\mc F^\mc N \ar[l]\ar[r]\ar[d]
        & \phi^*j_\dR^*Rf^\mc N_*\mc F^\mc N \ar[d]
        \\ Rf^\smallprism_*j^*_\dR \mc F^\mc N \ar[r]
        & Rf^\smallprism*j^*_\mr{HT}\mc F^\mc N 
        & Rf^\smallprism_*\phi^*R\pi_*\mc F^\mc N \ar[l]\ar[r]
        & Rf^\smallprism_*\phi^*j_\dR^*\mc F^\mc N,
    }\ex
    where the vertical maps are the base change maps; 
    the leftmost horizontal arrows come from the Frobenius structures on the prismatic $F$-gauges $Rf^\syn_*\mc F$ and $\mc F$; 
    the rest of the top horizontal arrows are the maps $a$ and $b$ from \cite[Remark 6.3.4]{BhattNotes} for $Rf^\mc N_*\mc F^\mc N$; 
    the rest of the bottom horizontal arrows are the $Rf^\smallprism_*$-push-forward of the maps $a$ and $b$ for $\mc F^\mc N$. 
    
    The left square commutes by the definition of the push-forward functor $Rf^\syn_*$. 
    The other two squares commute as the horizontal maps are adjunction maps. 
    Noting that the leftmost and rightmost vertical arrows are isomorphisms by \cite[Corollary A.0.4]{Hauck}, this implies that the two Frobenius structures agree. 
\end{proof}

\begin{cor}
    \label{prop: Gauge coh}
    Let $f\colon \mf{X}\to\Spf(W)$ be smooth and proper. Then, for an object $\mc{V}$ of $\cat{Vect}({\mf X}^\mr{syn})$, there is a canonical isomorphism
    \begin{equation*}
         H^i_\et\left(\mf{X}_{C},T_\et(\mc{V})/p^n\right)\isomto T_\et\left(R^if^\syn_\ast(\mc{V}/p^n)\right),
    \end{equation*}
of Galois representations, for any $n$ in $\bb{N}\cup\{\infty\}$ (where by convention $p^\infty=0$).
\end{cor}


\section{Relationship to Dieudonn\'e theory}\label{ss:relationship-to-DD-theory}
In this final section we discuss how $\bb{D}_\mr{crys}$ can be used to unite various Dieudonn\'e theories that appear in the literature. Throughout we use notation and terminology from \hyperref[notation-and-terminology]{Notation and terminology} without comment but now assume further that $p$ is odd. Furthermore, for a $p$-adically complete ring $S$ we denote by $\cat{BT}_p(S)$ the category of $p$-divisible groups over $S$.\footnote{See \cite[Lemma 2.4.4]{deJongCrystalline} for why this notation is not ambiguous.}

\subsection{The prismatic Dieudonn\'e functor} Let $S$ be a quasi-syntomic ring. Fix an object $H$ of $\cat{BT}_p(S)$. We may then consider the sheaf\footnote{To see that this is a sheaf, we combine the following observations: $\ov{\mc{O}}_\smallprism$ is a sheaf and $H$ is finitely continuous (because $H=\varinjlim H[p^n]$, with each $H[p^n]$ representable, and finite limits commute with filtered colimits in $\cat{Set}$).\label{footnote:underline-H-sheaf}}
\begin{equation*}
    H_{\ov{\mc{O}}_\smallprism}\colon S_\smallprism\to\cat{Grp},\qquad (A,I)\mapsto H(A/I)=H(\ov{\mc{O}}_\smallprism(A,I)).
\end{equation*}
We then have the following construction of Ansch\"{u}tz--Le Bras.

\begin{defn}[{\cite{AnschutzLeBrasDD}}] The \emph{prismatic Dieudonn\'e crystal} associated to $H$ is
\begin{equation*}
    \mc{M}_\smallprism(H)\defeq \mc{E}xt^1_{\cat{Ab}(S_\smallprism)}(H_{\ov{\mc{O}}_\smallprism},\mc{O}_\smallprism),
\end{equation*}
which has the structure of an $\mc{O}_\smallprism$-module on $S_\smallprism$ and a Frobenius morphism
\begin{equation*}
    \varphi_{\mc{M}_\smallprism(H)}\colon \phi^\ast\mc{M}_\smallprism(H)\to \mc{M}_\smallprism(H),
\end{equation*}
inherited from those structures on $\mc{O}_\smallprism$.
\end{defn}

The prismatic Dieudonn\'e crystal is a complete invariant of $H$.

\begin{thm}[{\cite{AnschutzLeBrasDD}}]\label{thm:ALB-equiv} The functor $\mc{M}_\smallprism$ defines a contravariant fully faithful embedding
\begin{equation*}
    \mc{M}_\smallprism\colon\cat{BT}_p(S)\to \cat{Vect}^\varphi_{[0,1]}(S_\smallprism),
\end{equation*}
which is an anti-equivalence if $S$ admits a quasi-syntomic cover $S\to \wt{S}$ with $\wt{S}$ perfectoid.\footnote{When $S$ is Noetherian and $\bb{Z}_p$-flat, this condition is equivalent to $S$ being regular by combining \cite{BIM} and \cite[Proposition 5.8]{AnschutzLeBrasDD}.}
\end{thm}
\begin{proof} By \cite[Theorem 4.74]{AnschutzLeBrasDD} (and \cite[Proposition 1.31]{IKY1}), 
the functor $\mc{M}_\smallprism$ forms an equivalence between $\cat{BT}_p(S)$ and the full subcategory of $\cat{Vect}^\varphi_{[0,1]}(S_\smallprism)$ consisting of so-called admissible objects (see \cite[Definition 4.5]{AnschutzLeBrasDD}). Thus, to prove the second part of the claim, it suffices to show that the existence of such a cover $S\to \wt{S}$ implies any object of $\cat{Vect}^\varphi(S_\smallprism)$ is admissible. As admissibility is clearly a local condition on $S_\qsyn$ (see \cite[Proposition 4.9]{AnschutzLeBrasDD}), it suffices to prove the claim over $\wt{S}$. But, this is the content of \cite[Proposition 4.12]{AnschutzLeBrasDD}.
\end{proof}

Next, we record the compatibility of $\mc{M}_\smallprism$ with the functors $M^\mr{SW}$ and $M^\mr{Lau}$ defined by \cite[Theorem 17.5.2]{ScholzeBerkeley} and \cite[Theorem 9.8]{LauDieu}, respectively. More precisely, suppose that $S$ is perfectoid. Then, there is a natural identification between $M^\mr{SW}(H)$ and $M^\mr{Lau}(H)$, by the unicity part of \cite[Theorem 17.5.2]{ScholzeBerkeley}, and we have the following result of Ansch\"{u}tz--Le Bras. 

\begin{prop}[{cf.\@ \cite[Proposition 4.48]{AnschutzLeBrasDD}}]\label{prop:ALB-Lau-SW} Let $S$ be a perfectoid ring and $H$ an object of $\cat{BT}_p(S)$. Then, we have canonical identifications:
\begin{equation*}
    M^\mr{SW}(H)^*=M^\mr{Lau}(H)=\mc{M}_\smallprism(H)(\Ainf(S),(\tilde{\xi})). 
\end{equation*}
\end{prop}

\subsection{Crystalline Dieudonn\'e theory}\label{ss:crystalline-DD-theory} We now recall the classical filtered crystalline Dieudonn\'e theory of Grothendieck--Messing. Fix $\mf{X}=\Spf(R)$ to be a formally framed base $W$-scheme.

\begin{defn}[{cf.\@ \cite[Definition 3.1]{KimBK}}]A \emph{filtered Dieudonn\'e crystal} on $R$ is an object $(\mc{F},\varphi_\mc{F},\Fil^\bullet_\mc{F})$ of $\cat{VectNF}_{[0,1]}^\varphi(\mf{X}_\crys)$ where:
\begin{itemize}[leftmargin=.3in]
\item $(\mc{F},\varphi_\mc{F})$ is an effective $F$-crystal,
\item $\Fil^1_\mc{F}$ is a direct factor of $\mc{F}_{\mf{X}}$, 
\item $\phi^\ast(\Fil^1_\mc{F})\otimes_{\mc{O}_\mf{X}}\mc{O}_{\mf{X}_k}$ is equal to the kernel of $\varphi_\mc{F}\colon \phi^\ast\mc{F}_{\mf X_k}\to\mc{F}_{\mf X_k}$, 
\item and there exists $V\colon \mc{F}_{\mf{X}_k}\to \phi^\ast\mc{F}_{\mf{X}_k}$ with $\varphi_\mc{F}\circ V=[p]_{\mc{F}_{\mf{X}_k}}$ and $V\circ\varphi_\mc{F}=[p]_{\phi^\ast\mc{F}_{\mf{X}_k}}$,
\end{itemize}
the category of which we denote $\cat{DieuF}(R)$. 
\end{defn}

\begin{rem} One may check that $\cat{DieuF}(R)=\cat{VectF}^{\varphi,\mr{sd}}_{[0,1]}(R_\crys)$, and we only prefer the former notation/terminology when discussing $p$-divisible groups for historical reasons. In particular, by this equality and the discussion in \S\ref{ss:category-of-filtered-f-crystals}, it makes sense to evaluate the filtration of a filtered Dieudonn\'e crystal on any object of $(R/W)_\crys$.
\end{rem}

Let us now fix an object $H$ of $\cat{BT}_p(R)$. We then consider the sheaf
\begin{equation*}
    H_{\ov{\mc{O}}_\crys}\colon (R/W)_\crys\to \cat{Grp},\quad (i\colon\mf{U}\hookrightarrow\mf{T},\gamma)\mapsto H(\mf{U})=H(\ov{\mc{O}}_\crys(i\colon \mf{U}\hookrightarrow\mf{T},\gamma)),
\end{equation*}
(which is a sheaf as in Footnote \ref{footnote:underline-H-sheaf}). We may then form the sheaf
\begin{equation*}
    \underline{\bb{D}}(H)\defeq \mc{E}xt^1_{\cat{Ab}((R/W)_\crys)}(H_{\ov{\mc{O}}_\crys},\mc{O}_\crys),
\end{equation*}
which comes with the structure of an $\mc{O}_\crys$-module from the second entry. By \cite[Th\'eor\`eme 3.3.3]{BBMDieuII}, the $\mc{O}_\crys$-module $\underline{\bb{D}}(H)$ is an object of $\cat{Vect}(R_\crys)$. It is simple to check that
\begin{equation*}
    (\iota_{0,\infty})_\ast\underline{\bb{D}}(H_k)=\underline{\bb{D}}(H),\qquad \iota_{0,\infty}^\ast\underline{\bb{D}}(H)=\underline{\bb{D}}(H_k),
\end{equation*}
(with notation as in \cite[\S2.3.1]{IKY1}). As in \cite[1.3.5]{BBMDieuII} we have a Frobenius morphism
\begin{equation*}
    \varphi\colon \phi^\ast\underline{\bb{D}}(H_k)\to \underline{\bb{D}}(H_k).
\end{equation*}
So, $\underline{\bb{D}}(H)$ is an object of $\cat{Vect}^\varphi(R_{\crys})$ called the \emph{Dieudonn\'e crystal} associated to $H$. 

One defines a \emph{Hodge filtration}  (see \cite[Corollaire 3.3.5]{BBMDieuII})
\begin{equation*}
    \mathrm{Fil}^1_{H,\text{Hodge}}\defeq \mc{E}xt^1_{\cat{Ab}((R/W)_\crys)}(H_{\ov{\mc{O}}_\crys},\mc{J}_\crys)_{\mf{X}}\subseteq \underline{\bb{D}}(H)_\mf{X},
\end{equation*}
where $\mc{J}_\crys\subseteq\mc{O}_\crys$ is the PD ideal sheaf.
\begin{defn} The functor
\begin{equation*}
    \bb{D}\colon \cat{BT}_p(R)\to \cat{DieuF}(R),\qquad H\mapsto \bb{D}(H)=(\underline{\bb{D}}(H),\mathrm{Fil}^1_{H,\text{Hodge}}),
\end{equation*}
is called the \emph{filtered Dieudonn\'e crystal} functor. 
\end{defn}

The filtered Dieudonn\'e crystal is also a complete invariant of $H$.

\begin{thm}[{de Jong, cf.\@ \cite[Theorem 3.17]{KimBK}}] The functor 
\begin{equation*}
    \bb{D}\colon \cat{BT}_p(R)\to \cat{DieuF}(R),
\end{equation*}
is an anti-equivalence of categories.
\end{thm}

Our first main result is that $\bb{D}_\mr{crys}$ transforms $\mc{M}_\smallprism$ into $\bb{D}$, improving upon \cite[Theorem 4.44]{AnschutzLeBrasDD} which proves the more naive statement obtained by ignoring filtrations.

\begin{thm}\label{thm:ALB-dJ-comparison} There is a natural equivalence $\bb{D}_\crys\circ \mc{M}_\smallprism\isomto \bb{D}$.
\end{thm}
\begin{proof} We begin by observing that there is a tautological identification between $\underline{\bb{D}}_\crys\circ \mc{M}_\smallprism$ and $\underline{\bb{D}}$ (see \cite[Theorem 4.44]{AnschutzLeBrasDD}). 
Thus, it suffices to check that the submodules $\Fil^1_{\bb{D}_\crys}$ and $\Fil^1_{H,\mr{Hodge}}$ agree. To check this, it suffices to pass to the faithfully flat cover $R\to \wt{R}$ (see \cite[Lemma 1.15]{IKY1}). But, observe that we have a commutative diagram of filtered rings
\begin{equation*}
    \xymatrix{(R,\Fil^\bullet_\triv) \ar[r]^-\beta\ar[dr] & (\Acrys(\wt{R}),\Fil^\bullet_\mr{PD})\ar[d]\\ & (\wt{R},\Fil^\bullet_\triv),}
\end{equation*}
so by the crystal property we need only check the equality after evaluation on $\Acrys(\wt{R})\twoheadrightarrow\wt{R}$. 

We first observe that by Proposition \ref{prop:lff-FL-range} and Remark \ref{rem:filtered-tensor-product-computation}, for any object $H$ of $\cat{BT}_p(R)$ there is a canonical identification of filtered $\Acrys(\wt{R})$-modules
\begin{equation*}
     \bb{D}_\crys(\mc{M}_\smallprism(H))(\Acrys(\wt{R})\twoheadrightarrow\wt{R})\isomto \phi^\ast\mc{M}_\smallprism(H)(\Ainf(\wt{R}),(\xi))\otimes_{(\Ainf(\wt{R}),\Fil^\bullet_\xi)}(\Acrys(\wt{R}),\Fil^\bullet_\mr{PD}).
\end{equation*}
On the other hand, by Lemma \ref{lem:[0,1]-fil}, the filtration on $\bb{D}(H)(\Acrys(\wt{R})\twoheadrightarrow\wt{R})$ is equal to the preimage of the filtration $\Fil^1\subseteq \underline{\bb{D}}(H)(\check{R})$, defined by $\bb{D}(H)(\check{R})$, under the surjection
\begin{equation*}
\Pi\colon\underline{\bb{D}}(H)(\Acrys(\check{R})\twoheadrightarrow\check{R})\to \underline{\bb{D}}(H)(\check{R}).
\end{equation*}
 Thus, it suffices to show the following equality of filtered $\Acrys(\check{R})$-modules:
\begin{equation*}
    \phi^\ast\mc{M}_\smallprism(H)(\Ainf(\wt{R}),(\xi))\otimes_{(\Ainf(\wt{R}),\Fil^\bullet_\xi)}(\Acrys(\wt{R}),\Fil^\bullet_\mr{PD})=(\underline{\bb{D}}(H)(\Acrys(\check{R})\twoheadrightarrow\check{R}),\Pi^{-1}(\Fil^1)).
\end{equation*}
But combining Proposition \ref{prop:ALB-Lau-SW} and the definition of $M^\mr{Lau}$ we have
\begin{equation*}
    \begin{aligned} 
    (\underline{\bb{D}}(H)(\Acrys(\wt{R})\to\wt{R}),\Pi^{-1}(\Fil^1)) &= \Phi^\mr{cris}_{\wt{R}}(H_{\wt{R}})\\ &=\lambda^\ast(M^\mr{Lau}(H_{\wt{R}}))\\ &= \lambda^\ast(\phi^\ast\mc{M}_\smallprism(H)(\Ainf(\wt{R}),(\xi)))\\ &=\phi^\ast\mc{M}_\smallprism(H)(\Ainf(\wt{R}),(\xi))\otimes_{(\Ainf(\wt{R}),\Fil^\bullet_\xi)}(\Acrys(\wt{R}),\Fil^\bullet_\mr{PD}),
    \end{aligned}
\end{equation*}
where $\Phi^\mr{cris}_{\wt{R}}$ is as in \cite[Theorem 6.3]{LauDieu} and $\lambda^\ast$ is as in \cite[Proposition 9.3]{LauDieu}.
\end{proof}

\begin{lem}\label{lem:[0,1]-fil}
Let $\mc F$ be an object of $\cat{VectF}_{[0,1]}(R_\crys)$ and $(A\twoheadrightarrow R')\to (B\twoheadrightarrow R')$ be a morphism in $(R/W)_\crys$, with $A\to B$ surjective. 
Then $\Fil^1_{\mc F}(A\twoheadrightarrow R')\subseteq \mc F(A\twoheadrightarrow R')$ is the preimage of the submodule $\Fil^1_{\mc F}(B\twoheadrightarrow R')\subseteq \mc F(B\twoheadrightarrow R')$ via the surjection $ \mc F(A\twoheadrightarrow R')\to \mc F(B\twoheadrightarrow R')$.
\end{lem}
\begin{proof}
    By the crystal property, we may work Zariski locally on $A$. But, Zariski locally on $A$, there is a basis $(e_\nu)_{\nu=1}^n$ of $\mc F(A\twoheadrightarrow R')$ and a subset $I$ of $\{1,\ldots,n\}$ such that  the filtration $\Fil^1(A\twoheadrightarrow R')$ is given by $\sum_{\nu\in I}A\cdot e_\nu+\sum_{\nu\notin I}\Fil^1_{\mr{PD}}(A)\cdot e_\nu$. 
    The claim then follows as $\Fil^1_{\mr{PD}}(A)$ is the preimage of $\Fil^1_{\mr{PD}}(B)$ under the surjection $A\to B$. 
\end{proof}

\begin{rem} When $R$ is small, Theorem \ref{thm:ALB-dJ-comparison} follows from Theorem \ref{thm:big-equiv-diagram}. Indeed, it suffices to show that $T^\ast_\crys(\bb{D}_\crys(\mc{M}_\smallprism(H)))$ is isomorphic to $T^\ast_\crys(\bb{D}(H))$. But, by Proposition \ref{prop:D crys and FL} and \cite[Proposition 3.35]{DLMS} the former is $T_p(H)$, which is also the latter by \cite[Corollary 5.3 and \S5.4]{KimBK}.
\end{rem}

\subsection{Breuil--Kisin--Kim Dieudonn\'e theory} In this final section we use Theorem \ref{thm:ALB-dJ-comparison} to show that Ansch\"{u}tz--Le Bras's prismatic Dieudonn\'e functor $\mc{M}_\smallprism$ recovers the Breuil--Kisin--Kim Dieudonn\'e functor $\mf{M}$.

\subsubsection{The Breuil Dieudonn\'e functor}\label{ss:Breuil-DD-functor} We begin by first recalling an intermediate construction between $\bb{D}_\mr{crys}$ and $\mf{M}$ constructed by Breuil.

\begin{defn}[{cf.\@ \cite[Definition 3.12]{KimBK}}]\label{defn:Kisin-modules} A \emph{\emph{(}minuscule\emph{)} Breuil-module} over $R$ is a quadruple of data $(\mc{M},\Fil^1_\mc{M},\varphi_\mc{M},\nabla^0_\mc{M})$ with: 
\begin{itemize}[leftmargin=.3in]
    \item $\mc{M}$ a finite projective $S_R$-module,
    \item $\Fil^1_\mc{M}\subseteq\mc{M}$ an $S_R$-submodule with $\Fil^1_\mr{PD}\cdot \mc{M}\subseteq \Fil^1_\mc{M}$ and $\mc{M}/\Fil^1_\mc{M}$ projective over $R$,
    \item $\varphi_\mc{M}\colon \phi^\ast\mc{M}\to \mc{M}$ is an $S_R$-linear map with $\varphi_M(\phi^\ast\Fil^1_\mc{M})=p\mc{M}$, 
    \item $\nabla^0_\mc{M}$ is a topologically quasi-nilpotent integrable connection (cf.\@ \cite[Remark 2.2.4]{deJongCrystalline}) on $\mc{M}_0\defeq \mc{M}\otimes_{S_R}R$, where $S_R\to R$ is induced by the map $\mf{S}_R\to R$ sending $u$ to $0$, such that the induced Frobenius $\varphi_{\mc{M}_0}$ is horizontal.
\end{itemize}
These naturally form a category which we denote by $\cat{VectF}^\varphi_{[0,1]}(S_R,\nabla^0)$.
\end{defn}

By \cite[Proposition 3.8 and Lemma 3.15]{KimBK}, 
there is a fully faithful embedding
\begin{equation}\label{eq:DieuF-Breuil-S-module-embedding}
    \cat{DieuF}(R)\to \cat{VectF}^\varphi_{[0,1]}(S_R,\nabla^0),\quad (\mc{F},\varphi_\mc{F},\Fil^\bullet_\mc{F})\mapsto ((\mc{F}, \Fil^1_\mc{F},\varphi_\mc{F})(S_R\twoheadrightarrow R), \nabla_{\mc{F}}).
\end{equation}
This then suggests the following definition.

\begin{defn}[{cf.\@ \cite[\S3.4]{KimBK}}] The functor
\begin{equation*}
    \mc{M}^\mr{Br}\colon \cat{BT}_p(R)\to\cat{VectF}^\varphi_{[0,1]}(S_R,\nabla^0),\qquad H\mapsto \mc{M}^\mr{Br}(H)\defeq(\bb{D}(H)(S_R\twoheadrightarrow R),\nabla_{\underline{\bb{D}}(H)}),
\end{equation*}
is called the \emph{Breuil Dieudonn\'e functor}.
\end{defn}

As \eqref{eq:DieuF-Breuil-S-module-embedding} and $\bb{D}$ are both fully faithful, it follows that $\mc{M}^\mr{Br}$ is also fully faithful.

\subsubsection{The Breuil--Kisin--Kim Dieudonn\'e functor} We now recall the definition of the Breuil--Kisin--Kim Dieudonn\'e functor.

\begin{defn}[{cf.\@ \cite[Definition 6.1]{KimBK}}] A \emph{\emph{(}minuscule\emph{)} Kisin module} over $R$ is a triple $(\mf{M},\varphi_\mf{M},\nabla^0_\mf{M})$ where:
\begin{itemize}[leftmargin=.3in]
\item $(\mf{M},\varphi_\mf{M})$ is an object of the category $\cat{Vect}^\varphi_{[0,1]}(\mf{S}_R,(E))$,
\item $\nabla^0_\mf{M}$ is a topologically quasi-nilpotent integrable connection on $\mf{M}_0\defeq \phi^\ast\mf{M}/u$.
\end{itemize}
We denote the category of such objects by $\cat{Vect}_{[0,1]}^\varphi(\mf{S}_R,\nabla^0)$.
\end{defn}

There is a functor
\begin{equation}\label{eq:BK-to-Breuil-functor}
\cat{Vect}^\varphi_{[0,1]}(\mf{S}_R,\nabla^0)\to \cat{VectF}^\varphi_{[0,1]}(S_R,\nabla^0), 
\end{equation}
where $(\mf{M},\varphi_\mf{M},\nabla^0_{\mf{M}_0})$ is sent to the object whose underlying filtered $S_R$-module is 
\begin{equation*}
    (\phi^\ast\mf{M},\Fil^\bullet_\mr{Nyg})\otimes_{(\mf{S}_R,\Fil^\bullet_E)}(S_R,\Fil^\bullet_\mr{PD})
\end{equation*}
(forgetting everything but the $1$-part of the filtration), with Frobenius given by $\varphi_{\phi^\ast\mf{M}}\otimes 1$, and with connection $\nabla^0_\mf{M}$, which is sensible via the natural identification of $R$-modules 
\begin{equation*}\phi^\ast\mf{M}\otimes_{\mf{S}_R}R\simeq \phi^\ast\mf{M}\otimes_{\mf{S}_R}S_R\otimes_{S_R}R.
\end{equation*}
This functor is fully faithful by \cite[Lemma 6.5]{KimBK}.

\begin{defn}[{\cite{KimBK}}] By \cite[Corollary 6.7]{KimBK}, the functor $\mc{M}^\mr{Br}$ factorizes through the fully faithful embedding \eqref{eq:BK-to-Breuil-functor}. Thus, we obtain a functor
\begin{equation*}
    \mf{M}\colon \cat{BT}_p(R)\to \cat{Vect}_{[0,1]}^\varphi(\mf{S}_R,\nabla^0),\quad H\mapsto \mf{M}(H)=(\underline{\mf{M}}(H),\varphi_{\underline{\mf{M}}(H)},\nabla^0_{\bb{D}(H)})
\end{equation*}
called the \emph{Breuil--Kisin--Kim Dieudonn\'e functor}.
\end{defn}

\begin{rem} When $R=\mc{O}_K$, this agrees with the functor from \cite{KisinFCrystal}. This is not explicitly explained in op.\@ cit.\@, and can be argued directly, but also follows by combining Proposition \ref{prop:Kim-prismatic-comp} and Remark \ref{rem:AB-Kim} below.
\end{rem}

To relate this to prismatic Dieudonn\'e modules, let us begin by observing that there are functors
\begin{equation*}
\begin{aligned} \mr{ev}_{\mf{S}_R}&\colon \cat{Vect}^\varphi_{[0,1]}(R_\smallprism)\to \cat{Vect}^\varphi_{[0,1]}(\mf{S}_R,(E)),\qquad (\mc{E},\varphi_\mc{E})\mapsto (\mc{E},\varphi_\mc{E})(\mf{S}_R,(E)),\\
\mr{ev}_{\mf{S}_R}^\mr{K} &\colon \cat{Vect}^\varphi_{[0,1]}(R_\smallprism)\to\cat{Vect}^\varphi_{[0,1]}(\mf{S}_R,\nabla^0),\qquad \,\,(\mc{E},\varphi_\mc{E})\mapsto (\mr{ev}_{\mf{S}_R}(\mc{E},\varphi_\mc{E}),\nabla_{\underline{\bb{D}}_\crys(\mc{E}.\varphi_\mc{E})}).
\end{aligned}
\end{equation*}
Here by $\nabla_{\underline{\bb{D}}_\crys(\mc{E},\varphi_\mc{E})}$ we abusively mean the pullback of $\nabla_{\underline{\bb{D}}_\crys(\mc{E},\varphi_\mc{E})}$ under the isomorphism 
\begin{equation*}
    \phi^\ast\mc{E}(\mf{S}_R,(E))/u\isomto \underline{\bb{D}}_\crys(\mc{E},\varphi_\mc{E})(R),
\end{equation*}
from Proposition \ref{prop:crystalline-realization-Kisin-comparison}.

\begin{prop}\label{prop:Kim-prismatic-comp} There is a natural identification $\mr{ev}_{\mf{S}_R}^\mr{K}\circ \mc{M}_\smallprism\isomto \mf{M}$.
\end{prop}
\begin{proof} Let $H$ be an object of $\cat{BT}_p(R)$. Then, by the definition of $\mf{M}$ it suffices to show that $\mr{ev}^\mr{K}_{\mf{S}_R}(\mc{M}_\smallprism(H))$ and $\mf{M}(H)$ have images under \eqref{eq:BK-to-Breuil-functor} which are canonically identified. But, this follows by combining Remark \ref{rem:filtered-tensor-product-computation}, Theorem \ref{thm:ALB-dJ-comparison}, and the definition of $\mc{M}^\mr{Br}(H)$.
\end{proof}

\begin{rem}\label{rem:AB-Kim} For $R=\mc{O}_K$ this was previously shown in \cite[Proposition 5.18]{AnschutzLeBrasDD}.
\end{rem}

Combining Theorem \ref{thm:ALB-equiv}, \cite[Corollary 6.7 and Corollary 10.4]{KimBK}, Proposition \ref{prop:Kim-prismatic-comp}, \cite[Proposition 1.28]{IKY2}, and Proposition \ref{prop:lff-FL-range}, we deduce the following.

\begin{cor}\label{cor:F-gauges-BK-modules} Let $R$ be a formally framed $W$-algebra. Then, there are equivalences of categories
\begin{equation*}
\begin{tikzcd}
	{\cat{Vect}_{[0,1]}(R^\mr{syn})} & {\cat{Vect}_{[0,1]}^\varphi(R_\smallprism)} & {\cat{Vect}^{\varphi}_{[0,1]}(\mf{S}_R,\nabla^0).}
	\arrow["{\mr{R}_\mf{X}}", from=1-1, to=1-2]
	\arrow["\sim"', from=1-1, to=1-2]
	\arrow["\sim"', from=1-2, to=1-3]
	\arrow["{\mr{ev}^{\mr{K}}_{\mf{S}_R}}", from=1-2, to=1-3]
\end{tikzcd}
\end{equation*}
If $R=W\ll t_1,\ldots,t_d\rr$ for some $d\geqslant 0$, then there are equivalences
\begin{equation*}
\begin{tikzcd}
	{\cat{Vect}_{[0,1]}(R^\mr{syn})} & {\cat{Vect}_{[0,1]}^\varphi(R_\smallprism)} & {\cat{Vect}^{\varphi}_{[0,1]}(\mf{S}_R,(E)).}
	\arrow["{\mr{R}_\mf{X}}", from=1-1, to=1-2]
	\arrow["\sim"', from=1-1, to=1-2]
	\arrow["\sim"', from=1-2, to=1-3]
	\arrow["{\mr{ev}_{\mf{S}_R}}", from=1-2, to=1-3]
\end{tikzcd}
\end{equation*}
\end{cor}

\begin{rem}The second claim of Corollary \ref{cor:F-gauges-BK-modules} was previously established by Ansch\"{u}tz--Le Bras (see \cite[Theorem 5.12]{AnschutzLeBrasDD}) and Ito (see \cite[Proposition 7.1.1]{Ito1}).
\end{rem}


\begin{thebibliography}{DLMS24}
	\providecommand{\url}[1]{\texttt{#1}}
	\providecommand{\urlprefix}{URL }
	\providecommand{\eprint}[2][]{\url{#2}}
	
	\bibitem[ALB23]{AnschutzLeBrasDD}
	J.~Ansch\"{u}tz and A.-C. Le~Bras, Prismatic {D}ieudonn\'{e} {T}heory, Forum
	Math. Pi 11 (2023), Paper No. e2.
	
	\bibitem[AY25]{AbhinandanYoucis}
	Abhinandan and A.~Youcis, Twisted crystalline and de Rham cohomology, 2025, in
	preparation.
	
	\bibitem[BBM82]{BBMDieuII}
	P.~Berthelot, L.~Breen and W.~Messing, Th\'{e}orie de {D}ieudonn\'{e}
	cristalline. {II}, vol. 930 of Lecture Notes in Mathematics, Springer-Verlag,
	Berlin, 1982.
	
	\bibitem[Bha23]{BhattNotes}
	B.~Bhatt, Prismatic $F$-gauges, 2023, unpublished course notes
	\href{https://www.math.ias.edu/~bhatt/teaching/mat549f22/lectures.pdf}.
	
	\bibitem[BIM19]{BIM}
	B.~Bhatt, S.~B. Iyengar and L.~Ma, Regular rings and perfect(oid) algebras,
	Comm. Algebra 47 (2019), no.~6, 2367--2383.
	
	\bibitem[BL22a]{BhattLurieAbsolute}
	B.~Bhatt and J.~Lurie, Absolute prismatic cohomology, 2022,
	\eprint{2201.06120}.
	
	\bibitem[BL22b]{BhattLuriePrismatization}
	B.~Bhatt and J.~Lurie, The prismatization of $p$-adic formal schemes, 2022,
	\eprint{2201.06124}.
	
	\bibitem[BMS18]{BMSI}
	B.~Bhatt, M.~Morrow and P.~Scholze, Integral {$p$}-adic {H}odge theory, Publ.
	Math. Inst. Hautes \'{E}tudes Sci. 128 (2018), 219--397.
	
	\bibitem[BO83]{BerthelotOgusFIsocrystals}
	P.~Berthelot and A.~Ogus, {$F$}-isocrystals and de {R}ham cohomology. {I},
	Invent. Math. 72 (1983), no.~2, 159--199.
	
	\bibitem[BS22]{BhattScholzePrisms}
	B.~Bhatt and P.~Scholze, Prisms and prismatic cohomology, Ann. of Math. (2) 196
	(2022), no.~3, 1135--1275.
	
	\bibitem[BS23]{BhattScholzeCrystals}
	B.~Bhatt and P.~Scholze, Prismatic {$F$}-crystals and crystalline {G}alois
	representations, Camb. J. Math. 11 (2023), no.~2, 507--562.
	
	\bibitem[dJ95]{deJongCrystalline}
	A.~J. de~Jong, Crystalline {D}ieudonn\'{e} module theory via formal and rigid
	geometry, Inst. Hautes \'{E}tudes Sci. Publ. Math.  (1995), no.~82, 5--96
	(1996).
	
	\bibitem[DLMS24]{DLMS}
	H.~Du, T.~Liu, Y.~S. Moon and K.~Shimizu, Completed prismatic {$F$}-crystals
	and crystalline {$\mathbb{Z}_p$}-local systems, Compos. Math. 160 (2024),
	no.~5, 1101--1166.
	
	\bibitem[Dri21]{DrinfeldRingGroupoids}
	V.~Drinfeld, On a notion of ring groupoid, 2021, \eprint{2104.07090}.
	
	\bibitem[Dri24]{Drinfeld}
	V.~Drinfeld, Prismatization, Selecta Math. (N.S.) 30 (2024), no.~3, Paper No.
	49, 150.
	
	\bibitem[Fal89]{Faltings89}
	G.~Faltings, Crystalline cohomology and {$p$}-adic {G}alois-representations, in
	Algebraic analysis, geometry, and number theory ({B}altimore, {MD}, 1988),
	pp. 25--80, Johns Hopkins Univ. Press, Baltimore, MD, 1989.
	
	\bibitem[FL82]{FL82}
	J.-M. Fontaine and G.~Laffaille, Construction de repr\'{e}sentations
	{$p$}-adiques, Ann. Sci. \'{E}cole Norm. Sup. (4) 15 (1982), no.~4, 547--608.
	
	\bibitem[Gao19]{GaoFLstr}
	H.~Gao, Fontaine-{L}affaille modules and strongly divisible modules, Ann. Math.
	Qu\'e. 43 (2019), no.~1, 145--159.
	
	\bibitem[GL23]{GuoLi}
	H.~Guo and S.~Li, Frobenius height of prismatic cohomology with coefficients,
	2023, \eprint{2309.06663(v2)}.
	
	\bibitem[GM26]{GMM}
	Z.~Gardner and K.~Madapusi, An algebraicity conjecture of Drinfeld and the
	moduli of $p$-divisible groups, 2026, \eprint{2412.10226(v4)}.
	
	\bibitem[GR24]{GuoReinecke}
	H.~Guo and E.~Reinecke, A prismatic approach to crystalline local systems,
	Invent. Math. 236 (2024), no.~1, 17--164.
	
	\bibitem[Hau24]{Hauck}
	M.~Hauck, A stacky approach to $p$-adic Hodge theory, 2024,
	\eprint{2409.10557}.
	
	\bibitem[Hok24]{Hokaj}
	C.~Hokaj, Fontaine-Laffaille Theory over Power Series Rings, 2024,
	\eprint{2407.21327}.
	
	\bibitem[Hub96]{HuberEC}
	R.~Huber, \'{E}tale cohomology of rigid analytic varieties and adic spaces,
	Aspects of Mathematics, E30, Friedr. Vieweg \& Sohn, Braunschweig, 1996.
	
	\bibitem[IKY23]{IKY2}
	N.~Imai, H.~Kato and A.~Youcis, The Prismatic Realization Functor for Shimura
	Varieties of Abelian Type, 2023, \eprint{2310.08472v5}.
	
	\bibitem[IKY26]{IKY1}
	N.~Imai, H.~Kato and A.~Youcis, A Tannakian framework for prismatic
	$F$-crystals, Forum Math. Sigma 14 (2026), Paper No. e100, 1--64.
	
	\bibitem[Ito25]{Ito1}
	K.~Ito, Prismatic {$G$}-displays and descent theory, Algebra Number Theory 19
	(2025), no.~9, 1685--1770.
	
	\bibitem[Ked19]{KedlayaAWS}
	K.~S. Kedlaya, Sheaves, Shtukas, and Stacks, in Perfectoid Spaces, vol. 242 of
	Mathematical Surveys and Monographs, pp. 45--191, American Mathematical
	Society, Providence, RI, 2019.
	
	\bibitem[Kie67]{Kiehl}
	R.~Kiehl, Theorem {A} und {T}heorem {B} in der nichtarchimedischen
	{F}unktionentheorie, Invent. Math. 2 (1967), 256--273.
	
	\bibitem[Kim15]{KimBK}
	W.~Kim, The relative {B}reuil-{K}isin classification of {$p$}-divisible groups
	and finite flat group schemes, Int. Math. Res. Not. IMRN  (2015), no.~17,
	8152--8232.
	
	\bibitem[Kis06]{KisinFCrystal}
	M.~Kisin, Crystalline representations and {$F$}-crystals, in Algebraic geometry
	and number theory, vol. 253 of Progr. Math., pp. 459--496, Birkh\"{a}user
	Boston, Boston, MA, 2006.
	
	\bibitem[Kis10]{KisIntShab}
	M.~Kisin, Integral models for {S}himura varieties of abelian type, J. Amer.
	Math. Soc. 23 (2010), no.~4, 967--1012.
	
	\bibitem[Lau18]{LauDieu}
	E.~Lau, Dieudonn\'{e} theory over semiperfect rings and perfectoid rings,
	Compos. Math. 154 (2018), no.~9, 1974--2004.
	
	\bibitem[LL25]{LiLiuDerived}
	S.~Li and T.~Liu, Comparison of prismatic cohomology and derived de {R}ham
	cohomology, J. Eur. Math. Soc. (JEMS) 27 (2025), no.~1, 183--268.
	
	\bibitem[LMP24]{LiuMoonPatel}
	T.~Liu, Y.~S. Moon and D.~Patel, Relative {F}ontaine-{M}essing theory over
	power series rings, Int. Math. Res. Not. IMRN  (2024), no.~5, 4384--4455.
	
	\bibitem[Lov17]{LoveringFCrystals}
	T.~Lovering, Filtered F-crystals on Shimura varieties of abelian type, 2017,
	\eprint{1702.06611}.
	
	\bibitem[LvO96]{LvO}
	H.~Li and F.~van Oystaeyen, Zariskian filtrations, vol.~2 of $K$-Monographs in
	Mathematics, Kluwer Academic Publishers, Dordrecht, 1996.
	
	\bibitem[Maz72]{Mazur1972}
	B.~Mazur, Frobenius and the Hodge filtration, Bull. Amer. Math. Soc. 78 (1972),
	no.~5, 653--667.
	
	\bibitem[Maz73]{mazur1973frobenius}
	B.~Mazur, Frobenius and the {H}odge filtration (estimates), Annals of
	Mathematics 98 (1973), no.~1, 58--95.
	
	\bibitem[SP]{StacksProject}
	{Stacks Project Authors}, \textit{Stacks Project},
	\url{http://stacks.math.columbia.edu}, 2023.
	
	\bibitem[SW20]{ScholzeBerkeley}
	P.~Scholze and J.~Weinstein, Berkeley lectures on {$p$}-adic geometry, vol. 207
	of Annals of Mathematics Studies, Princeton University Press, Princeton, NJ,
	2020.
	
	\bibitem[Tsu20]{Tsu20}
	T.~Tsuji, Crystalline $\mathbb{Z}_p$-Representations and
	$A_\mathrm{inf}$-Representations with {F}robenius, in p-adic Hodge Theory,
	pp. 161--319, Springer, 2020.
	
	\bibitem[TT19]{TanTong}
	F.~Tan and J.~Tong, Crystalline comparison isomorphisms in {$p$}-adic {H}odge
	theory: the absolutely unramified case, Algebra Number Theory 13 (2019),
	no.~7, 1509--1581.
	
	\bibitem[TVX26]{TVX}
	G.~Terentiuk, V.~Vologodsky and Y.~Xu, Prismatic {$F$}-gauges and
	Fontaine--Laffaille modules, Compositio Mathematica 162 (2026), 277--326.
	
	\bibitem[W{\"{u}}r23]{Wur23}
	M.~W{\"{u}}rthen, Divided prismatic Frobenius crystals of small height and the
	category $\mathcal{M}\mathcal{F}$, 2023, \eprint{2305.06081}.
	
\end{thebibliography}
\end{document}